\documentclass[11pt]{article}

\usepackage{amsmath,amssymb,amsthm}
\usepackage{tikz}
\usetikzlibrary{arrows.meta}
\usepackage[margin=1in]{geometry}
\newtheorem{theorem}{Theorem}
\newtheorem{lemma}{Lemma}
\newtheorem{proposition}{Proposition}
\newtheorem{remark}{Remark}
\theoremstyle{definition}
\newtheorem{definition}{Definition}
\theoremstyle{plain}

\newcommand{\HH}{\mathbb H}
\newcommand{\SSph}{\mathbb S}
\newcommand{\Om}{\Omega}
\newcommand{\Lip}{\operatorname{Lip}}
\newcommand{\diam}{\operatorname{diam}}
\newcommand{\eps}{\varepsilon}

\title{A Large-Diameter Fundamental-Gap Lower Bound for Horoconvex Domains}
\author{\begin{tabular}{c}
		Xianzhe Dai\\
		University of California, Santa Barbara\\
		\texttt{dai@math.ucsb.edu}\\[0.55em]
		John M. Ennis\\
		Aigora, Richmond, VA\\
		\texttt{john.m.ennis@aigora.com}\\[0.55em]
		Xuan Hien Nguyen\\
		Iowa State University\\
		\texttt{xhnguyen@iastate.edu}\\[0.55em]
		Guofang Wei\\
		University of California, Santa Barbara\\
		\texttt{wei@math.ucsb.edu}
\end{tabular}}
\date{June 10, 2026}

\begin{document}
	\maketitle
	
	\begin{abstract}
		We prove a large-diameter fundamental-gap lower bound for compact horoconvex
		domains in real hyperbolic space of curvature \(-1\).  The geometric part
		reduces large horoconvex domains to a fixed-width radial-height problem in all
		dimensions.  The analytic part proves the needed radial-height theorem by
		comparing the low-energy Dirichlet form with a limiting angular operator on the
		sphere, while the radial complement is separated by a one-dimensional branch
		gap and endpoint Green estimates.  The result gives the polynomial \(D^{-3}\)
		scale matching the Nguyen--Stancu--Wei large-diameter upper bound.
	\end{abstract}
	
	\paragraph*{Note added (June 2026).}
	The present project was publicly announced by Ennis on May 30, 2026; a
	May 31, 2026 follow-up post stated that the \(D^{-3}\) lower-bound theorem had
	been proved and displayed the manuscript abstract.  While the exposition of
	this manuscript was being revised, Park posted a contemporaneous preprint on
	June 9, 2026 \cite{ParkHoroconvex}, deriving the same \(D^{-3}\) scale by a
	different asymptotic route; Park's preprint also notes the public
	announcement.  We cite Park here to record the contemporaneous work and
	clarify the public chronology.
	
	\section{Introduction}
	
	For a bounded domain \(\Omega\), the fundamental gap
	\(\lambda_2(\Omega)-\lambda_1(\Omega)\) measures the separation between the
	ground state and the first excited Dirichlet state.  In Euclidean space,
	Andrews--Clutterbuck proved the sharp fundamental-gap conjecture for convex
	domains \cite{AndrewsClutterbuck}, building on the log-concavity tradition
	initiated by Brascamp--Lieb \cite{BrascampLieb}.  On the sphere, sharp
	positive-curvature analogues and refinements were obtained by
	Seto--Wang--Wei \cite{SetoWangWei}, He--Wei--Zhang \cite{HeWei}, and
	Dai--Seto--Wei \cite{DaiSetoWei}; see also the probabilistic approach of
	Cho--Wei--Yang \cite{ChoWeiYang} and the surface modulus-of-concavity work of
	Khan--Tuerkoen--Wei \cite{KhanTuerkoenWei}.  These results show how strongly
	the sign of curvature interacts with convexity and first-eigenfunction
	concavity.
	
	Negative curvature changes the picture.  Ordinary geodesic convexity is not
	strong enough to force any diameter-only lower bound:
	Bourni--Clutterbuck--Nguyen--Stancu--Wei--Wheeler showed that convex domains
	in hyperbolic space can have arbitrarily small fundamental gap \cite{BCNSWW}.
	Khan--Nguyen extended this negative-curvature failure mechanism beyond the
	constant-curvature hyperbolic model \cite{KhanNguyen}.  Related recent work
	of Clutterbuck--J\"ackel--Nguyen and J\"ackel shows that constant potentials
	do not minimize the fundamental gap on convex domains in negative curvature
	\cite{ClutterbuckJackelNguyen,JackelHadamard}.
	Nguyen--Stancu--Wei proposed horoconvexity as the natural stronger replacement in
	\(\mathbb H^n\), studied geodesic balls and large horoconvex domains, and
	proved the upper bound \(\lambda_2-\lambda_1\le C(n)D^{-3}\) for large
	horoconvex domains \cite[Theorem 1.1]{NSW}. 
	In a recent paper framed in the broader setting of conformally flat
	manifolds, Khan--Tuerkoen give a qualitative positive lower bound for compact
	horoconvex domains using conformal log-concavity estimates
	\cite{KhanTuerkoen}, building on the conformal-concavity framework of
	Khan--Saha--Tuerkoen
	\cite{KhanSahaTuerkoenMathZ,KhanSahaTuerkoenAGAG}.  Their explicit
	large-diameter bound is much smaller than the expected \(D^{-3}\) scale; in
	particular, it decays doubly exponentially in the diameter, and they raise the
	true diameter rate as an open question.
	
	 Using Krist\'aly's large-radius expansion and the 
	identification of the first nonradial branch
	\cite{Kristaly,NSW}, for geodesic balls we prove
	\[
	\lambda_2(B_R^{\HH^n})-\lambda_1(B_R^{\HH^n})
	\sim {4\pi^2\over (n-1)R^3}
	\qquad (R\to\infty).
	\]
	
	More generally, for horoconvex domains, Theorem~\ref{thm:main} gives the polynomial large-diameter lower bound
	matching the Nguyen--Stancu--Wei upper-bound power.  The proof uses radius
	control and a
	 radial-height spectral reduction.  The
	Borisenko--Miquel radius estimate first puts every large horoconvex domain in a
	fixed-width annulus about a common center.  The horospherical support
	representation then converts the boundary to a uniformly bounded Lipschitz
	radial height.  After smoothing and a varying-domain squeeze, the proof reduces
	to a radial-height spectral theorem in every dimension.  The
	limiting first-radial-branch form is a nonlocal angular operator on
	\(\mathbb S^{n-1}\).  A compactness argument, using positivity improvement of
	the subordinated Poisson semigroup, supplies a uniform angular gap.
	For comparison, we also record the simpler direct inball/circumball proof that
	already gives the theorem in dimensions \(n=2,3\).
	
	In fact, Theorem 2 establishes the same large-diameter lower bound for a class of domains going beyond the horoconvex category.
	
	Large-diameter degeneration also appears in work on stable large-scale
	structure under Ricci lower bounds, including universal-cover stability
	\cite{SormaniWeiCompact,SormaniWeiNoncompact,EnnisWeiCompact,EnnisWeiNoncompact},
	fundamental-group structure \cite{KapovitchWilking}, and semi-local simple
	connectedness \cite{PanWei}.  That topological literature is not used below,
	but it motivates the same organizing habit: identify a stable auxiliary object
	before estimating a degenerating sequence.  Here the stable object is the
	fixed-width radial-height class with its limiting angular operator.
	
	This leaves the broader hyperbolic log-concavity question open.  Related
	positive-curvature surface results were obtained by
	Khan--Nguyen--Tuerkoen--Wei \cite{KhanNguyenTuerkoenWei}.  Wei--Xiao
	recently proved super log-concavity of the first eigenfunction for
	horoconvex domains of sufficiently small diameter and showed the sharpness of
	that diameter restriction \cite{WeiXiao}.  The result below addresses the
	opposite large-diameter spectral regime; it does not prove even the weaker
	ordinary log-concavity inequality \(\operatorname{Hess}\log\psi_1\le0\) for
	general horoconvex domains.
	
	\section{Main Result and Proof Strategy}
	\label{sec:main-result}
	
	All hyperbolic spaces below have constant sectional curvature \(-1\).
	For compact domains, Dirichlet eigenvalues are taken on the interior.
	
	\begin{theorem}[Large-diameter horoconvex gap]
		\label{thm:main}
		For every \(n\geq 2\) there exist constants
		\[
		c_n>0,\qquad D_n<\infty
		\]
		such that every compact horoconvex domain
		\(\Om\subset \HH^n\) with diameter \(D\geq D_n\) satisfies
		\[
		\lambda_2(\Om)-\lambda_1(\Om)\geq c_n D^{-3}.
		\]
	\end{theorem}
	
	\begin{remark}
		The theorem is only asserted in the large-diameter regime.  No all-scale
		\(cD^{-3}\) statement is intended; small hyperbolic domains have the Euclidean
		\(D^{-2}\) scaling.
		The constant \(c_n\) produced here is not effective; the proof identifies the
		sharp large-diameter power rather than an optimal coefficient.  The compactness
		proof of Lemma~\ref{lem:angular-effective-gap} gives no quantitative lower
		bound for \(\delta_{\rm eff}(n,C)\), so this argument does not by itself
		produce an effective \(c_n\).
	\end{remark}
	
	\begin{remark}
		Khan--Tuerkoen
		\cite{KhanTuerkoen} gives a positive lower bound for compact horoconvex
		domains depending on dimension and diameter, in the sense of the qualitative
		Nguyen--Stancu--Wei question.  Their explicit large-diameter bound decays
		doubly exponentially in the diameter, and they ask for the optimal diameter
		rate.  The point of Theorem~\ref{thm:main} is the sharper polynomial
		large-diameter scale \(D^{-3}\), matching the power in the
		Nguyen--Stancu--Wei upper bound \cite[Theorem 1.1]{NSW} and in the ball
		asymptotic.
	\end{remark}
	
	In fact, the proof establishes the following more general radial-height
	estimate.  Theorem~\ref{thm:main} follows from it after the horoconvex
	geometric reduction and the support-to-radial bridge.
	
	Fix a base point $o\in \mathbb H^n$, let $(\rho, \theta)\in R^+\times S^{n-1}$ denote the geodesic polar coordinates centered at $o$.
	
	\begin{theorem}[Lipschitz radial-height theorem]
		\label{thm:sh-lip}
		Fix \(n\geq 2\), \(C<\infty\), and \(L<\infty\).  There exist constants
		\[
		c(n,C)>0,\qquad R_0(n,C,L)<\infty
		\]
		such that for every Lipschitz radial-height domain
		\[
		\Om_R(h)=\{0\leq \rho<R-h(\theta)\},\qquad
		0\leq h\leq C,\qquad \Lip(h)\leq L,
		\]
		and every \(R\geq R_0(n,C,L)\),
		\[
		\lambda_2(\Om_R(h))-\lambda_1(\Om_R(h))\geq c(n,C)R^{-3}.
		\]
		The constant \(c(n,C)\) is independent of \(L\); the Lipschitz constant
		affects only the threshold \(R_0\).
	\end{theorem}

Note that $B_{R-C}(o) \subset \Om_R(h) \subset B_{R}(o)$. Most importantly, $\Om_R(h)$ need not be horoconvex, requiring only the boundedness and Lipschitz condition for $h$. It is also worth pointing out that the proof actually establishes the following. For any $\epsilon>0$, there exists $R_0(n,C,L, \epsilon)<\infty$ such that 
\[
\lambda_2(\Om_R(h))-\lambda_1(\Om_R(h))\geq (1-\epsilon) (\mu_2(L_h)- \mu_1(L_h)) R^{-3}, 
\]
for all \(R\geq R_0(n,C,L, \epsilon)\). Here $L_h$ is the effective angular operator introduced in Section 6.2.
	\subsection{Outline of the Proof}
	
	The proof of Theorem~\ref{thm:main} uses the same radial-height route in every
	dimension.  First, the Borisenko--Miquel inradius/circumradius estimate, in the
	form used by Nguyen--Stancu--Wei, gives \(R-r\le\log 2\).  An elementary
	triangle argument then shows that every sufficiently large compact horoconvex
	domain lies between two concentric balls whose radii differ by \(2\log 2\).
	The horospherical support representation gives a radial height \(h\) over the
	circumcenter sphere with \(0\le h\le2\log 2\) and a Lipschitz bound depending
	only on this width.
	
	Theorem~\ref{thm:sh-lip} proves the uniform radial-height estimate for this
	Lipschitz class in every dimension \(n\ge2\).  Its proof first treats smooth
	heights and then uses spherical mollification together with the offset-domain
	squeeze, using individual eigenvalue convergence rather than any monotonicity
	of the fundamental gap.  The first radial branch is lifted degree-by-degree,
	which is why the limiting angular operator is the diagonal multiplier
	\(T_nY_l=(b_{n+2l}-b_n)Y_l\), not the ordinary spherical Laplacian.  The
	constant \(c(n,C)\) is independent of the Lipschitz constant \(L\), while
	\(L\) only enlarges the large-\(R\) threshold.  Thus the Lipschitz bound
	produced by the horospherical support step affects only how far out one must
	go before the estimate applies, not the gap constant.  The constants and the
	analytic proof of the radial-height theorem are recorded in
	Section~\ref{sec:radial-height-proof}, before the final discussion.
	
	We also record an elementary shortcut in dimensions \(n=2,3\):
	Proposition~\ref{prop:low-dim-ball-comparison} uses the sharper nonconcentric
	inball/circumball comparison and domain monotonicity of indexed Dirichlet
	eigenvalues to give the \(R^{-3}\) lower bound directly.  This proposition is
	not needed for the all-dimensional final assembly, but it is a useful check on
	the geometry and the large-ball coefficient arithmetic.
	
	\section{Geometric Reduction and a Low-Dimensional Shortcut}
	
	Let \(R\) be the circumradius of \(\Om\), let \(o\) be a circumcenter, let
	\(r\) be the inradius, and let \(a\) be the center of an inball.  Thus
	\[
	B_r(a)\subset \Om\subset B_R(o).
	\]
	The inclusion \(B_r(a)\subset B_R(o)\) forces \(d(o,a)+r\leq R\), hence
	\(d(o,a)\leq R-r\), by taking the point of \(B_r(a)\) on the geodesic ray
	from \(o\) through \(a\) farthest from \(o\).
	
	The Borisenko--Miquel inradius/circumradius estimate
	\cite[Theorem 1]{BorisenkoMiquel}, as quoted and used by
	Nguyen--Stancu--Wei \cite[Theorem 4.2 and formula (15)]{NSW}, gives
	\[
	R-r\leq\log 2
	\]
	for compact horoconvex domains in curvature \(-1\).  More precisely,
	\cite[formula (15)]{NSW} gives
	\[
	R-r
	\le
	\log\frac{(1+\sqrt{\tau})^2}{1+\tau}
	<\log 2,\qquad \tau=\tanh(r/2),
	\]
	and we use only the weaker universal bound quoted in
	\cite[formula (15)]{NSW}.  Nguyen--Stancu--Wei work with compact horoconvex
	domains in the supporting-horoball sense of their Definition 4.1.  If one
	instead starts from a smooth principal-curvature formulation and approximates,
	the radius inequality itself passes under Hausdorff limits because inradius
	and circumradius are Hausdorff-continuous; this observation does not require
	asserting a separate preservation theorem for the differential formulation of
	horoconvexity.
	For the rest of the reduction we impose the large-diameter threshold
	\[
	R>2\log 2 .
	\]
	Since \(R-r\le\log 2\), we have
	\[
	2r-R=R-2(R-r)\ge R-2\log 2.
	\]
	The threshold makes \(2r-R>0\), so the inner radius
	\[
	r-d(o,a)\ge2r-R
	\]
	is positive.
	For any \(y\) with \(d(y,o)<r-d(o,a)\), the triangle inequality gives
	\[
	d(y,a)\leq d(y,o)+d(o,a)<r,
	\]
	and hence
	\[
	B_{r-d(o,a)}(o)\subset B_r(a).
	\]
	Since \(r-d(o,a)\geq2r-R\), we obtain
	\[
	B_{2r-R}(o)\subset \Om\subset B_R(o).
	\]
	Therefore
	\[
	B_{R-2\log 2}(o)\subset \Om\subset B_R(o).
	\]
	This fixed-width common-center annulus is the only place where the universal
	width \(2\log 2\) enters.
	
	\begin{proposition}[Low-dimensional ball comparison]
		\label{prop:low-dim-ball-comparison}
		Let \(n=2\) or \(n=3\).  There are constants
		\[
		c_n^{\rm ball}>0,\qquad R_n^{\rm ball}<\infty
		\]
		such that every compact horoconvex domain \(\Om\subset\HH^n\) with
		circumradius \(R\ge R_n^{\rm ball}\) satisfies
		\[
		\lambda_2(\Om)-\lambda_1(\Om)\ge c_n^{\rm ball}R^{-3}.
		\]
	\end{proposition}
	
	\begin{proof}
		Let \(R\) be the circumradius, let \(r\) be the inradius, and choose centers
		\(o,a\) so that
		\[
		B_r(a)\subset \Om\subset B_R(o).
		\]
		Domain monotonicity of indexed Dirichlet eigenvalues gives
		\[
		\lambda_2(\Om)\ge \lambda_2(B_R(o)),\qquad
		\lambda_1(\Om)\le \lambda_1(B_r(a)).
		\]
		Therefore
		\[
		\lambda_2(\Om)-\lambda_1(\Om)
		\ge
		\lambda_2(B_R^{\HH^n})-\lambda_1(B_r^{\HH^n}).
		\]
		Set \(\alpha_R=R-r\).  The inradius/circumradius estimate gives
		\(0\le\alpha_R\le\log 2\).  Krist\'aly's large-radius first-eigenvalue
		expansion \cite[Theorem 1.1 and Remark 1.1(i)]{Kristaly}, together with the
		\(n\mapsto n+2\) dimension shift for the first \(l=1\) branch and the
		Benguria--Linde identification of the second ball eigenvalue
		\cite[Section 3, Lemma 3.1]{BenguriaLinde}, gives the ball-gap asymptotic.
		Using also
		\((r+\alpha)^{-2}=r^{-2}-2\alpha r^{-3}+O(r^{-4})\) uniformly for bounded
		\(\alpha\), we obtain, uniformly for \(0\le\alpha\le\log2\),
		\[
		\lambda_2(B_{r+\alpha}^{\HH^n})-\lambda_1(B_r^{\HH^n})
		=
		\left({4\pi^2\over n-1}-2\alpha\pi^2\right)r^{-3}+o(r^{-3}).
		\]
		The little-\(o(r^{-3})\) term is uniform for
		\(\alpha\in[0,\log2]\), since the expansion is evaluated at
		\(r+\alpha\in[r,r+\log2]\).
		For \(n=2,3\),
		\[
		\gamma_n:=\pi^2\left({4\over n-1}-2\log2\right)>0.
		\]
		After increasing \(R_n^{\rm ball}\), the preceding expansion with
		\(\alpha=\alpha_R\) gives
		\[
		\lambda_2(B_R^{\HH^n})-\lambda_1(B_r^{\HH^n})
		\ge {\gamma_n\over 2}r^{-3}.
		\]
		Since \(r\le R\), this is at least \((\gamma_n/2)R^{-3}\).
	\end{proof}
	
	This proposition is not used in the final all-dimensional proof below.  It is
	important, however, that the shortcut uses the original, not concentric,
	inball.  The comparison uses the nonconcentric inball and circumball, so the
	radius shift is only \(\log2\).  If one first replaced the domain by the
	concentric annulus
	\(B_{R-2\log2}(o)\subset\Om\subset B_R(o)\), the same ball arithmetic would
	use the cruder shift \(2\log2\), and the \(n=3\) coefficient would already have
	the wrong sign.  The radial-height proof avoids this scalar radius-loss by
	keeping the angular boundary displacement as part of the limiting problem.
	
	\section{From Horoconvexity to Lipschitz Radial Heights}
	
	Use the horospherical support function for horoconvex domains.  The
	Andrews--Chen--Wei framework defines the horoballs \(B_e(s)\) in Section 5.1
	and the horospherical support function in Section 5.2; in the published JEMS
	version, equations (5.1)--(5.2) define the support value and recover a compact
	horoconvex domain as the intersection of its supporting horoballs
	\cite[Sections 5.1--5.2]{ACW}.  We use this
	representation only to obtain a uniform
	Lipschitz bound for the polar radial height.  The later smoothing is ordinary
	spherical mollification of that radial height; the smooth approximating
	domains do not need to remain horoconvex.
	Thus, for compact nonsmooth sets, horoconvexity below is used in this
	supporting-horoball intersection sense.  To keep notation fixed, we denote
	the same supporting horoballs by \(\mathcal H_e(u)\) below rather than
	switching between \(B_e(s)\) and \(\mathcal H_e(u)\).
	
	The annular containment gives a polar radial-height description about \(o\).
	Writing the boundary as
	\[
	\rho(\theta)=R-h(\theta),\qquad \theta\in \SSph^{n-1},
	\]
	the fixed annular width gives
	\[
	0\leq h\leq 2\log 2.
	\]
	Figure~\ref{fig:radial-height} summarizes this representation.
	
	\begin{figure}[htbp]
		\centering
		\begin{tikzpicture}[
			scale=0.95,
			line cap=round,
			line join=round,
			>=Stealth,
			every node/.style={font=\small}
			]
			\definecolor{heightblue}{RGB}{20,116,124}
			\definecolor{heightgreen}{RGB}{92,132,96}
			\definecolor{heightgold}{RGB}{188,129,30}
			
			\node[anchor=west,font=\bfseries] at (-5.85,2.7) {Radial-height representation};
			
			\begin{scope}[shift={(-3.45,0)}]
				\node[font=\bfseries] at (0,2.15) {Annular control};
				\draw[black!55] (-2.0,1.95)--(2.0,1.95);
				\draw[black!60,thick] (0,0) circle (1.75);
				\draw[black!45] (0,0) circle (1.42);
				\foreach \a in {0,15,...,345} {
					\draw[heightblue!50] (0,0)--(\a:1.72);
				}
				\foreach \a in {7.5,37.5,...,337.5} {
					\draw[heightgreen!35] (0,0)--(\a:1.42);
				}
				\filldraw[fill=heightgreen!18,draw=heightgreen!75!black,very thick]
				plot[smooth cycle,tension=0.85] coordinates {
					(0:1.45) (32:1.58) (68:1.35) (108:1.49)
					(148:1.22) (191:1.50) (232:1.27) (275:1.54) (319:1.33)
				};
				\fill[heightblue!85!black] (0,0) circle (1.4pt);
				\draw[<->,heightgold,thick] (72:1.45)--(72:1.75);
				\node[heightgold!70!black] at (1.35,1.25) {$h(\theta)$};
				\draw[->,black!60] (-1.55,1.45) arc[start angle=135,end angle=170,radius=1.45];
				\node at (-2.0,1.05) {angle};
			\end{scope}
			
			\draw[->,black!55,thick] (-0.95,0)--(0.05,0);
			
			\begin{scope}[shift={(0.65,-1.35)}]
				\node[font=\bfseries] at (2.1,3.5) {Radial-height profile};
				\draw[black!55] (0.0,3.3)--(4.35,3.3);
				\draw[->,black!65] (0,0)--(4.55,0) node[right] {angle};
				\draw[->,black!65] (0,0)--(0,3.1);
				\node[rotate=90] at (-0.38,1.55) {height};
				
				\coordinate (p0) at (0,1.25);
				\coordinate (p1) at (0.7,1.72);
				\coordinate (p2) at (1.35,1.9);
				\coordinate (p3) at (2.2,1.25);
				\coordinate (p4) at (3.15,1.55);
				\coordinate (p5) at (4.2,2.25);
				
				\fill[heightgreen!16]
				(0,0)--(p0)
				.. controls (0.35,1.45) and (0.8,1.95) .. (p2)
				.. controls (1.75,1.85) and (1.8,1.15) .. (p3)
				.. controls (2.75,1.0) and (3.2,2.2) .. (p5)
				--(4.2,0)--cycle;
				\draw[heightgreen!70!black,very thick]
				(p0)
				.. controls (0.35,1.45) and (0.8,1.95) .. (p2)
				.. controls (1.75,1.85) and (1.8,1.15) .. (p3)
				.. controls (2.75,1.0) and (3.2,2.2) .. (p5);
				\draw[heightblue!90!black,very thick] (0,2.62)--(4.2,2.62);
				\node[left] at (0,2.62) {$R$};
				\draw[<->,heightgold,very thick] (1.35,2.62)--(1.35,1.9);
				\node[right,heightgold!70!black] at (1.38,2.25) {$h(\theta)$};
				\node[heightgreen!55!black] at (1.0,1.45) {$R-h(\theta)$};
				\draw[->,black!55] (3.5,1.8)--(3.95,2.2);
				\node[anchor=west] at (3.48,1.65) {Lipschitz};
			\end{scope}
		\end{tikzpicture}
		\caption{The annular containment \(B_{R-2\log2}(o)\subset\Omega\subset B_R(o)\)
			is rewritten as a radial graph \(\rho(\theta)=R-h(\theta)\).  The bounded
			annular width gives \(0\le h\le2\log2\), while the supporting-horoball
			description gives the Lipschitz control used below.}
		\label{fig:radial-height}
	\end{figure}
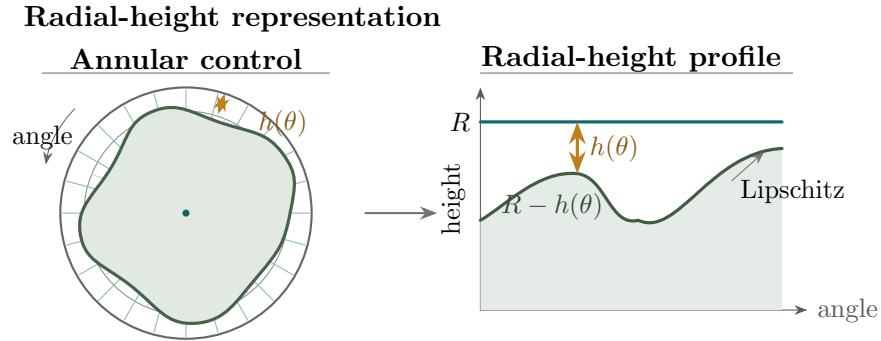
	
	\begin{lemma}[Radial representation from supporting horoballs]
		\label{lem:radial-representation}
		Fix \(R>C>0\), and suppose
		\[
		B_{R-C}(o)\subset \Om\subset B_R(o)
		\]
		with \(\Om\) compact and horoconvex.  Then \(\Om\) is star-shaped with
		respect to \(o\), its radial function
		\[
		\rho(\theta)=\sup\{t\ge0:\exp_o(t\theta)\in\Om\}
		\]
		is single-valued and finite, and
		\[
		\rho(\theta)=\inf_{e\in\SSph^{n-1}} \rho_{R,\sigma(e)}(\theta,e),\qquad
		h(\theta)=R-\rho(\theta)
		=\sup_{e\in\SSph^{n-1}} H_{R,\sigma(e),+}(\theta,e),
		\]
		where \(u_{\Om}(e)=R-\sigma(e)\), \(e\in\SSph^{n-1}\), is the horospherical
		support function and \(\rho_{R,\sigma(e)}(\theta,e)\) is the radial boundary
		of the supporting horoball with support value \(R-\sigma(e)\).
	\end{lemma}
	
	\begin{proof}
		Every supporting horoball in the ACW representation contains \(\Om\), hence
		contains \(B_{R-C}(o)\).  In particular it contains \(o\) in its interior.
		Since horoballs are geodesically convex, each supporting horoball meets a
		geodesic ray from \(o\) in an interval containing the initial point.  The
		endpoint may be included or excluded depending on the open-domain convention;
		only the radial endpoint matters here.  Intersecting these intervals over all
		supporting horoballs shows that \(\Om\) meets each ray in an interval
		\([0,\rho(\theta))\).  The outer inclusion \(\Om\subset B_R(o)\) makes
		\(\rho\) finite, and the inner inclusion gives \(\rho(\theta)\ge R-C\).
		The intersection formula for \(\Om\) then gives the infimum formula for
		\(\rho\).  Subtracting from \(R\) gives the supremum formula for \(h\), and
		the positive part may be inserted because \(\Om\subset B_R(o)\) implies
		\(h\ge0\).
	\end{proof}
	
	\begin{lemma}[Support-to-radial Lipschitz bound]
		\label{lem:sr-lip}
		Fix \(C<\infty\).  Suppose a compact horoconvex domain satisfies
		\[
		B_{R-C}(o)\subset \Om\subset B_R(o).
		\]
		Then, for all \(R\geq C+1\), its polar radial height
		\(h(\theta)=R-\rho(\theta)\) satisfies
		\[
		0\leq h\leq C,\qquad \Lip(h)\leq L(C),
		\qquad L(C)=\sqrt{2}\,e^{C/2}+1.
		\]
	\end{lemma}
	
	\begin{proof}
		The \(L^\infty\) bound follows from the annulus and
		Lemma~\ref{lem:radial-representation}.  Use the ACW support
		convention \cite[equations (5.1)--(5.2)]{ACW}
		\[
		u_{\Om}(e)
		=
		\sup_{x\in \Om}
		\log\bigl(\cosh\rho(x)-\sinh\rho(x)\langle\theta(x),e\rangle\bigr).
		\]
		The centered ball \(B_a(o)\) has support value \(a\), and monotonicity gives
		\[
		R-C\leq u_{\Om}(e)\leq R.
		\]
		Write \(u_{\Om}(e)=R-\sigma(e)\), with \(0\leq \sigma(e)\leq C\).  Since
		\(\Om\) is the intersection of its supporting horoballs, it is enough to estimate one
		supporting wall and then take suprema.  Here the centered geodesic balls
		\(B_a(o)\) are used only to normalize the support range; the single-wall
		kernel below is the boundary of a horoball
		\[
		\mathcal H_e(u)
		=
		\left\{
		\log(\cosh\rho-\sinh\rho\,\langle\theta,e\rangle)\leq u
		\right\}.
		\]
		This is exactly ACW's hyperboloid convention: if
		\(X=(\sinh\rho\,\theta,\cosh\rho)\), then
		\(-X\cdot(e,1)=\cosh\rho-\sinh\rho\langle\theta,e\rangle\).
		
		For a single supporting horoball with support value \(R-\sigma\), put
		\(c=\langle\theta,e\rangle\).  Its radial boundary is determined by
		\[
		\cosh\rho-\sinh\rho\,c=e^{R-\sigma}.
		\]
		Solving the resulting quadratic in \(e^\rho\) gives, for \(c<1\),
		\[
		\rho_{R,\sigma}(\theta,e)
		=
		\log
		\frac{e^{R-\sigma}+\sqrt{e^{2(R-\sigma)}-(1-c^2)}}{1-c}.
		\]
		Thus the height relative to \(B_R(o)\) is
		\[
		H_{R,\sigma}(\theta,e)
		=
		\sigma+\log(1-c)
		-\log\!\left(1+\sqrt{1-e^{-2(R-\sigma)}(1-c^2)}\right).
		\]
		Let
		\[
		q=\sqrt{1-e^{-2(R-\sigma)}(1-c^2)}.
		\]
		On the active set \(H_{R,\sigma}\geq0\),
		\[
		1-c\geq e^{-\sigma}(1+q)\geq e^{-C}.
		\]
		Indeed, \(H_{R,\sigma}\ge0\) is equivalent to
		\[
		e^{\rho_{R,\sigma}}\le e^R,
		\]
		and the quadratic solution
		\[
		e^{\rho_{R,\sigma}}
		=
		\frac{e^{R-\sigma}(1+q)}{1-c}
		\]
		then gives the displayed lower bound for \(1-c\).
		Moreover \(R\geq C+1\) implies \(q\geq\sqrt{1-e^{-2}}\).  Using
		\[
		|\nabla_{\SSph^{n-1}}c|^2=1-c^2,
		\]
		one obtains the exact identity
		\[
		\frac{dH_{R,\sigma}}{dc}
		=
		-\frac{1}{1-c}
		-\frac{e^{-2(R-\sigma)}c}{q(1+q)}
		=
		-\frac{q+c}{(1-c^2)q},
		\qquad
		|\nabla_{\SSph^{n-1}}H_{R,\sigma}|
		=
		\frac{|q+c|}{q\sqrt{1-c^2}}.
		\]
		The apparent singularity at \(c=-1\) is removable because
		\[
		q^2-c^2=(1-c^2)(1-e^{-2(R-\sigma)}),
		\]
		or, equivalently,
		\[
		\frac{dH_{R,\sigma}}{dc}
		=
		-\frac{1-e^{-2(R-\sigma)}}{(q-c)q}
		\]
		where the right side has a finite limit as \(c\to-1\),
		while the active-set lower bound on \(1-c\) keeps the kernel away from the
		singular direction \(c=1\).  Equivalently, the differentiated formula gives
		the crude estimate
		\[
		|\nabla_{\SSph^{n-1}}H_{R,\sigma}|
		\leq
		\sqrt{\frac{1+c}{1-c}}
		+\frac{e^{-2(R-\sigma)}|c|\sqrt{1-c^2}}{q(1+q)}
		\leq
		\sqrt{\frac{1+c}{1-c}}
		+\frac{e^{-2(R-\sigma)}}{q(1+q)}
		\leq \sqrt{2}\,e^{C/2}+1.
		\]
		In the last step, the first term is bounded on the active set by
		\(1-c\ge e^{-C}\), which gives
		\((1+c)/(1-c)\le2e^C\).  The second term is bounded by \(1\) because
		\(R-\sigma\ge1\) and \(q\ge\sqrt{1-e^{-2}}\).
		This estimate is only needed on the active side of the truncation.  Indeed,
		the contrapositive of the active-set bound gives
		\[
		1-c<e^{-C}\quad\Longrightarrow\quad H_{R,\sigma}<0.
		\]
		Thus the possible large-gradient region near \(c=1\) is killed by the
		positive-part truncation.  On the open set \(\{H_{R,\sigma}>0\}\) the
		displayed gradient bound holds, while on \(\{H_{R,\sigma}<0\}\) the positive
		part is identically zero.  The preceding implication gives an open
		neighborhood of the only singular direction \(c=1\) on which the positive
		part is identically zero; on the complement \(H_{R,\sigma}\) is smooth away
		from the removable point \(c=-1\).  Hence the weak spherical gradient of
		\(H_{R,\sigma,+}\) is
		\[
		\nabla_{\SSph^{n-1}}H_{R,\sigma}\,
		\mathbf 1_{\{H_{R,\sigma}>0\}}
		\]
		and is bounded in \(L^\infty\) by \(L(C)\).  Equivalently,
		\(H_{R,\sigma,+}\in W^{1,\infty}(\SSph^{n-1})\) with weak gradient bounded
		by \(L(C)\), and integration along minimizing spherical geodesics gives the
		same global Lipschitz bound.  Equivalently, the standard
		absolutely-continuous-on-lines property for \(W^{1,\infty}\) functions on the
		sphere integrates the weak-gradient bound along geodesic arcs, and continuity
		fixes the representative.  Therefore
		\(H_{R,\sigma,+}=\max\{0,H_{R,\sigma}\}\) is globally \(L(C)\)-Lipschitz.
		
		By Lemma~\ref{lem:radial-representation}, the radial boundary is the infimum
		of the supporting horoball boundaries:
		\[
		\rho(\theta)=\inf_e \rho_{R,\sigma(e)}(\theta,e).
		\]
		Consequently
		\[
		h(\theta)=R-\rho(\theta)
		=
		\sup_e H_{R,\sigma(e)}(\theta,e)
		=
		\sup_e H_{R,\sigma(e),+}(\theta,e),
		\]
		where the last equality uses \(\Om\subset B_R(o)\).  A supremum of functions
		with a common Lipschitz constant has the same Lipschitz constant; no
		regularity of the index map \(e\mapsto\sigma(e)\), nor uniqueness of a
		maximizing supporting wall, is used.  This proves the lemma.
	\end{proof}
	
	Choose a positive smooth approximate identity \(K_j\) on \(SO(n)\) and set
	\[
	h_j(\theta)=\int_{SO(n)}h(A\theta)K_j(A)\,dA.
	\]
	Here \(dA\) is the normalized Haar measure on \(SO(n)\).
	Then \(h_j\in C^\infty(\SSph^{n-1})\), \(h_j\to h\) uniformly,
	\[
	0\leq h_j\leq 2\log 2,
	\qquad
	\Lip(h_j)\leq \Lip(h)\leq L(2\log 2).
	\]
	The Lipschitz contraction follows because rotations are isometries of the
	sphere.  Thus the original domain is approximated by smooth radial-height
	domains
	\[
	\Om_R(h_j)=\{0\leq \rho<R-h_j(\theta)\}
	\]
	with common bounds
	\[
	0\leq h_j\leq 2\log 2,\qquad \Lip(h_j)\leq L(2\log 2).
	\]
	
%	\section{Radial-Height Spectral Input}
%	\label{sec:radial-height-input}
%	
%	The all-dimensional analytic input was stated in
%	Section~\ref{sec:main-result} as Theorem~\ref{thm:sh-lip}.  The next section
%	uses that input,, to prove Theorem~\ref{thm:main}.  
%	Proposition~\ref{prop:low-dim-ball-comparison} records an independent
%	shortcut in dimensions \(n=2,3\), but it is not used in the final
%	all-dimensional assembly.
	
	\section{Proof of Theorem~\ref{thm:main}}
	
	In this section, we prove Theorem~\ref{thm:main} using Theorem~\ref{thm:sh-lip} together with the geometric reduction and the support-to-radial bridge. The proof
	of Theorem~\ref{thm:sh-lip} will be given in
	Section~\ref{sec:radial-height-proof}.
	
	For the horoconvex theorem, set
	\[
	C_*=2\log 2,\qquad L_*=L(2\log 2),\qquad
	R_n^{\rm main}=R_0(n,C_*,L_*).
	\]
	Enlarge \(R_n^{\rm main}\), if necessary, so that
	\[
	R_n^{\rm main}>\max\{2\log 2,C_*+1\}.
	\]
	Define \(c_n=c(n,C_*)\).  Thus the final theorem constants are the
	radial-height constants in all dimensions, and depend only on \(n\).
	Choose
	\[
	D_n=\max(2R_n^{\rm main},4\log 2).
	\]
	
	Let \(\Om\subset\HH^n\) be compact and horoconvex, with diameter
	\(D\ge D_n\).  Let \(R\) be its circumradius and \(o\) a circumcenter.  Since
	\(\Om\subset B_R(o)\), one has \(D\le 2R\), hence
	\[
	R\ge D/2\ge R_n^{\rm main}.
	\]
	Conversely, any point of \(\Om\) is the center of a ball of radius \(D\)
	containing \(\Om\), so the minimality of the circumradius gives \(R\le D\).
	Thus \(R^{-3}\ge D^{-3}\).  The inequality \(R\ge R_n^{\rm main}\) also makes
	the radial-height threshold available.
	By the geometric reduction,
	\[
	B_{R-2\log 2}(o)\subset \Om\subset B_R(o).
	\]
	Write the radial boundary as \(\rho(\theta)=R-h(\theta)\).  Lemmas
	\ref{lem:radial-representation} and \ref{lem:sr-lip} give
	\[
	0\le h\le C_*=2\log 2,\qquad \Lip(h)\le L_*=L(2\log 2).
	\]
	Because \(h\) is continuous, the interior of the compact domain is exactly
	the open radial set
	\[
	\Om_R(h)=\{0\le \rho<R-h(\theta)\}.
	\]
	Apply Theorem~\ref{thm:sh-lip} to the Lipschitz radial-height domain
	\(\Om=\Om_R(h)\) with \(C=C_*\) and \(L=L_*\).  Since \(R\ge R_n^{\rm main}\),
	\[
	\lambda_2(\Om)-\lambda_1(\Om)
	\geq c(n,C_*)R^{-3}.
	\]
	Since \(c_n=c(n,2\log2)\) and \(R\le D\), this gives
	\[
	\lambda_2(\Om)-\lambda_1(\Om)\ge c_nD^{-3}.
	\]
	
	\section{Proof of the Radial-Height Theorem}
	\label{sec:radial-height-proof}
	
	This section proves Theorem~\ref{thm:sh-lip}.  The argument first proves the
	estimate for smooth heights with constants depending only on \(n,C,L\) as
	specified below; Lemma~\ref{lem:offset-squeeze} then passes the estimate to
	Lipschitz heights by smooth spherical mollification.
	
	\paragraph{Notation guide for Section~\ref{sec:radial-height-proof}.}
	To reduce collisions in the long radial-height proof, we keep the following
	conventions fixed.  The capital operator \(D_S\) is the spherical degree
	operator and the lower-case \(D_s\) is the conjugated radial derivative.  The
	symbol \(\psi\) denotes the Euler digamma function only inside the spectral
	multiplier \(\phi_n\); radial eigenfunctions are written with indexed symbols
	such as \(\psi_k\) or \(\psi_{k,l,R}\).  The letter \(\rho\) denotes the polar
	radius in the geometric reduction, while \(\rho_k\) denotes finite-box
	frequency roots and \(\rho_{\rm res}\) is the high-mode reserve fraction.  The
	unadorned \(\epsilon\) appears only in the auxiliary truncation used to prove
	the radial branch gap, whereas the collar scale in the main form estimates is
	\(\eps_R=R^{-\delta_{\rm cut}}\).  The height bound is always the fixed
	constant \(C\); local inequalities may use new constants \(c_0,C_0,c_1,C_1\)
	when they would otherwise collide with that height bound.
	
	\subsection{The pullback and splitting}
	
	The proof of the smooth-height estimate in Theorem~\ref{thm:sh-lip} begins by
	pulling the radial graph
	\[
	\Om_R(h)=\{0\le \rho<R-h(\theta)\}
	\]
	back to the fixed cylinder \([0,R]\times\SSph^{n-1}\), after the usual radial
	Schrodinger conjugation.  The pullback uses a cutoff flattening
	\[
	r=F_h(s,\theta)=s-\beta_R(s)h(\theta),
	\]
	which is the identity near the singular endpoint, equals the boundary shift at
	\(s=R\), and keeps all forms on a fixed Hilbert space.  This pullback is a
	fixed-coordinate model of the same radial-height domain.
	
	In that fixed space, an arbitrary test vector is decomposed into
	\[
	w=J_R f_{\leq}+q_{\rm rad,\leq}+w_>.
	\]
	Here \(J_Rf_{\leq}\) is the finite angular part of the degree-adapted first
	radial branch, \(q_{\rm rad,\leq}\) is the finite angular radial complement,
	and \(w_>\) contains the high angular modes.  The rest of the proof controls
	these three pieces separately:
	\begin{center}
		\begin{tabular}{@{}p{0.17\textwidth}p{0.35\textwidth}p{0.38\textwidth}@{}}
			\textbf{piece} & \textbf{estimate} & \textbf{purpose}\\ \hline
			\(J_R f_{\leq}\) &
			finite first-branch Hadamard flux &
			produces \(T_n+2\pi^2h+b_n\)\\
			\(q_{\rm rad,\leq}\) &
			radial branch gap and projected Green/Jost bounds &
			keeps radial complements above the low block\\
			\(w_>\) &
			high angular reserve and error absorption &
			keeps high modes above the low block
		\end{tabular}
	\end{center}
	The final lower-form assembly combines these three controls and then compares
	the compressed low block with the effective angular operator.  The last step of
	the section is only a passage from smooth heights to Lipschitz heights by the
	offset-domain squeeze.
	
	The constants are fixed in the order
	\[
	\delta_{\rm eff}\to \eta_{\rm gap}\to N_A\to c_A\to \rho_{\rm res}
	\to c_{\rm rad}^0,R_{\rm rad}\to c_{\rm rad}\to \eta_{\rm form}
	\to \delta_{\rm cut}\to M_C\to N\to R_0.
	\]
	Here \(N_A=N_A(n)\) is the dimension-only floor after which the unperturbed
	centrifugal coefficient satisfies \(c_l\ge\lambda_l/2\), and
	\(c_A=c_A(n)=1/2\) is the structural high-angular comparison constant,
	\(c_{\rm rad}^0=c_{\rm rad}^0(n)\) and
	\(R_{\rm rad}=R_{\rm rad}(n)\) come from the one-dimensional radial branch
	gap, and \(c_{\rm rad}\le c_{\rm rad}^0\) denotes the post-reservation
	radial-complement margin used later.  The final \(R_0\) is enlarged past
	\(R_{\rm rad}\).  The constant \(M_C\) records the bounded low/high
	first-branch coupling and depends only on \(C\).  The Lipschitz constant \(L\)
	enters through the fixed
	high-mode cutoff \(N\) and the final threshold \(R_0\), not through the gap
	constant \(\delta_{\rm eff}(n,C)\).
	Two analytic inputs are especially important.  The angular effective gap is a
	compactness statement for \(T_n+2\pi^2h\) on the sphere, proved from positivity
	of a subordinated Poisson semigroup.  The radial branch gap is a
	one-dimensional separated estimate; its critical two-dimensional channel is
	kept inside the all-dimensional radial-height theorem used by
	Theorem~\ref{thm:main}.
	
	\subsection{Effective angular operator and constant choices}
	\label{sec:constants}
	
	We now fix the constants used in the smooth-height estimate.  The estimates
	below use only the height and its first angular derivatives.  Whenever a
	displayed coordinate expansion contains \(\Delta_{\mathbb S^{n-1}}h\), the
	term is integrated by parts on the sphere before constants are fixed.
	
	The constant choices in Theorem~\ref{thm:sh-lip} are made before the
	large-\(R\) limit.  We first define the limiting angular operator.  Let
	\(D_S\) be the spherical degree operator.  The subscript \(S\) refers to the
	sphere; this operator is not the lower-case radial derivative \(D_s\) used
	later in this section.  It is characterized by
	\[
	D_SY_l=lY_l,
	\qquad
	D_S=\sqrt{-\Delta_{\SSph^{n-1}}+\left(\frac{n-2}{2}\right)^2}
	-\frac{n-2}{2}.
	\]
	Thus \(D_S\) is the usual first-order self-adjoint pseudodifferential operator
	obtained from the spherical Laplacian by spectral calculus.  Now set
	\[
	T_n=\phi_n(D_S),\qquad
	\phi_n(x)=2\pi^2
	\left[
	\psi\!\left(\frac{n-1}{2}+x\right)
	-\psi\!\left(\frac{n-1}{2}\right)
	\right],
	\]
	where \(\psi\) is the digamma function.  Thus \(T_n\) is a self-adjoint
	spectral multiplier on the sphere.  The displayed formula is the definition:
	the diagonal action on spherical harmonics is the corresponding
	spectral-theorem description:
	\[
	T_nY_l=\tau_{n,l}Y_l,\qquad
	\tau_{n,l}=b_{n+2l}-b_n,
	\]
	where \(b_m\) is the \(R^{-3}\) coefficient in the large-radius ball
	first-eigenvalue expansion in \(\HH^m\).  The equality
	\(\tau_{n,l}=\phi_n(l)\) follows from the digamma recurrence; in even
	dimensions the \(-\log 2\) terms in the coefficients \(b_m\) cancel.  The
	proof below writes out this check explicitly.  The operator \(T_n\) is the
	\(R^{-3}\)-scale angular kinetic operator obtained from the first radial
	branches of the separated ball operators \(H_{l,R}\).  It is nonlocal in the
	same sense that a fractional power of the Laplacian is nonlocal, since
	functional calculus builds it from the spherical harmonic decomposition.  With
	this notation let
	\[
	L_h=T_n+2\pi^2h+b_n.
	\]
	The additive scalar \(b_n\) is kept
	so that the limiting operator records the correct absolute \(R^{-3}\)
	coefficient, but it cancels from all spectral gaps.  The proof below supplies
	\[
	\mu_2(L_h)-\mu_1(L_h)\geq \delta_{\rm eff}(n,C)>0
	\]
	uniformly for \(0\leq h\leq C\).  Then set
	\[
	\eta_{\rm gap}=\delta_{\rm eff}(n,C)/32.
	\]
	Choose once and for all a dimension-only high-mode floor \(N_A(n)\) so that
	the angular coefficient
	\[
	c_l=l(l+n-2)+\frac{(n-1)(n-3)}4
	\]
	satisfies \(c_l\ge \lambda_l/2\) for every \(l>N_A(n)\), where
	\(\lambda_l=l(l+n-2)\).  Then the high-angular reserve estimate below holds
	with the dimension-only constant
	\[
	c_A(n)=1/2.
	\]
	Choose a reserve fraction
	\[
	\rho_{\rm res}\in(0,1/4)
	\]
	for the high-angular separated form before choosing the form-loss parameter.
	Let \(c_{\rm rad}^0(n)>0\) and \(R_{\rm rad}(n)<\infty\) be constants
	supplied by Lemma~\ref{lem:radial-branch-gap}.  These are fixed at this
	stage; later \(R_0\) is enlarged so that \(R_0\ge R_{\rm rad}(n)\).
	We reserve a fixed fraction of the radial-complement form for the
	first-branch/radial-complement mixed terms in
	Lemma~\ref{lem:first-radial-complement-mixed}.  After this reservation choose
	a smaller constant
	\[
	0<c_{\rm rad}(n)\le c_{\rm rad}^0(n)
	\]
	for the post-reservation radial-complement margin.  Every later displayed
	\(c_{\rm rad}R^{-2}\|q_{\rm rad}\|^2\) uses this smaller constant.  The
	parameter \(\eta\) in
	Lemma~\ref{lem:first-radial-complement-mixed} is fixed before
	\(\delta_{\rm cut}\) and \(N\) are chosen.
	Choose \(\eta_{\rm form}=\eta_{\rm form}(n,C,\rho_{\rm res})\) so the
	scale-invariant form losses leave at least half of the radial-complement
	margin.  The finite-low costs produced by these routings are
	\(o(R^{-3})\) after the final \(R_0\) enlargement.  The only fixed
	order-\(R^{-3}\) loss kept on the low first-branch block below is the
	\(\eta_{\rm gap}\) Young loss from the bounded low/high first-branch
	coupling.
	Choose the transition exponent
	\[
	\delta_{\rm cut}\in(1/2,2/3),\qquad
	\eps_R=R^{-\delta_{\rm cut}},
	\]
	where \(\delta_{\rm cut}<1\), and hence the stricter upper bound
	\(\delta_{\rm cut}<2/3\), gives
	\(R^{-1}\eps_R^{-1}=R^{\delta_{\rm cut}-1}=o(1)\) in the endpoint mass
	controls.  The lower bound fixes the endpoint collar on the small scale used
	in the cutoff estimates below, and all local remainders may depend on this
	already-chosen \(\delta_{\rm cut}\).
	Let \(M_C\) denote a fixed constant in the bounded low/high
	first-branch cross estimate in Lemma~\ref{lem:all-degree-lower-form-assembly}.
	For definiteness one may take
	\[
	M_C=4\pi^2 C+1 .
	\]
	This depends only on the height bound \(C\).  Since
	\(\|2\pi^2M_h\|_{L^2\to L^2}\le2\pi^2C\) and
	\(M_C=4\pi^2C+1>2\pi^2C\), this constant dominates the leading
	order-\(R^{-3}\) bounded multiplication part of the low/high first-branch
	coupling with no large-\(R\) enlargement.  The non-leading low/high
	corrections are routed separately into
	\(\tau'_{N,R}\operatorname{Sep}_>\) and the \(o_{N,C,L}(R^{-3})\)
	remainder in Lemma~\ref{lem:all-degree-lower-form-assembly}:
	\[
	|\operatorname{Coupl}_{\le,>}(f_{\leq},w_>)|
	\le M_CR^{-3}\|f_{\leq}\|\,\|w_>\|.
	\]
	Then choose the angular cutoff
	\[
	N=N(n,C,L,\eta_{\rm gap},\eta_{\rm form},\rho_{\rm res},M_C)\ge N_A(n),
	\]
	large enough to satisfy the high-mode reserve and compression requirements
	below, and also large enough that the finite-block ground-state energy
	approximates the full effective ground-state energy uniformly in
	\(0\le h\le C\):
	\[
	\mu_1(P_NL_hP_N|_{P_NL^2})
	\le \mu_1(L_h)+\delta_{\rm eff}/32 .
	\]
	This last uniform approximation follows from a uniform truncation estimate
	for the ground states of \(T_n+2\pi^2h\).  Let \(u_h\) be a normalized ground
	state of \(L_h\).  The min-max estimate with the constant test function gives
	\(\langle u_h,T_nu_h\rangle\le 2\pi^2C\), so the spectral theorem for \(T_n\)
	gives
	\[
	\|(I-P_N)u_h\|_2^2
	\le 2\pi^2C\,\tau_{n,N+1}^{-1}
	=:\kappa_N^2,\qquad \kappa_N\to0
	\]
	uniformly for \(0\le h\le C\).  Write \(g=P_Nu_h\).  Then
	\(\|g\|_2^2\ge1-\kappa_N^2\), and truncation does not increase the kinetic
	term:
	\[
	\langle g,T_ng\rangle\le \langle u_h,T_nu_h\rangle .
	\]
	For the potential term, \(0\le2\pi^2h\le2\pi^2C\) gives
	\[
	\bigl|2\pi^2\langle g,hg\rangle
	-2\pi^2\langle u_h,hu_h\rangle\bigr|
	\le 2\pi^2C(2\|(I-P_N)u_h\|_2+\|(I-P_N)u_h\|_2^2)
	\le 6\pi^2C\,\kappa_N
	\]
	after increasing \(N\) so \(\kappa_N\le1\).  Hence
	\[
	\langle g,L_hg\rangle
	\le \mu_1(L_h)+6\pi^2C\,\kappa_N,
	\]
	and therefore, using \(\|g\|_2^2\ge1-\kappa_N^2\),
	\[
	\mu_1(P_NL_hP_N|_{P_NL^2})-\mu_1(L_h)
	\le
	{6\pi^2C\,\kappa_N+|\mu_1(L_h)|\kappa_N^2
		\over 1-\kappa_N^2}.
	\]
	Moreover \(b_n\le\mu_1(L_h)\le b_n+2\pi^2C\), so \(\mu_1(L_h)\) is uniformly
	bounded for \(0\le h\le C\).  Choosing \(N\) large enough makes the last
	display at most \(\delta_{\rm eff}/32\), uniformly in \(0\le h\le C\).
	Finally
	\[
	R_0=R_0(n,C,L,\rho_{\rm res},\eta_{\rm form},\delta_{\rm cut},N).
	\]
	This threshold is also enlarged past the dimension-only radial threshold
	\(R_{\rm rad}(n)\).  After the previous choices are fixed, write it as
	\(R_0(n,C,L)\).  The final radial-height constant may be taken as
	\[
	c(n,C)=\delta_{\rm eff}(n,C)/4.
	\]
	This is the point where the \(L\)-dependence has disappeared from the
	constant.  The Lipschitz bound affects the proof only by enlarging the finite
	cutoff \(N\), the constants in the shear/Jacobian and low/high routing
	estimates, and the final threshold needed to absorb the corresponding
	\(o_{N,C,L}(R^{-3})\) remainders.  It does not enter the angular compactness
	gap \(\delta_{\rm eff}(n,C)\), the dimension-only radial branch constants, or
	the final value of \(c(n,C)\).
	Explicitly, after \(N\) and all earlier parameters are fixed, each occurrence
	of \(o_{N,C,L}(R^{-3})\) below means
	\[
	\epsilon_{N,C,L}(R)R^{-3},\qquad
	\epsilon_{N,C,L}(R)\to0
	\]
	as \(R\to\infty\), with \(N,C,L\) held fixed.  The final threshold
	\(R_0(n,C,L)\) is enlarged so that the finitely many such remainders appearing
	in the assembly are smaller than the displayed reserved fractions of
	\(\delta_{\rm eff}(n,C)R^{-3}\).  This enlargement changes only the
	threshold, not \(c(n,C)\).
	When a local lemma records the already-fixed collar exponent explicitly as
	\(o_{N,C,L,\delta_{\rm cut}}(R^{-3})\), this is absorbed into the aggregate
	notation \(o_{N,C,L}(R^{-3})\) below after \(\delta_{\rm cut}\) has been
	chosen.
	
	\subsection{Coordinates and the Degree-Adapted First Branch}
	
	For the unperturbed separated model, write the radial variable as \(s\).
	Start from the hyperbolic metric
	\[
	g=ds^2+\sinh^2(s)\gamma_{AB}d\theta^A d\theta^B
	\]
	and perform the physical radial Schrodinger conjugation: multiplication by
	\((\sinh s)^{(n-1)/2}\) carries the physical radial \(L^2\) measure
	\(\sinh^{n-1}s\,ds\) to the flat measure \(ds\).  With
	\[
	a_n=(n-1)/2,\qquad c_n^0=(n-1)(n-3)/4,
	\]
	set
	\[
	E_R=a_n^2+\pi^2/R^2.
	\]
	The separated radial operator in angular degree \(l\) is
	\[
	H_{l,R}
	=
	-\frac{d^2}{ds^2}+a_n^2+
	\bigl(l(l+n-2)+c_n^0\bigr)\sinh^{-2}(s),
	\]
	with the Dirichlet condition at \(s=R\) and the Friedrichs endpoint condition
	at \(s=0\).  Let \(y_{l,R}\) be the normalized positive first eigenfunction
	of \(H_{l,R}\).  For
	\[
	f(\theta)=\sum_{l,m}f_{l,m}Y_{l,m}(\theta),
	\]
	define the degree-adapted first-radial-branch lift
	\[
	J_R f=\sum_{l,m}f_{l,m}y_{l,R}(s)Y_{l,m}(\theta).
	\]
	The normalization is the \(L^2(0,R;ds)\) normalization after the physical
	radial Schrodinger conjugation, and the spherical harmonics are orthonormal in
	\(L^2(\mathbb S^{n-1})\).  Hence \(J_R\) is an exact isometry from
	\(L^2(\mathbb S^{n-1})\) into the fixed flattened Hilbert space:
	\[
	\|J_Rf\|_{L^2(ds\,d\theta)}=\|f\|_{L^2(\mathbb S^{n-1})}.
	\]
	All shifted forms below are written in this unitarily pulled-back Hilbert
	space; in physical variables this is the same equality after undoing the
	unitary radial conjugation.
	The lift is defined in the flattened coordinate \(s\); hence
	\(y_{l,R}(R)=0\) imposes the Dirichlet condition on the graph boundary after
	the pullback.
	This degree adaptation is essential: it produces the multiplier \(T_n\), not
	the ordinary angular Laplacian.
	In the later boundary-flattening calculation, \(r\) denotes the physical
	radial coordinate and \(s\in[0,R]\) denotes the fixed flattened coordinate,
	with \(r=F_h(s,\theta)\).
	
	\subsection{Boundary Flattening Tools}
	\label{sec:boundary-flattening-tools}
	
	These two estimates are the local pullback tools used later to route the
	shear and angular-Jacobian coefficient classes.  They are stated before the
	finite-mode analysis because they depend only on the flattening map, the
	height bounds, and the first angular derivatives of the height, not on the
	later radial branch or finite-block estimates.  We use the cutoff field
	\[
	F_h(s,\theta)=s-\beta_R(s)h(\theta),
	\]
	with \(\beta_R=0\) near the singular endpoint, \(\beta_R=s/R\) on the
	regular bulk, \(\beta_R(R)=1\), and
	\[
	\beta_R'=O(R^{-1}),\qquad
	\beta_R''=O(R^{-1}\eps_R^{-1})
	\]
	on the transition layer \(\eps_R\le s\le2\eps_R\).  The same field is
	used in the coefficient-class bookkeeping below.
	
	\begin{lemma}[Flattened shear form]
		\label{lem:flattened-shear-form}
		Let \(z\) be a smooth test function on the physical graph domain and set
		\(\widetilde z(s,\theta)=z(F_h(s,\theta),\theta)\), where
		\[
		F_h(s,\theta)=s-\beta_R(s)h(\theta).
		\]
		Assume \(|\beta_R'|\le C_\beta/R\), so that
		\(F_s=\partial_sF_h\ge1/2\) on the full cylinder once \(R\) is large in
		terms of \(C\) and \(C_\beta\).  Then the physical Dirichlet form pulls back as
		\[
		\begin{aligned}
			\int |\nabla z|^2\sinh^{n-1}r\,dr\,d\theta
			&=
			\int F_s^{-1}\sinh^{n-1}(F_h)|\partial_s\widetilde z|^2\,ds\,d\theta  \\
			&\quad+
			\int F_s\sinh^{n-3}(F_h)
			\left|\nabla_\theta\widetilde z
			+\frac{\beta_R}{F_s}\partial_s\widetilde z\,\nabla h\right|^2
			ds\,d\theta .
		\end{aligned}
		\]
		Consequently, if \(0\le h\le C\), \(\Lip(h)\le L\), and \(R\) is large
		enough in terms of \(C\), then for every \(\eta>0\)
		\[
		\begin{aligned}
			&\left|
			\int 2\beta_R\sinh^{n-3}(F_h)\,
			\partial_s\widetilde z\,\nabla h\cdot\nabla_\theta\widetilde z
			\,ds\,d\theta\right|  \\
			&\qquad\le
			\eta\int F_s^{-1}\sinh^{n-1}(F_h)|\partial_s\widetilde z|^2
			+C_{C,L}\eta^{-1}R^{-2}
			\int F_s\sinh^{n-3}(F_h)|\nabla_\theta\widetilde z|^2 .
		\end{aligned}
		\]
		The principal shear loss is therefore a small radial or angular kinetic loss,
		depending on the Young routing.  It is not a scalar
		\(\sinh^{-2}(F_h)|\widetilde z|^2\) term.  Weighted scalar terms enter only
		from lower-order Jacobian and radial half-density coefficients after the
		fixed Schrodinger conjugation.
	\end{lemma}
	
	\begin{proof}
		The chain rule gives
		\[
		\partial_r z=F_s^{-1}\partial_s\widetilde z,\qquad
		\nabla_\theta z=\nabla_\theta\widetilde z
		+\frac{\beta_R}{F_s}\partial_s\widetilde z\,\nabla h .
		\]
		Multiplying the radial and angular gradient terms by the pulled-back volume
		element \(F_s\sinh^{n-1}(F_h)\,ds\,d\theta\) gives the displayed identity.
		Since \(\beta_R=0\) near the singular endpoint and
		\(\beta_R\le s/R\) elsewhere, while \(F_h\ge s/2\) after increasing \(R\),
		\[
		\frac{\beta_R|\nabla h|}{\sinh(F_h)}\le \frac{C_{C,L}}{R}.
		\]
		Thus the mixed integral is bounded by
		\[
		\frac{C_{C,L}}{R}
		\left(\int F_s^{-1}\sinh^{n-1}(F_h)|\partial_s\widetilde z|^2\right)^{1/2}
		\left(\int F_s\sinh^{n-3}(F_h)|\nabla_\theta\widetilde z|^2\right)^{1/2},
		\]
		and Young's inequality with parameter \(t=\eta R/C_{C,L}\), followed by
		renaming the constant, proves the stated asymmetric estimate.
	\end{proof}
	
	\begin{lemma}[Angular Jacobian integration by parts]
		\label{lem:angular-jacobian-ibp}
		Let \(h_j\) be smooth heights with \(0\le h_j\le C\) and
		\(|\nabla h_j|\le L\), converging uniformly to a Lipschitz height \(h\).
		Set \(F_j=s-\beta_R(s)h_j(\theta)\).  Suppose an angular Jacobian or
		half-density expansion produces a scalar term of the form
		\[
		I_j[w]=
		\int b_j(s,\theta)w^2\,\Delta_{\mathbb S^{n-1}}h_j\,ds\,d\theta,
		\qquad
		b_j=\beta_R(s)\coth(F_j)\sinh^{-2}(F_j),
		\]
		or the same expression with \(b_j\) multiplied by a smooth coefficient of
		\(F_j\) whose value and first derivative are uniformly bounded.  Define
		\[
		\alpha_j(s,\theta)
		=
		\beta_R(s)^2
		\bigl(\coth^2(F_j)+\operatorname{csch}^2(F_j)\bigr).
		\]
		Then \(\alpha_j\le C_C\), and \(\alpha_j\le C_CR^{-2}\) on \(s\le1\).  For
		every \(\eta>0\),
		\[
		|I_j[w]|
		\le
		\eta\int \sinh^{-2}(F_j)|\nabla_\theta w|^2\,ds\,d\theta
		+C_{C,L}\eta^{-1}
		\int \alpha_j\sinh^{-2}(F_j)|w|^2\,ds\,d\theta ,
		\]
		with constants independent of \(j\) and \(R\) large.
	\end{lemma}
	
	\begin{proof}
		Integrate by parts on the sphere.  Since
		\(\Delta_{\mathbb S^{n-1}}h_j=\operatorname{div}\nabla h_j\),
		\[
		I_j[w]=
		-\int \nabla h_j\cdot\nabla_\theta(b_jw^2)\,ds\,d\theta .
		\]
		The derivative falling on \(w^2\) is bounded by
		\[
		C_L\int \beta_R\coth(F_j)\sinh^{-2}(F_j)
		|w|\,|\nabla_\theta w|\,ds\,d\theta .
		\]
		Because \(\beta_R=0\) near \(s=0\), \(\beta_R\le s/R\), and
		\(F_j\ge s/2\) after increasing \(R\), the factor
		\((\beta_R\coth(F_j))^2\) is bounded by \(\alpha_j\).  Young's inequality
		gives the stated angular kinetic loss and the
		\(\alpha_j\sinh^{-2}(F_j)|w|^2\) scalar remainder.
		
		When the derivative falls on \(b_j\), the identity
		\[
		\nabla_\theta\{\coth(F_j)\sinh^{-2}(F_j)\}
		=
		O\!\left((\operatorname{csch}^2(F_j)+\coth^2(F_j))\sinh^{-2}(F_j)\right)
		\beta_R\nabla h_j
		\]
		and \(|\nabla h_j|\le L\) give
		\[
		|\nabla_\theta b_j|
		\le C_{C,L}\alpha_j\sinh^{-2}(F_j),
		\]
		again using
		\(\alpha_j=\beta_R^2(\coth^2(F_j)+\operatorname{csch}^2(F_j))\).
		This proves the scalar bound.  The version with a bounded smooth coefficient
		has the same proof: when the derivative falls on the coefficient, it supplies
		one factor \(\beta_R|\nabla h_j|\), while \(b_j\) already supplies the factor
		\(\beta_R\coth(F_j)\sinh^{-2}(F_j)\), and the product is again controlled by
		\(\alpha_j\sinh^{-2}(F_j)\).  Finally, \(\alpha_j\le C_C\) everywhere, while
		on \(s\le1\) the estimate \(\beta_R\le s/R\) and
		\(\coth^2(F_j)+\operatorname{csch}^2(F_j)\le C/s^2\) give
		\(\alpha_j\le C_CR^{-2}\).
	\end{proof}
	
	\paragraph{Ball asymptotics and the dimension shift.}
	These give the first-radial-branch coefficients.  In angular degree \(l\),
	the singular coefficient is
	\[
	l(l+n-2)+\frac{(n-1)(n-3)}4
	=
	\frac{(n+2l-1)(n+2l-3)}4,
	\]
	which is exactly the radial first-eigenvalue coefficient in dimension
	\(m=n+2l\).  This matches the full coefficient, not only the inner factor in
	parentheses.  Thus the tempting shift obtained by matching only
	\(l(l+n-2)\) is wrong; for \(l=1\) the admissible shift is \(n+2\), not
	\(n+4\).  If \(e_{l,R}^{(n)}\) is the first radial eigenvalue in degree \(l\)
	for the dimension-\(n\) ball, then
	\[
	e_{l,R}^{(n)}
	=
	\lambda_1^{(n+2l)}(R)
	-\bigl(a_{n+2l}^2-a_n^2\bigr).
	\]
	Krist\'aly's large-radius first-eigenvalue expansion
	\cite[Theorem 1.1, Remark 1.1(i), and the root asymptotics derived in
	equations (3.13), (3.23)]{Kristaly} therefore gives, for each fixed \(l\),
	\[
	e_{l,R}^{(n)}
	=
	a_n^2+\pi^2/R^2+b_{n+2l}R^{-3}+o_l(R^{-3}).
	\]
	The appeal to Remark~1.1(i) is used here with the following precision: after
	the fixed dimension shift \(m=n+2l\), the first four terms of the ball-root
	expansion leave a genuine little-\(o_m(R^{-3})\) remainder, not merely a
	bounded \(O_m(R^{-3})\) error.  This is the precision needed later when a
	finite set of degrees is maximized after \(N\) has been fixed.
	Explicitly, for odd \(m=2p+1\),
	\[
	b_m=2\pi^2\sum_{j=1}^{p-1}\frac1j
	\]
	with the empty sum equal to \(0\), while for even \(m=2p\),
	\[
	b_m=4\pi^2
	\left(\sum_{j=1}^{p-1}\frac{1}{2j-1}-\log 2\right).
	\]
	These are just the \(R^{-3}\) coefficients obtained by expanding the
	squares in Krist\'aly's odd and even asymptotic formulas; the case \(m=2\)
	uses the stated empty-sum convention.
	Only finitely many degrees \(l\le N\) enter the lower-form reduction, and
	\(N\) is fixed before \(R_0\); hence the maximum of these fixed-degree
	remainders is still \(o(R^{-3})\).  The diagonal operator is then defined by
	the exact coefficient identity \(T_nY_l=(b_{n+2l}-b_n)Y_l\).
	Combining Krist\'aly's odd and even coefficient formulas gives the algebraic
	digamma form used below, where
	\(\psi(x)=\Gamma'(x)/\Gamma(x)\) denotes the Euler digamma function,
	\[
	b_{n+2l}-b_n
	=2\pi^2
	\left[
	\psi\!\left(\frac{n-1}{2}+l\right)
	-\psi\!\left(\frac{n-1}{2}\right)
	\right].
	\]
	Indeed, if \(n=2p+1\) is odd then
	\[
	b_{n+2l}-b_n
	=
	2\pi^2\sum_{j=p}^{p+l-1}{1\over j}
	=
	2\pi^2[\psi(p+l)-\psi(p)].
	\]
	If \(n=2p\) is even, the \(-\log 2\) terms cancel and
	\[
	b_{n+2l}-b_n
	=
	4\pi^2\sum_{j=p}^{p+l-1}{1\over 2j-1}
	=
	2\pi^2
	\left[\psi\!\left(p+l-\frac12\right)
	-\psi\!\left(p-\frac12\right)\right].
	\]
	These are the displayed formula with \((n-1)/2=p\) in the odd case and
	\((n-1)/2=p-\frac12\) in the even case.
	In particular, the \(l=1\) coefficient is
	\[
	b_{n+2}-b_n
	=2\pi^2\left(\frac{2}{n-1}\right)
	=\frac{4\pi^2}{n-1},
	\]
	using \(\psi(x+1)-\psi(x)=1/x\), matching the ball gap asymptotic.
	Benguria--Linde
	\cite[Section 3, Lemma 3.1]{BenguriaLinde} identifies the second Dirichlet
	eigenvalue of a hyperbolic ball with the first \(l=1\) branch.  Equivalently,
	for the large-\(R\) calibration needed here, the radial branch gap proved
	below puts the second \(l=0\) radial level at distance at least \(cR^{-2}\)
	above the ground state, for some \(c>0\), while the first \(l=1\) level has gap
	\(4\pi^2((n-1)R^3)^{-1}+o(R^{-3})\).  The \(R^{-3}\) coefficients above come
	from Krist\'aly through the dimension shift.
	
	\paragraph{Angular effective gap.}
	\begin{lemma}[Angular effective gap]
		\label{lem:angular-effective-gap}
		Fix \(n\ge2\) and \(C<\infty\).  There is a constant
		\(\delta_{\rm eff}(n,C)>0\) such that, for every measurable
		\(h:\mathbb S^{n-1}\to[0,C]\),
		\[
		\mu_2(T_n+2\pi^2h+b_n)-\mu_1(T_n+2\pi^2h+b_n)
		\geq \delta_{\rm eff}(n,C).
		\]
	\end{lemma}
	
	\begin{proof}
		The scalar \(b_n\) does not affect gaps, so write
		\[
		A_h=T_n+2\pi^2h .
		\]
		We first record that the first eigenvalue of \(A_h\) is simple for every
		measurable \(h\) with \(0\le h\le C\).  We use the functional-calculus
		definition of \(T_n\) from Subsection~\ref{sec:constants}: if \(D_S\) is the
		spherical degree operator, \(D_SY_l=lY_l\), equivalently
		\[
		D_S=\sqrt{-\Delta_{\SSph^{n-1}}+\left(\frac{n-2}{2}\right)^2}
		-\frac{n-2}{2}.
		\]
		Again, this is an angular operator on \(\SSph^{n-1}\), not the radial
		derivative \(D_s\) used in the one-dimensional estimates below.  The operator
		\(T_n\) is defined from \(D_S\) by spectral calculus, and the dimension-shift
		asymptotic identifies the scalar function:
		\[
		T_n=\phi_n(D_S),\qquad
		\phi_n(x)=2\pi^2
		\left[
		\psi\!\left(\frac{n-1}{2}+x\right)
		-\psi\!\left(\frac{n-1}{2}\right)
		\right],
		\]
		so \(T_nY_l=(b_{n+2l}-b_n)Y_l\).  For \(a=(n-1)/2\), the integral
		representation
		\[
		\psi(a+x)-\psi(a)
		=
		\int_0^\infty
		\frac{e^{-at}(1-e^{-xt})}{1-e^{-t}}\,dt
		\]
		shows that \(\phi_n\) is a Bernstein function.  Hence \(e^{-\sigma T_n}\)
		is a Bochner subordinate of the spherical Poisson semigroup \(e^{-tD_S}\)
		\cite[Theorem 3.2 and Chapter 13]{SSV}.  For \(r=e^{-t}\), this semigroup
		agrees with integration against the spherical Poisson kernel
		\[
		P_r(\theta,\eta)
		=
		|\SSph^{n-1}|^{-1}
		\frac{1-r^2}{(1-2r\,\theta\cdot\eta+r^2)^{n/2}},
		\]
		with respect to surface measure \(d\eta\).  Indeed, by the standard
		Poisson-kernel expansion, equivalently by Funk--Hecke, the degree \(l\)
		zonal component of \(P_r\) is \(r^l\), so the integral operator acts on
		degree \(l\) harmonics by the multiplier \(r^l=e^{-tl}\).
		Let \(\mu_\sigma\) be the subordinator law.  Since \(\phi_n(0)=0\) and
		\(\phi_n(\lambda)\to\infty\),
		\[
		\mu_\sigma(\{0\})
		=
		\lim_{\lambda\to\infty}\int_0^\infty e^{-\lambda t}\,d\mu_\sigma(t)
		=
		\lim_{\lambda\to\infty}e^{-\sigma\phi_n(\lambda)}
		=0 .
		\]
		Thus there is a compact interval \([t_0,t_1]\subset(0,\infty)\) with
		\(\mu_\sigma([t_0,t_1])>0\).  On this interval the Poisson kernel is bounded
		below by a positive constant; for instance
		\[
		P_{e^{-t}}(\theta,\eta)
		\ge
		|\SSph^{n-1}|^{-1}\frac{1-e^{-2t_0}}{2^n}
		=:p_* .
		\]
		Therefore, for \(f\ge0\),
		\[
		(e^{-\sigma T_n}f)(\theta)
		=
		\int_0^\infty (e^{-tD_S}f)(\theta)\,d\mu_\sigma(t)
		\ge
		m_\sigma\int_{\SSph^{n-1}}f(\eta)\,d\eta,\qquad
		m_\sigma:=\mu_\sigma([t_0,t_1])p_*>0 .
		\]
		Let \(V_h=2\pi^2h\).  Since \(T_n\) is self-adjoint and nonnegative and
		\(V_h\) is a bounded symmetric multiplication operator,
		\(A_h=T_n+V_h\) is semibounded self-adjoint.  The semigroup Trotter product
		formula \cite[Theorem 5.12]{Teschl}, together with the inequalities
		\[
		e^{-2\pi^2C\sigma/k}I\le e^{-(\sigma/k)V_h}\le I
		\]
		give, first for the positive Trotter products and then by strong \(L^2\)
		convergence,
		\[
		e^{-\sigma A_h}f
		\ge
		e^{-2\pi^2C\sigma}e^{-\sigma T_n}f
		\ge
		e^{-2\pi^2C\sigma}m_\sigma
		\int_{\SSph^{n-1}}f
		\]
		for every nonzero \(f\ge0\).  Since \(T_n\) has compact resolvent and \(V_h\)
		is bounded, \(e^{-\sigma A_h}\) is compact.  The last display makes it
		positivity improving.  By the Perron--Frobenius theorem for compact
		positivity-improving operators \cite[Theorem XIII.43]{ReedSimonIV}, the
		spectral radius of \(e^{-\sigma A_h}\) is a simple eigenvalue with a strictly
		positive eigenvector.  Equivalently, the first eigenvalue of \(A_h\) is simple
		for every such \(h\).
		Since \(A_h\) is real and self-adjoint, this real simplicity is the same as
		complex simplicity: the real and imaginary parts of any complex first
		eigenfunction are real first eigenfunctions.
		
		We now prove uniformity by compactness.  Let
		\[
		\mathcal H_C=\{h\in L^\infty(\SSph^{n-1}):0\le h\le C\}.
		\]
		The unit ball of \(L^\infty\) is weak-\(*\) compact, and here it is
		sequentially compact on bounded sets because \(L^1(\SSph^{n-1})\) is
		separable.  Suppose, to the contrary, that the gaps of \(A_h\) are not
		uniformly bounded below on \(\mathcal H_C\).  Then there are \(h_j\in
		\mathcal H_C\) such that
		\[
		\mu_2(A_{h_j})-\mu_1(A_{h_j})\to0 .
		\]
		After passing to a subsequence, \(h_j\to h\) weak-\(*\) in \(L^\infty\) for
		some \(h\in\mathcal H_C\).
		
		Let \(u_j,v_j\) be \(L^2\)-orthonormal eigenfunctions for the first two
		eigenvalues of \(A_{h_j}\).  The min-max principle and \(0\le V_{h_j}\le
		2\pi^2C\) give
		\[
		\mu_2(A_{h_j})\le \mu_2(T_n)+2\pi^2C,
		\]
		so
		\[
		\langle u_j,T_nu_j\rangle,\ \langle v_j,T_nv_j\rangle
		\le \mu_2(T_n)+2\pi^2C .
		\]
		Thus \(u_j\) and \(v_j\) are bounded in the form domain
		\[
		\mathcal Q=\operatorname{Dom}(T_n^{1/2}),\qquad
		\|w\|_{\mathcal Q}^2=\|w\|_2^2+\langle w,T_nw\rangle .
		\]
		Since \(T_n\) has eigenvalues tending to infinity and each spherical-harmonic
		degree has finite multiplicity, the eigenvalue list with multiplicity tends
		to infinity.  Hence the embedding \(\mathcal Q\hookrightarrow
		L^2(\SSph^{n-1})\) is compact.  Passing to a further subsequence,
		\[
		u_j\to u,\qquad v_j\to v
		\quad\text{strongly in }L^2,
		\]
		and weakly in \(\mathcal Q\).  The limits are still orthonormal.  Passing to
		a further subsequence, the two eigenvalues converge to a common value
		\(\lambda\), because their difference tends to zero.
		
		For each fixed \(\varphi\in\mathcal Q\), the products
		\(u_j\overline{\varphi}\) and \(v_j\overline{\varphi}\) converge in \(L^1\).
		For instance,
		\[
		\begin{aligned}
			&\left|\int h_j u_j\overline{\varphi}
			-\int h u\overline{\varphi}\right|  \\
			&\qquad\le
			C\|(u_j-u)\overline{\varphi}\|_{L^1}
			+\left|\int (h_j-h)u\overline{\varphi}\right|
			\to0 .
		\end{aligned}
		\]
		The same argument applies to \(v_j\).  Together with weak convergence in
		\(\mathcal Q\), this allows passage to the limit in the weak eigenvalue
		equations.  More explicitly, the equations are used in form-pairing form:
		for every \(\varphi\in\mathcal Q\),
		\[
		\langle T_n^{1/2}u_j,T_n^{1/2}\varphi\rangle
		+2\pi^2\int h_ju_j\overline\varphi
		=
		\mu_1(A_{h_j})\langle u_j,\varphi\rangle,
		\]
		and similarly for \(v_j\).  Weak convergence in \(\mathcal Q\) gives
		convergence of the kinetic half-power pairing, while the displayed \(L^1\)
		estimate gives convergence of the potential term.  Hence
		\[
		A_hu=\lambda u,\qquad A_hv=\lambda v .
		\]
		It remains to identify \(\lambda\) as the first eigenvalue of \(A_h\).  For
		any normalized \(\varphi\in\mathcal Q\),
		\[
		\mu_1(A_{h_j})\le
		\langle\varphi,T_n\varphi\rangle
		+2\pi^2\int_{\SSph^{n-1}}h_j|\varphi|^2
		\to
		\langle\varphi,T_n\varphi\rangle
		+2\pi^2\int_{\SSph^{n-1}}h|\varphi|^2 .
		\]
		Taking the infimum over \(\varphi\) gives \(\lambda\le\mu_1(A_h)\).  Since
		\(\lambda\) is an eigenvalue of \(A_h\), it is not below the bottom of the
		spectrum, so \(\mu_1(A_h)\le\lambda\).  Thus \(\lambda=\mu_1(A_h)\), and
		\(u,v\) are two orthonormal first eigenfunctions of \(A_h\).  This contradicts
		the simplicity proved above.  Therefore
		\[
		\inf_{h\in\mathcal H_C}\bigl(\mu_2(A_h)-\mu_1(A_h)\bigr)>0 .
		\]
		This infimum is the required \(\delta_{\rm eff}(n,C)\), and adding the scalar
		\(b_n\) does not change the gap.
	\end{proof}
	
	\paragraph{Fixed angular truncation.}
	Choose one cutoff
	\[
	N=N(n,C,L,\eta_{\rm gap},\eta_{\rm form},\rho_{\rm res},M_C)
	\]
	with \(N\ge N_A(n)\), before the large-\(R\) threshold is chosen.
	
	\paragraph{Overview of the finite-mode estimates.}
	The next technical block has four separate jobs.  The angular effective gap
	above is a compactness statement on the sphere: after the radial first branch
	has been reduced to the angular form \(T_n+2\pi^2h+b_n\), it supplies a
	uniform first-to-second angular gap.  The radial branch gap below is a
	one-dimensional statement for each separated radial row; its only delicate
	case is the critical \((n,l)=(2,0)\) channel, where the second eigenvalue is
	controlled by a critical Bessel comparison and the first eigenvalue by a
	simple sine trial.  The finite-box Jost data lemma then records the ODE and
	Volterra estimates needed for normalized radial eigenfunctions on the long
	interval.  Finally, the finite first-branch Hadamard lemma converts those
	radial endpoint estimates into the \(2\pi^2h\) boundary-flux matrix used in
	the lower form.  Thus the first two lemmas create the gap/reserve, while the
	last two lemmas compute the finite first-branch coefficient to order
	\(R^{-3}\).
	
	\paragraph{Radial branch gap.}
	\begin{lemma}[Radial branch gap and critical channel]
		\label{lem:radial-branch-gap}
		For each \(n\ge2\) there are constants
		\[
		c_{\rm rad}^0(n)>0,\qquad R_{\rm rad}(n)<\infty
		\]
		such that, for every angular degree \(l\ge0\) and every
		\(R\ge R_{\rm rad}(n)\), the first two Friedrichs-Dirichlet eigenvalues of
		the separated radial operator \(H_{l,R}\) satisfy
		\[
		\lambda_{2,l,R}-\lambda_{1,l,R}\ge c_{\rm rad}^0(n)R^{-2}.
		\]
	\end{lemma}
	
	\begin{proof}
		Write
		\[
		c_l=l(l+n-2)+\frac{(n-1)(n-3)}{4}.
		\]
		Except in the critical channel \((n,l)=(2,0)\), one has \(c_l\ge0\), and the
		only zero case is \((n,l)=(3,0)\).  For \(c_l\ge0\), the potential
		\(c_l\sinh^{-2}s\) is convex on \((0,\infty)\), since
		\[
		\frac{d^2}{ds^2}\sinh^{-2}s
		=
		4\coth^2s\,\sinh^{-2}s+2\sinh^{-4}s>0.
		\]
		This includes the zero potential in the \((n,l)=(3,0)\) channel.  Truncate
		the endpoint to
		\((\epsilon,R)\) and impose Dirichlet conditions at both endpoints.  The
		one-dimensional convex-potential gap theorem \cite{Lavine,AndrewsClutterbuck}
		then gives
		\[
		\lambda_2(H_{l,R}^{\epsilon})-\lambda_1(H_{l,R}^{\epsilon})
		\ge 3\pi^2(R-\epsilon)^{-2}.
		\]
		Extend \(H_0^1(\epsilon,R)\) functions by zero to \((0,R)\).  As
		\(\epsilon\downarrow0\) these closed form domains increase monotonically, and
		their union is dense in the Friedrichs form domain, the completion of
		\(C_c^\infty(0,R)\) under the form norm.  The corresponding eigenvalues
		therefore converge to the singular-endpoint Friedrichs eigenvalues by
		monotone form convergence.  For each \(\epsilon\),
		\[
		\lambda_2(H_{l,R}^{\epsilon})
		\geq
		\lambda_1(H_{l,R}^{\epsilon})+3\pi^2(R-\epsilon)^{-2}.
		\]
		Taking limits in this inequality gives
		\(\lambda_{2,l,R}\geq\lambda_{1,l,R}+3\pi^2R^{-2}\).
		The right side of the truncated gap bound depends only on the interval length,
		not on the coefficient \(c_l\).  Consequently the same limiting inequality
		holds for every noncritical angular degree with the same constant
		\(3\pi^2\).  No uniform-in-\(l\) convergence rate as
		\(\epsilon\downarrow0\) is being used; the limit is taken degree by degree
		after a bound whose right side is already independent of \(l\).
		
		For the exceptional critical channel, where \((n,l)=(2,0)\), first subtract
		the harmless additive constant \(a_2^2=1/4\):
		\[
		\widehat H_R=H_R-\frac14
		=-\frac{d^2}{ds^2}-\frac{1}{4}\sinh^{-2}s.
		\]
		This does not change the branch gap.  Compare \(\widehat H_R\) with the exact
		critical Bessel operator
		\[
		B_R=-\frac{d^2}{ds^2}-\frac{1}{4}s^{-2}
		\]
		with the same Friedrichs endpoint at \(0\) and Dirichlet endpoint at \(R\).
		Because
		\[
		\sinh^{-2}s=s^{-2}-\frac13+O(s^2)\qquad(s\downarrow0),
		\]
		the difference \(\frac14(s^{-2}-\sinh^{-2}s)\) extends to a bounded potential
		at the singular endpoint.  Thus \(\widehat H_R\) and \(B_R\) are closures of
		the same critical Friedrichs core, and have the same form domain: the bounded
		potential perturbation is continuous with respect to the closed critical
		Friedrichs form and does not change its closure.  Since \(\sinh s\ge s\), one
		has \(\widehat H_R\ge B_R\) in form order on that common domain, and the
		Courant--Fischer min-max monotonicity for closed semibounded forms therefore
		gives \(\lambda_k(\widehat H_R)\ge\lambda_k(B_R)\) for every indexed
		eigenvalue \(k\).  In particular, the critical Bessel spectrum gives
		\[
		\lambda_2(\widehat H_R)\ge j_{0,2}^2R^{-2}.
		\]
		Here the Friedrichs eigenfunctions of \(B_R\) are
		\(\sqrt{s}\,J_0(j_{0,k}s/R)\); the companion
		\(\sqrt{s}\,Y_0(j_{0,k}s/R)\sim\sqrt{s}\log s\) branch is excluded by the
		critical Friedrichs endpoint condition.
		For the first eigenvalue, the simpler Dirichlet trial
		\[
		u_R(s)=\sin(\pi s/R)
		\]
		is admissible in the same critical Friedrichs form domain: it vanishes at
		\(R\), has \(u_R(s)=O(s)\) at the singular endpoint, and has finite critical
		form energy.  Since the potential in \(\widehat H_R\) is nonpositive,
		\[
		\begin{aligned}
			\lambda_1(\widehat H_R)
			&\le
			\frac{\int_0^R |u_R'(s)|^2\,ds
				-\frac14\int_0^R\sinh^{-2}s\,|u_R(s)|^2\,ds}
			{\int_0^R |u_R(s)|^2\,ds}  \\
			&\le
			\frac{\int_0^R |u_R'(s)|^2\,ds}
			{\int_0^R |u_R(s)|^2\,ds}
			=
			\pi^2R^{-2}.
		\end{aligned}
		\]
		The second positive zero of \(J_0\) satisfies \(j_{0,2}>5>\pi\)
		\cite[Sections 10.21 and 10.75(iii)]{DLMF}, so
		\[
		j_{0,2}^2-\pi^2>0 .
		\]
		Thus \(\widehat H_R\), and hence \(H_R\), has a critical branch gap
		\(\ge c_{\rm crit}R^{-2}\) with, for instance,
		\[
		c_{\rm crit}=\frac12\bigl(j_{0,2}^2-\pi^2\bigr)>0 .
		\]
		Combining this with the noncritical
		channels fixes the dimension-only constants \(c_{\rm rad}^0(n)>0\) and
		\(R_{\rm rad}(n)<\infty\) used in the lower form.  This uses the Bessel
		zero reference cited above, together
		with the standard Friedrichs-extension and Sturm--Liouville endpoint
		framework \cite[Section 2.3 and Chapter 9]{Teschl}.
	\end{proof}
	
	\paragraph{Finite-box Jost data.}
	\begin{lemma}[Finite-box Jost data]
		\label{lem:finite-box-jost-data}
		Fix the angular cutoff \(N\).  For the finitely many separated radial
		operators \(H_{l,R}\), \(0\le l\le N\), the regular finite-box construction
		has constants \(R_N<\infty\) and \(0<c_0<C_0<\infty\), depending only on \(N\)
		and \(n\), with the following properties.  First, the normalized first branch
		satisfies, uniformly for \(l\le N\),
		\[
		\partial_s y_{l,R}(R)
		=
		-\sqrt{2/R}\,\pi/R+O_N(R^{-5/2}).
		\]
		Second, in the critical channel \((n,l)=(2,0)\), let
		\((\lambda_{k,0,R},\psi_k)\) be the \(L^2(0,R)\)-normalized eigenpairs,
		ordered increasingly, and set
		\[
		\rho_k^2=\lambda_{k,0,R}-a_2^2,\qquad
		\mu_k=\lambda_{k,0,R}-\lambda_{1,0,R}
		=\rho_k^2-\rho_1^2 .
		\]
		There is a number \(\rho_0>0\) such that, for \(R\ge R_N\) and \(k\ge2\),
		\[
		\mu_k\ge c_0k^2R^{-2},\qquad
		c_0\,\frac{k}{R}\le \rho_k\le C_0\,\frac{k}{R},
		\]
		and
		\[
		\begin{array}{ll}
			|\psi_k(s)|\le CkR^{-3/2}s^{1/2},
			& \rho_k<\rho_0/2,\\[3pt]
			|\psi_k(s)|\le Ck^{1/2}R^{-1}s^{1/2},
			& \rho_k\ge\rho_0/2,\ \rho_k s\le1,\\[3pt]
			|\psi_k(s)|\le CR^{-1/2},
			& \rho_k s\ge1,
		\end{array}
		\qquad 0<s\le1 .
		\]
	\end{lemma}
	
	\noindent\emph{Orientation for this lemma.}
	This lemma supplies the finite-box ODE data used later by the Hadamard and
	projected-Green estimates.  Its first output is the common first-branch
	boundary slope \(-\sqrt{2/R}\,\pi/R+O_N(R^{-5/2})\).  Its second output,
	needed only in the critical \((n,l)=(2,0)\) channel, is a controlled list of
	radial spectral gaps, root-growth bounds, and endpoint amplitudes for the
	radial complement.  The proof chooses the regular tail and Volterra
	parameters, normalizes the first branch, compares the critical roots with the
	Bessel model and the sine-trial first-root bound, and proves the endpoint
	amplitude bounds in the low-frequency and bounded-away-from-zero frequency
	regimes.
	
	\noindent\emph{Zero-energy critical solution for later estimates.}
	For the later finite-box Jost and endpoint estimates, we record the positive
	zero-energy Friedrichs solution of the shifted critical equation
	\[
	\left(-\partial_s^2-\frac14\sinh^{-2}s\right)y=0.
	\]
	It is
	\[
	y(s)=\sqrt{\tanh(s/2)}
	F\!\left(\frac12,\frac12;1;\tanh^2(s/2)\right).
	\]
	Near the critical endpoint,
	\[
	y(s)=2^{-1/2}s^{1/2}+O(s^{5/2}),
	\]
	with no \(s^{1/2}\log s\) component.  At infinity, the logarithmic case
	\(c=a+b\) in the DLMF connection formula \cite[15.8.10]{DLMF}, with
	\(a=b=1/2\), gives
	\[
	y(s)=\pi^{-1}s+O(1).
	\]
	This record is not used as the first-eigenvalue branch-gap trial above; it is
	used only in the finite-box and endpoint normalization estimates below.
	
	\begin{proof}
		This lemma is proved before the projected Green estimate below; it uses only
		the regular Sturm--Liouville construction, Bessel comparison, and Volterra
		estimates for the finite-box radial equation.  No later projected endpoint
		estimate is used here.  We separate the standard regular-solution
		construction from the two local normalization calculations.  For each fixed
		\(l\le N\), subtract the common
		constant \(a_n^2\) and write the finite-box equation in the form
		\[
		\bigl(-\partial_s^2+V_l(s)\bigr)u=\rho^2u,\qquad
		u(s)\sim s^{\nu_l+1/2}\quad(s\downarrow0).
		\]
		Here
		\[
		V_l(s)=\bigl(l(l+n-2)+c_n^0\bigr)\sinh^{-2}s.
		\]
		The choice order in this lemma is: fix \(N\), choose \(S_N\), then choose
		\(\rho_0\), and finally choose \(R_N\).  On the regular tail
		\(s\ge S_N\), the potentials \(V_l\) are exponentially
		decaying, uniformly over the finite set \(l\le N\).  The zero-energy critical
		regular solution is nonconstant, so on the regular tail the zeros of its
		derivative are isolated.  Choose \(S_N\) away from those zeros and large
		enough that the Volterra operator built from the free sine kernel is a small
		contraction, for example by making
		\[
		\sup_{l\le N}\int_{S_N}^\infty (1+t)|V_l(t)|\,dt
		\]
		sufficiently small.  The bound
		\[
		|\rho^{-1}\sin\rho(y-x)|\le y-x
		\]
		also covers the limiting case \(\rho=0\), so \(S_N\) is fixed before any
		positive-frequency threshold.  After \(S_N\) is fixed, choose
		\(\rho_0>0\) by continuity of the Cauchy data on \([0,S_N]\), and then
		choose \(R_N\) large enough for the phase averages on \([S_N,R]\).  The
		Volterra construction and spectral-parameter continuity
		are the standard Sturm--Liouville construction
		\cite[Theorem 9.1]{Teschl}; the Bessel comparisons at the critical endpoint
		use the large- and small-argument estimates for \(J_0,Y_0\) recorded in
		\cite[Sections 10.2, 10.8, and 10.17]{DLMF}.
		
		For the first branch, the regular solution on the tail is a perturbation of a
		single sine mode with frequency \(\pi/R+O_N(R^{-2})\).  Its
		\(L^2(0,R)\)-normalization has amplitude
		\(\sqrt{2/R}+O_N(R^{-3/2})\).  More explicitly, if \(U_R\) is the
		unnormalized regular solution and \(A_R\) is its sine-tail amplitude on
		\([S_N,R]\), then \(A_R\asymp_N R\) and tail averaging gives
		\[
		\|U_R\|_{L^2(0,R)}^2
		=
		\frac12A_R^2R+O_N(A_R^2)+O_N(R^2)
		=
		\frac12A_R^2R\bigl(1+O_N(R^{-1})\bigr).
		\]
		Here the fixed endpoint interval contributes lower order mass, and the
		Volterra dressing on the exponentially small tail contributes the displayed
		lower-order error.  Thus the normalized tail amplitude is
		\(\sqrt{2/R}(1+O_N(R^{-1}))\).  Since the Dirichlet endpoint phase is
		\(\sin(\rho_{1,l,R}(R-s))\) up to the same Volterra dressing and
		\(\rho_{1,l,R}=\pi/R+O_N(R^{-2})\), differentiating at \(s=R\) gives
		\[
		\partial_s y_{l,R}(R)
		=
		-\sqrt{2/R}\,\pi/R+O_N(R^{-5/2}),
		\]
		uniformly for the finite set \(l\le N\).
		This includes the critical channel \((n,l)=(2,0)\): after \(S_N\) is fixed,
		the critical Friedrichs endpoint only determines the fixed Cauchy data at
		\(S_N\), while the long interval \([S_N,R]\) is governed by the same
		low-frequency Volterra tail and the same linear zero-frequency growth that
		sets the \(R^{3/2}\) normalization scale.
		
		It remains to record the critical-channel estimates.  In this paragraph
		\((n,l)=(2,0)\), \(V_0(s)=-\frac14\sinh^{-2}s\), and we use the
		\(\rho_k,\mu_k,\psi_k\) notation from the statement.  The lower root growth
		is a comparison
		with the exact critical Bessel operator
		\[
		B_R=-\frac{d^2}{ds^2}-\frac14s^{-2}
		\]
		with the same Friedrichs endpoint at \(0\) and Dirichlet endpoint at \(R\).
		Since \(\sinh s\ge s\), \(V_0(s)\ge -\frac14s^{-2}\), hence
		\[
		\rho_k^2\ge j_{0,k}^2R^{-2}.
		\]
		The first critical eigenvalue estimate proved in
		Lemma~\ref{lem:radial-branch-gap} gives
		\(\rho_1^2\le \pi^2R^{-2}\).  Since
		\(j_{0,k}=(k-\frac14)\pi+O(k^{-1})\), there are \(k_0\) and \(c>0\) such
		that \(j_{0,k}^2-\pi^2\ge ck^2\) for \(k\ge k_0\).  For the finitely many
		\(2\le k<k_0\), the strict inequality
		\(\pi^2<j_{0,2}^2\le j_{0,k}^2\) gives the same bound after decreasing \(c\).
		Therefore, after enlarging \(R_N\),
		\[
		\mu_kR^2=\rho_k^2R^2-\rho_1^2R^2
		\ge j_{0,k}^2-\pi^2\ge ck^2,\qquad k\ge2.
		\]
		The separate lower root bound follows from \(j_{0,k}\ge ck\):
		\[
		\rho_k\ge c\,k/R,\qquad k\ge2 .
		\]
		The upper root growth is the min-max half of the argument.  On
		\((S_N,R)\), take the \(k\)-dimensional space spanned by the first \(k\)
		Dirichlet sine modes, vanishing at \(S_N\) and \(R\), and extend it by zero
		to \((0,R)\).  These functions lie in the Friedrichs form domain, and on this
		space
		\[
		\frac{\int |\phi'|^2+V_0|\phi|^2}{\int |\phi|^2}
		\le C_N k^2R^{-2}
		\]
		for \(R\ge R_N\), because \(S_N\) is fixed and \(V_0\le0\) on the tail.
		Thus \(\rho_k\le Ck/R\).
		
		The endpoint amplitudes use two normalizations of the same Friedrichs-regular
		solution.  If \(\rho_k<\rho_0/2\), the bounded-frequency Volterra expansion
		and the zero-energy critical solution give
		\[
		|u_0(s,\rho_k)|\le Cs^{1/2},\qquad 0<s\le1 .
		\]
		The normalization needs the small-frequency Cauchy amplitude at the fixed
		matching point \(S_N\).  Let
		\[
		\mathcal A_0(\rho)^2
		=
		u_0(S_N,\rho)^2+\rho^{-2}\partial_su_0(S_N,\rho)^2 .
		\]
		As \(\rho\downarrow0\), the regular solution converges in Cauchy data on
		\([0,S_N]\) to the zero-energy Friedrichs solution.  The matching point
		\(S_N\) is chosen on the tail where this zero-energy solution has nonzero
		slope.  Hence, after decreasing \(\rho_0\) if needed,
		\[
		\rho^2\mathcal A_0(\rho)^2
		\longrightarrow |\partial_su_0(S_N,0)|^2>0,
		\qquad
		\mathcal A_0(\rho)\ge c_N\rho^{-1},
		\qquad 0<\rho<\rho_0/2 .
		\]
		Equivalently, after using the Dirichlet condition to write the tail phase
		from \(R\), the sine tail has amplitude comparable to \(\rho_k^{-1}\).  Write
		\(\zeta_k=\rho_kR\).  The lower root bound gives
		\(\zeta_k\ge c k\), hence \(\zeta_k\ge c_*>0\) for \(k\ge2\).  If
		\(\zeta_k\) is bounded, the rescaled sine-tail integral has a positive lower
		bound by compactness on \(\zeta_k\in[c_*,K]\); if \(\zeta_k>K\), the same
		lower bound follows from averaging over many half-periods.  Equivalently, the
		tail Volterra solution is a perturbed sinusoid of amplitude comparable to
		\(\mathcal A_0(\rho_k)\), and the same reverse-triangle estimate used below
		absorbs the relative sup-norm Volterra error into the principal sinusoidal
		term.  The preceding compactness/phase-averaging argument gives
		\[
		\int_{S_N}^{R}|u_0(s,\rho_k)|^2\,ds
		\ge c(R-S_N)\mathcal A_0(\rho_k)^2
		\ge c\rho_k^{-2}R .
		\]
		Enlarging \(K\) and then \(R_N\) gives, uniformly over all low roots,
		\[
		\|u_0(\cdot,\rho_k)\|_{L^2(0,R)}\ge c\rho_k^{-1}R^{1/2}.
		\]
		Therefore the normalized eigenfunction satisfies
		\[
		|\psi_k(s)|\le C\rho_kR^{-1/2}s^{1/2}
		\le CkR^{-3/2}s^{1/2}.
		\]
		For \(\rho_k\ge\rho_0/2\), write
		\[
		V_0(s)=-\frac14s^{-2}+W(s),
		\]
		where \(W\) is bounded on \((0,S_N]\).  We first record the a priori
		critical Bessel-Volterra bound, before using it in the Cauchy-data lower
		estimate.  Let \(u_B(s,\rho)=s^{1/2}J_0(\rho s)\), and write the regular
		solution on \((0,S_N]\) as
		\[
		u_0(s,\rho)=u_B(s,\rho)+\int_0^sK_\rho(s,t)W(t)u_0(t,\rho)\,dt,
		\]
		where \(K_\rho\) is the Green kernel built from
		\(s^{1/2}J_0(\rho s)\) and \(s^{1/2}Y_0(\rho s)\).  The small- and
		large-argument Bessel bounds in DLMF Sections 10.8 and 10.17, together with
		the Wronskian normalization in DLMF Section 10.5, give
		\[
		|K_\rho(s,t)|
		\le C_N
		\bigl(s^{1/2}\wedge\rho^{-1/2}\bigr)
		\bigl(t^{1/2}\wedge\rho^{-1/2}\bigr)
		\bigl(1+\mathbf 1_{\{\rho t\le1\}}|\log(\rho t)|\bigr),
		\qquad 0<t\le s\le S_N .
		\]
		With the weighted norm
		\[
		\|u\|_\rho
		=
		\sup_{0<s\le S_N}
		{|u(s)|\over s^{1/2}\wedge\rho^{-1/2}},
		\]
		and \(m_\rho(t)=t^{1/2}\wedge\rho^{-1/2}\), the kernel estimate gives
		\[
		\int_0^{S_N}
		m_\rho(t)^2
		\bigl(1+\mathbf 1_{\{\rho t\le1\}}|\log(\rho t)|\bigr)\,dt
		\le C(S_N,\rho_0),\qquad \rho\ge\rho_0/2.
		\]
		Indeed, the part \(0<t\le\rho^{-1}\) is \(O(\rho^{-2})\) after the change of
		variables \(y=\rho t\), while the remaining part is bounded by
		\(S_N\rho^{-1}\).  The Volterra operator is triangular, so its iterated
		norms are bounded by the usual factorial estimate with constants depending
		only on \(S_N\), \(\rho_0\), and \(\|W\|_\infty\).  Picard iteration gives,
		uniformly in \(\rho\ge\rho_0/2\),
		\[
		|u_0(s,\rho)|\le
		\begin{cases}
			Cs^{1/2}, & \rho s\le1,\\
			C\rho^{-1/2}, & \rho s\ge1 .
		\end{cases}
		\]
		This a priori bound for the unnormalized regular solution is independent of
		the normalization estimate below.
		We next make the Cauchy-data lower bound explicit.  At the fixed matching
		point \(S_N\), set
		\[
		\mathcal A(\rho)^2
		=
		u_0(S_N,\rho)^2+\rho^{-2}\partial_su_0(S_N,\rho)^2 .
		\]
		For the model solution \(u_B(s,\rho)=s^{1/2}J_0(\rho s)\),
		\[
		\rho^{-1}\partial_su_B(S_N,\rho)
		=
		\frac{1}{2\rho\sqrt{S_N}}J_0(\rho S_N)
		-\sqrt{S_N}\,J_1(\rho S_N).
		\]
		The large-argument expansions of \(J_0\) and \(J_1\), with
		\(x=\rho S_N\) and \(S_N\) fixed, give
		\[
		\begin{aligned}
			S_NJ_0(x)^2
			&+
			\left(
			\frac{1}{2\rho\sqrt{S_N}}J_0(x)
			-\sqrt{S_N}\,J_1(x)
			\right)^2  \\
			&=\frac{2}{\pi\rho}+O_N(\rho^{-2}).
		\end{aligned}
		\]
		Indeed \(J_0(x)=\sqrt{2/(\pi x)}\cos(x-\pi/4)+O(x^{-3/2})\) and
		\(J_1(x)=\sqrt{2/(\pi x)}\sin(x-\pi/4)+O(x^{-3/2})\), so the two leading
		Cauchy components are phase-shifted and their squares add to
		\(2/(\pi\rho)\); in particular they cannot vanish simultaneously.  The
		Bessel Volterra equation for the bounded residual \(W\) perturbs this scaled
		Cauchy vector by only
		\(O_N(\rho^{-3/2})\): splitting the integral at \(t=\rho^{-1}\), the collar
		\(0<t<\rho^{-1}\) uses
		\[
		|K_\rho(S_N,t)|+\rho^{-1}|\partial_sK_\rho(S_N,t)|
		\le C_N\rho^{-1/2}t^{1/2}(1+|\log(\rho t)|),
		\]
		so the logarithmic small-argument \(Y_0\) bound and the regular estimate
		\(u_0(t,\rho)=O(t^{1/2})\) give an \(O_N(\rho^{-5/2})\) contribution,
		while \(\rho^{-1}\le t\le S_N\) contributes \(O_N(\rho^{-3/2})\) by the
		large-argument Bessel bounds and the compact-collar high-window bound
		\(u_0(t,\rho)=O_N(\rho^{-1/2})\).  Therefore
		\[
		\mathcal A(\rho)\ge c\rho^{-1/2}
		\]
		for all sufficiently large \(\rho\).  On the remaining compact interval
		\(\rho_0/2\le\rho\le\rho_1\), the same lower bound follows after decreasing
		\(c\), because the regular solution depends continuously on \(\rho\) and
		vanishing of both Cauchy data at \(S_N\) would force the solution to be
		identically zero.
		
		The free-region Volterra construction on the fixed tail transfers this
		amplitude to
		\[
		u_0(s,\rho)
		=
		\mathcal A(\rho)\cos(\rho(s-S_N)-\theta(\rho))+E(s,\rho),
		\qquad
		\sup_{S_N\le s\le R}|E(s,\rho)|
		\le\epsilon\,\mathcal A(\rho),
		\]
		where \(\epsilon\in(0,1/8)\) is fixed by the original choice of \(S_N\);
		the tail Volterra operator built from the free sine kernel is an
		\(\epsilon\)-contraction in the sup norm, uniformly for
		\(\rho\ge\rho_0/2\).  Hence
		\[
		\int_{S_N}^{R}|E(s,\rho)|^2\,ds
		\le \epsilon^2\mathcal A(\rho)^2(R-S_N).
		\]
		The reverse triangle inequality in \(L^2(S_N,R)\) gives
		\[
		\begin{aligned}
			\|u_0(\cdot,\rho)\|_{L^2(S_N,R)}
			&\ge
			\mathcal A(\rho)
			\left(\int_{S_N}^{R}
			\cos^2(\rho(s-S_N)-\theta(\rho))\,ds
			\right)^{1/2}  \\
			&\qquad
			-\left(\int_{S_N}^{R}|E(s,\rho)|^2\,ds\right)^{1/2}.
		\end{aligned}
		\]
		For \(\rho=\rho_k\ge\rho_0/2\) and \(R\ge R_N\),
		\[
		\int_{S_N}^R
		\cos^2(\rho_k(s-S_N)-\theta(\rho_k))\,ds
		\ge \frac{R-S_N}{2}-\frac{1}{2\rho_k}
		\ge \frac{R}{8},
		\]
		after enlarging \(R_N\) so that \(R-S_N\ge R/2\) and
		\((2\rho_k)^{-1}\le R/8\).  Since \(R-S_N\le R\), the preceding reverse
		triangle inequality gives
		\[
		\|u_0(\cdot,\rho_k)\|_{L^2(S_N,R)}
		\ge
		\mathcal A(\rho_k)R^{1/2}
		\left(\frac{1}{\sqrt8}-\epsilon\right)
		\ge c\mathcal A(\rho_k)R^{1/2}.
		\]
		Here \(c>0\) because \(\epsilon<1/8<1/\sqrt8\).
		Since \(\mathcal A(\rho_k)^2\ge c/\rho_k\), this gives
		\[
		\|u_0(\cdot,\rho_k)\|_{L^2(0,R)}^2\ge cR/\rho_k .
		\]
		Normalizing and using the upper root bound \(\rho_k\le Ck/R\) gives
		\[
		|\psi_k(s)|\le C\rho_k^{1/2}R^{-1/2}s^{1/2}
		\le Ck^{1/2}R^{-1}s^{1/2}
		\quad(\rho_ks\le1),
		\]
		while the large-argument Bessel estimate gives
		\[
		|\psi_k(s)|\le CR^{-1/2}\quad(\rho_ks\ge1).
		\]
		Taking the maximum of the finitely many thresholds and the minimum of the
		positive constants gives the stated uniform constants.
	\end{proof}
	
	\paragraph{Finite-\(N\) endpoint estimates.}
	\begin{lemma}[Finite-\(N\) projected Green and Jost estimates]
		\label{lem:critical-endpoint-jost-green}
		Fix \(1/2<\delta_{\rm cut}<2/3\) and the angular cutoff \(N\), and set
		\(\eps_R=R^{-\delta_{\rm cut}}\).  For every \(0\le l\le N\), the projected
		radial complement of \(H_{l,R}\) satisfies the endpoint collar estimates
		below, with constants and \(o_R(1)\) rates allowed to depend on the fixed
		cutoff \(N\) and on \(\delta_{\rm cut}\).  All radial functions in this lemma
		are in the flattened \(L^2(0,R;ds)\) normalization, and all projections are
		orthogonal for that inner product.  It is enough to display the
		critical radial channel.  In the noncritical channels the inverse-square
		coefficient is strictly above the Friedrichs critical value, the endpoint
		powers are stronger, and the radial branch gap from
		Lemma~\ref{lem:radial-branch-gap}, established independently before this
		projected Green estimate, supplies the needed complement coercivity.  In the
		critical channel \((n,l)=(2,0)\), let
		\(\lambda_{k,0,R}\) and \(\psi_k\) denote the
		eigenvalues and normalized eigenfunctions of \(H_{0,R}\), put
		\[
		\mu_k=\lambda_{k,0,R}-\lambda_{1,0,R},
		\]
		and write \(\rho_k\) for the corresponding finite-box roots.  There are
		constants \(\rho_0>0\), \(R_N<\infty\), and \(0<c_0<C_0<\infty\), depending only
		on the fixed cutoff \(N\), such that for all \(R\ge R_N\)
		\[
		\mu_k\ge c_0k^2R^{-2},\qquad
		c_0\,\frac{k}{R}\le \rho_k\le C_0\,\frac{k}{R},\qquad k\ge2.
		\]
		Moreover the endpoint amplitudes satisfy
		\[
		\begin{array}{ll}
			|\psi_k(s)|\le CkR^{-3/2}s^{1/2},
			& \rho_k<\rho_0/2,\\[3pt]
			|\psi_k(s)|\le Ck^{1/2}R^{-1}s^{1/2},
			& \rho_k\ge\rho_0/2,\ \rho_k s\le1,\\[3pt]
			|\psi_k(s)|\le CR^{-1/2},
			& \rho_k s\ge1,
		\end{array}
		\qquad 0<s\le1,
		\]
		and the projected Green diagonal
		\[
		G_R(s,s)=\sum_{k\ge2}\frac{|\psi_k(s)|^2}{\mu_k}
		\]
		satisfies
		\[
		\sup_{\eps_R\le s\le1}G_R(s,s)\le C .
		\]
		If \(Q_{0,R}\) is the projection away from the first critical radial branch
		and
		\[
		k[q]=
		\langle q,Q_{0,R}(H_{0,R}-\lambda_{1,0,R})Q_{0,R}q\rangle ,
		\]
		then every \(q=Q_{0,R}q\) in the projected critical radial complement
		satisfies
		\[
		R^{-1}\int_{\eps_R}^{1}s^{-2}|q(s)|^2\,ds
		\le o_R(1)k[q],
		\qquad
		R^{-2}\int_{\eps_R}^{1}|D_sq|^2\,ds
		\le o_R(1)k[q],
		\]
		where the \(o_R(1)\) rate may depend on the fixed collar exponent
		\(\delta_{\rm cut}\) and on the fixed cutoff \(N\).  The smallness comes
		from \(\delta_{\rm cut}<1\), in particular from
		\(R^{-1}\eps_R^{-1}=R^{\delta_{\rm cut}-1}\to0\).
	\end{lemma}
	
	\begin{proof}
		We begin with the finite-box input, then derive the projected Green estimates
		and the endpoint collar bounds.
		
		\emph{Finite-box Jost supply.}  In the critical channel write
		\[
		\rho_k=\rho_{k,0,R},\qquad
		\psi_k=\psi_{k,0,R},\qquad
		\mu_k=\lambda_{k,0,R}-\lambda_{1,0,R}.
		\]
		Here \(\lambda_{k,l,R}\) and \(\psi_{k,l,R}\) denote the eigenvalues and
		normalized eigenfunctions of the separated radial operator \(H_{l,R}\) on
		\((0,R)\), so that \(\lambda_{1,l,R}=e_{l,R}\) and
		\(\psi_{1,l,R}=y_{l,R}\).  Let \(u_l(s,\rho)\) be the Friedrichs-regular
		solution of
		\[
		\bigl(-\partial_s^2+V_l(s)\bigr)u=\rho^2u,\qquad
		u_l(s,\rho)\sim s^{\nu_l+1/2}\quad (s\downarrow0),
		\]
		and write \(u_0\) for the critical channel.
		After subtracting the common constant \(a_n^2\), the finite-box equation has
		potential
		\[
		V_l(s)=\bigl(l(l+n-2)+c_n^0\bigr)\sinh^{-2}s ;
		\]
		in the critical channel \((n,l)=(2,0)\), this is
		\(V_0(s)=-\frac14\sinh^{-2}s\).
		For each fixed \(l\le N\) the same construction is run with \(V_l\) in place
		of \(V_0\).  Since only finitely many degrees \(l\le N\) occur, the common
		constants below are obtained by taking the maximum of the resulting
		thresholds and the minimum of the resulting positive lower constants.  The
		displayed argument is written in the critical channel, where the logarithmic
		endpoint companion is the binding case.
		Lemma~\ref{lem:finite-box-jost-data} supplies constants
		\(\rho_0>0\), \(R_N<\infty\), and \(0<c_0<C_0<\infty\), depending only on the
		fixed angular cutoff, such that for \(R\ge R_N\) and \(k\ge2\),
		\[
		\mu_k\ge c_0k^2R^{-2},\qquad
		c_0\,\frac{k}{R}\le \rho_k\le C_0\,\frac{k}{R},
		\]
		together with the three endpoint amplitude regimes
		\[
		\begin{array}{ll}
			|\psi_k(s)|\le CkR^{-3/2}s^{1/2},
			& \rho_k<\rho_0/2,\\[3pt]
			|\psi_k(s)|\le Ck^{1/2}R^{-1}s^{1/2},
			& \rho_k\ge\rho_0/2,\ \rho_k s\le1,\\[3pt]
			|\psi_k(s)|\le CR^{-1/2},
			& \rho_k s\ge1,
		\end{array}
		\qquad 0<s\le1 .
		\]
		These finite-box estimates are the only Jost input used in the rest of this
		lemma.  The choice order \(S_N\to\rho_0\to R_N\), the high-window Cauchy-data
		normalization, and the first-branch boundary slope have all been packaged in
		Lemma~\ref{lem:finite-box-jost-data}; the projected Green argument below
		uses only the displayed root-growth and endpoint-amplitude consequences.
		
		\emph{Projected Green diagonal.}  If \(Q_{0,R}\) denotes projection away from
		the first critical radial branch, set
		\[
		k[q]=\langle q,Q_{0,R}(H_{0,R}-\lambda_{1,0,R})Q_{0,R}q\rangle .
		\]
		The projection is essential: without removing the first critical branch, the
		right side could vanish while the endpoint mass is positive.  For
		\[
		G_R(s,s)=\sum_{k\ge2}\frac{|\psi_k(s)|^2}{\mu_k},
		\]
		Cauchy--Schwarz in the projected spectral expansion gives
		\(|q(s)|^2\le k[q]G_R(s,s)\): indeed, if
		\(q=\sum_{k\ge2}a_k\psi_k\), then
		\[
		|q(s)|^2
		\le \left(\sum_{k\ge2}\mu_k|a_k|^2\right)
		\left(\sum_{k\ge2}\frac{|\psi_k(s)|^2}{\mu_k}\right)
		=k[q]G_R(s,s).
		\]
		The diagonal bound
		\(\sup_{\eps_R\le s\le1}G_R(s,s)\le C\) is proved by splitting the sum
		according to the global frequency threshold \(\rho_0/2\) and the local collar
		product \(\rho_k s\):
		\[
		G_R(s,s)=\Sigma_{\rm low}(s)+\Sigma_{\rm seam}(s)+\Sigma_{\rm high}(s),
		\qquad \eps_R\le s\le1 .
		\]
		The index enlargements used below are
		\[
		\begin{array}{c|c}
			\text{regime} & \text{consequence of the root bounds} \\ \hline
			\rho_k<\rho_0/2 & k\le C R,\\
			\rho_k\ge\rho_0/2,\ \rho_k s\le1 & cR\le k\le C R/s,\\
			\rho_k s\ge1 & k\ge c R/s .
		\end{array}
		\]
		Let \(k_\ast\) be the first index with \(\rho_{k_\ast}\ge\rho_0/2\).  The
		root bounds give \(c_0R\le k_\ast\le K_{\rm root}R\), and the upper root-growth bound
		\(\rho_k\le Ck/R\) is used only for \(k\ge k_\ast\).  For genuinely low
		roots \(\rho_k<\rho_0/2\), the lower root-growth bound
		\(\rho_k\ge ck/R\) gives \(k\le \rho_0R/(2c)\), hence \(O(R)\) such indices,
		and
		\[
		\Sigma_{\rm low}(s)
		\le \sum_{\rho_k<\rho_0/2}
		\frac{Ck^2R^{-3}s}{ck^2R^{-2}}
		\le \sum_{\rho_k<\rho_0/2}CR^{-1}s
		\le Cs .
		\]
		For seam roots \(\rho_k\ge\rho_0/2\) and \(\rho_k s\le1\), the lower
		root-growth bound gives \(k\le C_1R/s\).  The lower endpoint
		\(k\ge c_0R\) follows from \(\rho_k\ge\rho_0/2\) and the upper root-growth
		bound \(\rho_k\le Ck/R\).  Thus the actual seam index set is contained in
		\(c_0R\le k\le C_1R/s\).  The following sum is therefore over a superset of
		the actual seam indices; indices in this range with \(\rho_k s>1\) are not in
		\(\Sigma_{\rm seam}\) and are handled by \(\Sigma_{\rm high}\).  Hence
		\[
		\Sigma_{\rm seam}(s)
		\le \sum_{k_\ast\le k\le C_1R/s}
		\frac{CkR^{-2}s}{ck^2R^{-2}}
		\le C_2s\sum_{c_0R\le k\le C_1R/s}\frac1k
		\le C_3s\bigl(1+\log(C_4/s)\bigr)\le C_5 .
		\]
		The second inequality only enlarges the index range, using
		\(k_\ast\ge c_0R\).  The last inequality uses the elementary bound
		\(\sup_{0<s\le1}s\log(1/s)=e^{-1}\).
		For high local-Bessel roots \(\rho_k s\ge1\), the upper root-growth bound
		\(\rho_k\le Ck/R\) gives \(k\ge cR/s\), so the actual high index set is
		contained in that tail.  Therefore
		\[
		\Sigma_{\rm high}(s)
		\le \sum_{k\ge cR/s}
		\frac{CR^{-1}}{ck^2R^{-2}}
		\le CR\sum_{k\ge cR/s}\frac1{k^2}
		\le Cs .
		\]
		In fact these three estimates give the sharper bound
		\[
		G_R(s,s)\le Cs\bigl(1+\log(C/s)\bigr),\qquad 0<s\le1,
		\]
		and in particular the projected diagonal is uniformly bounded on
		\([\eps_R,1]\).  The point of the genuinely low/seam/high local-Bessel split
		is to keep distinct the global frequency dichotomy controlling box
		normalization from the local Bessel dichotomy controlling endpoint size.
		On the transition layer \(I_R=[\eps_R,2\eps_R]\), the same bound gives
		\[
		\int_{I_R}|q|^2\,ds
		\le C\eps_R^2\log(C/\eps_R)\,k[q]
		\]
		for every projected vector \(q\).  This is the estimate used when
		\(G_s'=O(R^{-1}\eps_R^{-2})\) is routed in the transition coefficient class.
		
		\emph{Endpoint form controls.}  For the collar scale
		\(\eps_R=R^{-\delta_{\rm cut}}\), \(1/2<\delta_{\rm cut}<2/3\), the
		cruder uniform diagonal bound gives the weighted collar mass estimate
		\[
		R^{-1}\int_{\eps_R}^{1}s^{-2}|q(s)|^2\,ds
		\le C R^{-1}\eps_R^{-1}k[q]
		=O(R^{\delta_{\rm cut}-1})k[q]
		=o_R(1)k[q].
		\]
		Here the last equality uses \(\delta_{\rm cut}<1\).
		The derivative estimate is not obtained by differentiating a singular
		projected Green kernel.  On the endpoint collar write \(q=yv\), where \(y\)
		is the positive zero-energy critical solution for the shifted operator
		\(\widehat H=H_{0,R}-a_2^2\); equivalently \(H_{0,R}y=a_2^2y\).  The local
		ground-state identity needed here is
		\[
		\langle yv,(H_{0,R}-\lambda_{1,0,R})yv\rangle
		=
		\int y^2|v'|^2
		-(\lambda_{1,0,R}-a_2^2)\|yv\|^2 .
		\]
		For \(v\in C_c^\infty((0,R))\), expand \(q=yv\):
		\[
		|q'|^2+V_0q^2
		=
		y^2|v'|^2+y'^2v^2+2yy'vv'+V_0y^2v^2 .
		\]
		Integrating \(2yy'vv'=yy'(v^2)'\) by parts gives the bulk contribution
		\(-(yy')'v^2=-(y'^2+yy'')v^2\).  Since
		\(-y''+V_0y=a_2^2y\), the remaining zero-order terms collapse to
		\(a_2^2y^2v^2\), and no endpoint contribution appears because \(v\) is
		compactly supported.  Subtracting \(\lambda_{1,0,R}\|yv\|^2\) gives the
		displayed identity.  For finite spectral sums and then for a general
		projected complement vector, this identity is closed in the critical
		Friedrichs form sense, not by evaluating a classical endpoint trace.  Insert
		the standard logarithmic cutoff at the critical endpoint \(s=0\), use the
		Dirichlet trace at \(R\), and let the cutoff tend to one.
		The cutoff cost tends to zero because \(y^2\sim s\) and the Friedrichs
		endpoint excludes the \(s^{1/2}\log s\) branch.  Thus no singular-endpoint
		boundary term is produced in the closed-form identity.  Keeping the closure
		inside the projected spectral subspace, write
		\[
		q=\sum_{k\ge2}a_k\psi_k,\qquad
		q^{(M)}=\sum_{2\le k\le M}a_k\psi_k .
		\]
		The finite sums \(q^{(M)}\) lie in the Friedrichs operator domain and in
		\(\operatorname{Ran}Q_{0,R}\).  The compact-core identity, closed in the
		operator graph norm on these finite sums, gives
		\[
		\int y^2\left|\left(q^{(M)}/y\right)'\right|^2
		=
		k[q^{(M)}]+(\lambda_{1,0,R}-a_2^2)\|q^{(M)}\|^2 .
		\]
		Since \(q^{(M)}=Q_{0,R}q^{(M)}\) is projected off the first critical branch,
		the branch gap
		\(\lambda_{2,0,R}-\lambda_{1,0,R}\ge c_{\rm rad}^0R^{-2}\) from
		Lemma~\ref{lem:radial-branch-gap} gives the same inequality with the smaller
		post-reservation constant \(c_{\rm rad}\le c_{\rm rad}^0\).  Thus, through
		the spectral expansion,
		\[
		k[q^{(M)}]\ge c_{\rm rad}R^{-2}\|q^{(M)}\|^2,
		\qquad
		\|q^{(M)}\|^2\le c_{\rm rad}^{-1}R^2k[q^{(M)}].
		\]
		Together with \(k[q^{(M)}]\to k[q]\), \(q^{(M)}\to q\) in \(L^2\), and the
		critical first-eigenvalue estimate
		\(\lambda_{1,0,R}-a_2^2=O(R^{-2})\), this makes the right side bounded by
		\(Ck[q]\) and Cauchy.  This defines the closed
		\(h\)-transform derivative for \(q=yv\) and yields
		\[
		\int y^2|v'|^2\le Ck[q].
		\]
		Thus no pointwise singular endpoint trace is evaluated for the limiting
		function, and the closure uses the projected form norm for which \(k[\cdot]\)
		is the exact energy.  The shifted critical endpoint expansion gives
		\(y(s)=c_0s^{1/2}+O(s^{5/2})\) and
		\(|y'/y|\le Cs^{-1}\), while the conjugated radial derivative \(D_s\) has its
		Jacobian/log-derivative coefficient absorbed into the same \(s^{-2}|q|^2\)
		term.  Hence
		\[
		|D_sq|^2\le C y^2|v'|^2+C s^{-2}|q|^2
		\]
		on the fixed collar, and therefore
		\[
		R^{-2}\int_{\eps_R}^{1}|D_sq|^2\,ds
		\le CR^{-2}k[q]
		+CR^{-1}
		\left(R^{-1}\int_{\eps_R}^{1}s^{-2}|q|^2\,ds\right)
		=o_R(1)k[q].
		\]
		On the regular tail \(s\ge1\), the coefficient-error terms carry an
		exponentially decaying factor, bounded by \(Ce^{-2s}(1+s)^2\).  The same
		projected Jost split, now in the bounded-potential region, gives the
		corresponding exponentially weighted tail estimates beyond the endpoint
		collar.  Thus this local estimate uses a projected Green diagonal and a closed
		\(h\)-transform, not a critical Hardy inequality.
	\end{proof}
	
	\begin{lemma}[Finite first-branch Hadamard boundary-flux matrix]
		\label{lem:finite-first-branch-hadamard}
		Fix the dimension, the height and Lipschitz bounds \(C,L<\infty\), the
		collar exponent \(\delta_{\rm cut}\in(1/2,2/3)\), and the angular cutoff
		\(N\).  Set \(\eps_R=R^{-\delta_{\rm cut}}\).  For every smooth \(h\) with
		\(0\le h\le C\) and \(\Lip(h)\le L\), use the cutoff flattening field
		\(F_h(s,\theta)=s-\beta_R(s)h(\theta)\), where \(\beta_R=0\) near the
		origin, \(\beta_R=s/R\) on the regular bulk \(s\ge2\eps_R\),
		\(\beta_R(R)=1\), and
		\(\beta_R'=O(R^{-1})\), \(\beta_R''=O(R^{-1}\eps_R^{-1})\) on the
		transition layer.  For \(l,l'\le N\), let \(M_{ll',R}\) denote the radial
		coefficient in the polarized first variation of the shifted physical
		Dirichlet form on the finite first branches \(u_{l,R}\) and \(u_{l',R}\),
		after separating the angular factor
		\(\int_{\mathbb S^{n-1}}hY_l\overline{Y_{l'}}\,d\theta\).  Thus
		\(M_{ll',R}\) is the radial-only scalar multiplying this angular integral;
		the angular integral is not included in \(M_{ll',R}\).  It records only the
		Rellich normal-derivative boundary flux contribution.  The skew-adjointness
		boundary formula for \(A_X\), whose boundary integral contains
		\(u\overline v\) rather than \(\partial_\nu u\,\partial_\nu v\), vanishes for
		Dirichlet eigenfunctions and does not contribute to \(M_{ll',R}\).
		Then, uniformly for \(l,l'\le N\) and uniformly over all such heights \(h\),
		\[
		M_{ll',R}
		=
		\partial_s y_{l,R}(R)\,\partial_s y_{l',R}(R)
		+o_{N,C,L,\delta_{\rm cut}}(R^{-3})
		=
		2\pi^2R^{-3}+o_{N,C,L,\delta_{\rm cut}}(R^{-3}).
		\]
		Consequently, if \(f_{\le}=\sum_{\alpha\le N}c_\alpha Y_\alpha\), then
		\[
		\sum_{\alpha,\beta\le N}c_\alpha\overline{c_\beta}
		\int_{\mathbb S^{n-1}}hY_\alpha\overline{Y_\beta}\,d\theta\,
		M_{\alpha\beta,R}
		=
		2\pi^2R^{-3}\langle f_{\le},hf_{\le}\rangle
		+o_{N,C,L,\delta_{\rm cut}}(R^{-3})\|f_{\le}\|^2 .
		\]
		Moreover, for the finite first-branch vector \(w_{\le}=J_Rf_{\le}\), the
		exact shifted pulled-back form satisfies
		\[
		Q_h[w_{\le}]-E_R\|w_{\le}\|^2
		=
		R^{-3}\langle f_{\le},
		P_N(T_n+2\pi^2h+b_n)P_Nf_{\le}\rangle
		+o_{N,C,L,\delta_{\rm cut}}(R^{-3})\|f_{\le}\|^2 .
		\]
	\end{lemma}
	
	\noindent\emph{Orientation for this lemma.}
	This lemma is the finite first-branch coefficient calculation.  Its statement
	says that, after flattening the graph boundary and restricting to the fixed
	angular degrees \(l,l'\le N\), the only order-\(R^{-3}\) height-dependent
	matrix is the Rellich normal-derivative boundary flux, whose radial scalar is
	\(2\pi^2R^{-3}+o(R^{-3})\).  The proof is organized into named stages:
	Hadamard identity and sign convention; finite-height reduction; boundary-flux
	coefficient; spectral-defect commutator and cutoff flow; endpoint, transition,
	and shear remainders; and assembly of the finite first-branch block.
	To verify this, the finite-height issue is isolated in Stage 2: it is exactly
	the uniform second-\(t\)-derivative bound for the pulled-back form along
	\(F_t(s,\theta)=s-t\beta_R(s)h(\theta)\).  Once that bound is in place, the
	finite graph displacement has the same leading coefficient as the infinitesimal
	Hadamard variation, and the rest of the lemma only routes the named remainder
	classes.
	
	\begin{proof}
		\emph{Stage 1: Hadamard identity and sign convention.}
		The matrix is computed for the physical hyperbolic Dirichlet form before
		comparing the separated \(l\)-rows.  Let \(L=-\Delta_{\mathbb H^n}\) on
		\(B_R\), and let \(u_{l,R}\) and \(u_{l',R}\) be exact ball eigenfunctions with
		eigenvalues \(e_{l,R}\) and \(e_{l',R}\).  Introduce the path
		\[
		F_t(s,\theta)=s-t\beta_R(s)h(\theta),\qquad 0\le t\le1,
		\]
		and let \(X_h=-\beta_Rh\,\partial_s\) be its velocity field at \(t=0\).  On
		\(\partial B_R\), \(X_h\cdot\nu=-h\).  Let \(B_h^{E_R}\) denote the
		derivative at \(t=0\) of the \(L^2\)-unitarily pulled-back shifted physical
		form \(q_{E_R}=q-E_R\|\cdot\|^2\).  We use the local-height sign convention:
		positive \(h\) is inward motion of the boundary.  Thus the outward normal
		velocity is \(V=X_h\cdot\nu=-h\), and the Dirichlet Hadamard sign is
		\(-V\).  In polarized form, for Dirichlet eigenfunctions \(u\) and \(v\)
		with eigenvalues \(e_u\) and \(e_v\), the exact Hadamard--Rellich identity is
		\[
		B_h^{E_R}(u,v)
		=
		-\int_{\partial B_R}(X_h\cdot\nu)\,
		\partial_\nu u\,\partial_\nu v\,d\sigma_R
		+\mathcal C_X^{E_R}(u,v).
		\]
		Here
		\[
		A_X=-X\cdot\nabla-\frac12\operatorname{div}_gX,
		\]
		denotes the skew-adjoint operator associated with the unitary pullback.  With this
		orientation of the unitary,
		\[
		\mathcal C_X^{E_R}(u,v)
		=
		(e_u-E_R)\langle A_Xu,v\rangle
		+(e_v-E_R)\langle u,A_Xv\rangle
		=
		(e_u-e_v)\langle A_Xu,v\rangle .
		\]
		The inverse unitary convention changes only the sign of the final
		commutator.  This sign is immaterial below; the proof uses only the exact
		diagonal cancellation and the absolute value of the off-diagonal term.  The
		mass part has no boundary term because \(u=v=0\) on \(\partial B_R\).  The
		boundary part of the identity is therefore
		\[
		\int_{\partial B_R}h\,\partial_\nu u\,\partial_\nu v\,d\sigma_R,
		\]
		so an inward deformation gives the positive first variation expected from
		domain monotonicity.  The flux is evaluated on the unperturbed sphere
		\(s=R\), because it is the first-variation coefficient at \(t=0\); the
		remaining effects of moving to the graph boundary \(s=R-h(\theta)\) are
		absorbed into the finite-block remainders \(\Delta_{ll',R}\) controlled in
		Stage 5.  Here is
		the calculation behind the identity.  Before
		the half-density unitary is inserted, differentiating the pulled-back
		Dirichlet energy gives the integrand
		\[
		(\operatorname{div}X)\langle\nabla u,\nabla v\rangle
		-\langle\nabla_{\nabla u}X,\nabla v\rangle
		-\langle\nabla_{\nabla v}X,\nabla u\rangle
		-E_R(\operatorname{div}X)uv .
		\]
		Writing \(Xu=X\cdot\nabla u\), define the Rellich vector field
		\[
		\mathcal R_X(u,v)
		=
		\langle\nabla u,\nabla v\rangle X
		-(Xu)\nabla v-(Xv)\nabla u .
		\]
		Then
		\[
		\begin{aligned}
			\operatorname{div}\mathcal R_X(u,v)
			&=
			(\operatorname{div}X)\langle\nabla u,\nabla v\rangle
			-\langle\nabla_{\nabla u}X,\nabla v\rangle
			-\langle\nabla_{\nabla v}X,\nabla u\rangle  \\
			&\qquad
			-(Xu) \Delta_{\mathbb H^n}v-(Xv)\Delta_{\mathbb H^n}u .
		\end{aligned}
		\]
		Because \(\beta_R=0\) on a neighborhood of the origin, \(X_h\) and
		\(\mathcal R_{X_h}(u,v)\) vanish there.  Thus the divergence theorem produces
		no inner-boundary contribution from the singular polar endpoint.  On
		\(\partial B_R\), \(u=v=0\), their tangential derivatives vanish there, and
		\[
		\mathcal R_X(u,v)\cdot\nu
		=-(X\cdot\nu)\partial_\nu u\,\partial_\nu v .
		\]
		Integrating the divergence identity therefore gives the boundary term above.
		The two \(J_t^{1/2}\) factors in the \(L^2\)-unitary pullback add the
		corresponding \(-\frac12(\operatorname{div}X)\) contributions in the two
		slots.  Thus the remaining bulk terms are exactly the skew-adjoint
		commutator terms encoded by \(A_X=-X\cdot\nabla-\frac12\operatorname{div}X\).
		Pairing these terms with
		\((-\Delta_{\mathbb H^n}-E_R)u=(e_u-E_R)u\) and
		\((-\Delta_{\mathbb H^n}-E_R)v=(e_v-E_R)v\) gives the displayed
		\(\mathcal C_X^{E_R}(u,v)\), up to the harmless global sign determined by
		which of the two inverse unitary conventions is used.
		For the fixed finite first-branch block, the off-diagonal commutator part is
		lower order:
		\[
		(e_{l,R}-e_{l',R})\langle A_Xu_{l,R},u_{l',R}\rangle
		=O_{N,C,L}(R^{-4}),
		\]
		as verified below from the branch asymptotics and the \(O_N(R^{-1})\) size of
		the radial generator.  Thus the leading \(R^{-3}\) term in the first
		variation comes only from the Rellich boundary flux.
		
		\emph{Stage 2: finite-height reduction.}
		The displayed identity supplies the first derivative at the round ball.  To
		use it for the actual finite graph displacement, let
		\(\mathcal B_t^{E_R}\) be the shifted pulled-back finite-sector form obtained
		from the path \(F_t(s,\theta)=s-t\beta_R(s)h(\theta)\), evaluated on the
		fixed first-branch inputs.  Then
		\[
		\mathcal B_1^{E_R}-\mathcal B_0^{E_R}
		=
		\dot{\mathcal B}_0^{E_R}
		+\int_0^1(1-t)\ddot{\mathcal B}_t^{E_R}\,dt .
		\]
		The estimates below are uniform in \(t\), because \(0\le th\le C\) and
		\(\Lip(th)\le L\).  Each second \(t\)-derivative term contains either a
		second height/Jacobian factor, an extra \(R^{-1}\), a transition coefficient
		on \([\eps_R,2\eps_R]\), or a tangential correction carrying
		\(\sinh^{-2}(F_t)\).  More explicitly, after writing
		\(a_t=\partial_sF_t=1-t\beta_R'h\) and expanding the Schrodinger-flat
		pullback, \(\ddot{\mathcal B}_t^{E_R}\) is a finite sum of bilinear
		integrals of the following five types:
		\[
		\begin{array}{ll}
			\text{endpoint:} & R^{-1}s^{-2}\beta_Rh \text{ or its differentiated
				transition version},\\
			\text{transition:} & \beta_R'',\,(\beta_R')^2,\text{ or }G_s' \text{ on }
			[\eps_R,2\eps_R],\\
			\text{shear:} & \beta_R^2|\nabla h|^2\sinh^{-2}(F_t)D_sy\,D_sy',\\
			\text{Jacobian:} & (a_t-1)^2,\ \partial_s\log a_t,\text{ and angular
				half-density factors},\\
			\text{extra height:} & \beta_R^2h^2 \text{ multiplying a regular-bulk
				coefficient.}
		\end{array}
		\]
		These are exactly the term families later routed through the endpoint,
		transition, shear, Jacobian, and extra-height-factor estimates; no second
		derivative introduces a new order-\(R^{-3}\) coefficient family.  The endpoint,
		transition, and regular-bulk bounds displayed below therefore give, uniformly
		for \(l,l'\le N\),
		\[
		|\ddot{\mathcal B}_t^{E_R}(u_{l,R},u_{l',R})|
		\le C_{N,C,L}R^{-4}.
		\]
		Thus the finite-height displacement has the same leading coefficient as the
		round-sphere first variation, and the integral remainder is included in
		\(\Delta_{ll',R}\).
		
		Applying this to \(u_{l,R}\) and \(u_{l',R}\), the shifted spectral part is
		\[
		(e_{l,R}-E_R)\langle A_Xu_{l,R},u_{l',R}\rangle
		+(e_{l',R}-E_R)\langle u_{l,R},A_Xu_{l',R}\rangle .
		\]
		For Dirichlet eigenfunctions \(u\) and \(v\), integration by parts gives the
		skew-adjointness identity
		\[
		\langle A_Xu,v\rangle+\langle u,A_Xv\rangle
		=
		-\int_{\partial B_R}(X\cdot\nu)u\overline v\,d\sigma_R=0,
		\]
		because \(u=v=0\) on \(\partial B_R\), with no inner-boundary contribution
		since \(\beta_R=0\) near the origin.  Taking \(v=u\) gives the diagonal
		vanishing, and taking \(u=u_{l,R}\), \(v=u_{l',R}\) gives the off-diagonal
		skew relation on the finite first-branch span.  This skew-adjointness
		identity is distinct from the Hadamard--Rellich boundary flux identity, which
		involves the normal derivatives \(\partial_\nu u\,\partial_\nu v\).  The
		off-diagonal commutator term
		\[
		(e_u-E_R)\langle A_Xu,v\rangle
		+(e_v-E_R)\langle u,A_Xv\rangle
		\]
		therefore simplifies to
		\((e_u-e_v)\langle A_Xu,v\rangle\) by this skew relation, not by a boundary
		flux calculation.  The normal-derivative flux
		\(-(X\cdot\nu)\partial_\nu u\,\partial_\nu v\) was produced above by the
		Rellich vector field \(\mathcal R_X(u,v)\).
		Thus for \(l\ne l'\), the shifted spectral part rewrites as
		\[
		(e_{l,R}-e_{l',R})\langle A_Xu_{l,R},u_{l',R}\rangle,
		\]
		which is an off-diagonal error absorbed into the Stage 5 remainder estimates.
		Therefore the leading
		coefficient is the boundary flux; row-dependent separated Schrodinger
		centrifugal terms have already canceled in the physical Hadamard identity
		before the \(l\)-rows are read off.
		
		\emph{Stage 3: boundary-flux coefficient.}
		If \(u_{l,R}=(\sinh s)^{-(n-1)/2}y_{l,R}Y_l\), then \(y_{l,R}(R)=0\).  Hence
		the hyperbolic boundary measure cancels the Schrodinger half-density, and the
		boundary-flux scalar, after the angular factor is separated, is
		\[
		\partial_s y_{l,R}(R)\,\partial_s y_{l',R}(R)
		\]
		with the sign convention fixed above.
		Indeed,
		\[
		\begin{aligned}
			\partial_\nu u_{l,R}|_{s=R}
			&=
			(\sinh R)^{-(n-1)/2}
			\partial_s y_{l,R}(R)Y_l(\theta),
			\\
			&\int_{\partial B_R}
			h\,\partial_\nu u_{l,R}\,\partial_\nu u_{l',R}\,d\sigma_R  \\
			&\qquad =
			\partial_s y_{l,R}(R)\partial_s y_{l',R}(R)
			\int_{\SSph^{n-1}}
			h\,Y_l\overline{Y_{l'}}\,d\theta ,
		\end{aligned}
		\]
		because the derivative of the half-density term is multiplied by
		\(y_{l,R}(R)=0\), and the two factors
		\((\sinh R)^{-(n-1)/2}\) cancel the boundary measure
		\(d\sigma_R=\sinh^{n-1}R\,d\theta\) exactly.
		The Hadamard derivative is taken at the round sphere, where the normal is
		radial.  Since \(y_{l,R}(R)=0\), the tangential derivatives of
		\((\sinh s)^{-(n-1)/2}y_{l,R}(s)Y_l(\theta)\) vanish on the unperturbed
		boundary.  If one instead expands the graph normal for
		\(r=R-h(\theta)\), the remaining tangential correction carries a factor
		\(\sinh^{-2}R\) and bounded finite-sector angular derivatives, hence
		contributes only \(O_{N,L}(e^{-2R}R^{-3})\).
		
		\emph{Stage 4: spectral-defect commutator and cutoff flow.}
		For \(l\ne l'\), the first branches are not exactly \(E_R\)-eigenvectors.
		The fixed-degree ball asymptotics give
		\[
		e_{l,R}-E_R=b_{n+2l}R^{-3}+o_l(R^{-3})=O_N(R^{-3})
		\]
		uniformly for \(l\le N\).  Here we use Krist\'aly's fixed-dimension expansion
		with a genuine \(o_m(R^{-3})\) remainder after the dimension shift
		\(m=n+2l\), as encoded by the first-four-terms statement in
		\cite[Remark 1.1(i)]{Kristaly}; the set \(\{b_{n+2l}:l\le N\}\) is finite,
		and the finitely many fixed-degree remainders have maximum \(o_N(R^{-3})\).
		Hence
		\(e_{l,R}-e_{l',R}=O_N(R^{-3})\).  The resulting shifted-form generator
		terms are lower order.  More precisely, after the angular factor has been
		separated, the radial part of the Schrodinger-conjugated generator for the
		reference scaling field is, up to the harmless global sign determined by the
		height-variation convention,
		\[
		G_R=-(s/R)\partial_s-\frac1{2R}.
		\]
		The fixed-sector endpoint and bulk estimates from
		Lemma~\ref{lem:finite-box-jost-data} give
		\[
		\|G_Ry_{l,R}\|_2
		\le
		\|(s/R)\partial_s y_{l,R}\|_2+\frac1{2R}\|y_{l,R}\|_2
		\le C_NR^{-1}.
		\]
		This is an \(L^2\)-operator-vector bound, so by Cauchy--Schwarz it applies to
		off-diagonal matrix elements.  Hence the off-diagonal generator term is
		\(O_N(R^{-3})O_N(R^{-1})=O_N(R^{-4})\).  In the diagonal case the
		Dirichlet boundary cancellation above gives zero before any estimate is
		needed.
		Here the Schrodinger-conjugated flat generator is computed as follows.  If
		\(X=X^s\partial_s\), then
		\[
		\operatorname{div}_gX=\partial_sX^s+(n-1)\coth(s)X^s .
		\]
		Conjugating
		\(-X^s\partial_s-\frac12\operatorname{div}_g X\) by
		\((\sinh s)^{(n-1)/2}\) gives
		\[
		-X^s\partial_s-\frac12\partial_s X^s .
		\]
		The \((n-1)\coth(s)X^s\) term cancels exactly against the derivative of the
		half-density, so a field \(X^s=\sigma(s)h(\theta)\) has flat generator
		\[
		-\sigma(s)h\partial_s-\frac12\sigma'(s)h .
		\]
		It contains no angular derivative.  The actual cutoff flattening is compared
		with the pure normal flow at the level of the physical shifted form, before
		the separated rows are read off.
		The generator estimate keeps the actual weighted derivative:
		\[
		\|(s/R)\partial_s y_{l,R}\|_2\le C_NR^{-1}.
		\]
		On \([S_N,R]\) the common finite-box profile gives
		\[
		|\partial_s y_{l,R}(s)|\le C_NR^{-3/2},
		\qquad
		\int_{S_N}^R|\partial_s y_{l,R}|^2\,ds\le C_NR^{-2}.
		\]
		On the endpoint interval the factor \(s/R\) is essential in the critical
		channel; the endpoint bound below gives a contribution
		\(O_N(R^{-5/2})\) to the \(L^2\)-norm of
		\((s/R)\partial_s y_{l,R}\) there.
		
		Thus the radial interior extension affects the finite first-branch block only
		through the displayed eigenvalue-defect commutators.  If one instead writes a
		nonradial graph-normal extension, the tangential boundary polarization is a
		symmetric finite-dimensional angular bilinear form, with typical terms such as
		\(\sinh^{-2}R\langle\nabla_S Y_l,h\nabla_S Y_{l'}\rangle\).  Hence it is
		bounded by
		\[
		C_{N,L}\sinh^{-2}R\,|\partial_s y_{l,R}(R)|\,|\partial_s y_{l',R}(R)|
		=O_{N,L}(e^{-2R}R^{-3}).
		\]
		The cutoff field \(-\beta_R h\partial_s\) and the reference scaling field
		\(-(s/R)h\partial_s\) are radial, have the same boundary normal component,
		and coincide on the regular bulk \(s\ge2\eps_R\).  Their difference is radial,
		vanishes at \(s=R\), has zero boundary normal trace, and has a unitary
		generator vector of \(L^2\)-size \(O_{N,L}(R^{-1})\) on first-branch inputs.
		Indeed, if
		\(\sigma_{\rm diff}=s/R-\beta_R\), then the flat generator contains
		\(-\sigma_{\rm diff}h\partial_s-\frac12\sigma_{\rm diff}'h\), but no
		\(\sigma_{\rm diff}''\); the second-derivative transition coefficients are
		absorbed into the Stage 5 finite-block remainders.  On the protected
		origin \(\sigma_{\rm diff}=s/R\), and on the transition layer
		\(\sigma_{\rm diff}=O(\eps_R/R)\), \(\sigma_{\rm diff}'=O(R^{-1})\), so the
		same endpoint and transition estimates give the stated \(O_{N,L}(R^{-1})\)
		vector bound.  The zero boundary trace means the difference field contributes
		no boundary flux.  Its bulk contribution is instead the commutator
		\[
		(e_{l,R}-e_{l',R})
		\langle A_{X_{\rm diff}}u_{l,R},u_{l',R}\rangle
		=
		O_N(R^{-3})O_{N,L}(R^{-1})
		=
		O_{N,L}(R^{-4}),
		\]
		which is \(o_{N,C,L,\delta_{\rm cut}}(R^{-3})\).
		
		Therefore the first-order cutoff-to-normal-flow difference is already
		\(O_{N,L}(R^{-4})\) globally.  The remaining finite first-branch
		coefficient remainders after subtracting the first-order local-height
		model are only endpoint, transition, and higher-order terms.
		
		\emph{Stage 5: endpoint, transition, and shear remainders.}
		On the
		transition layer \(I_R=[\eps_R,2\eps_R]\), the fixed-sector endpoint estimates
		give, uniformly for \(l\le N\),
		\[
		|y_{l,R}(s)|\le C_NR^{-3/2}s^{1/2},\qquad
		|s\partial_s y_{l,R}(s)|\le C_NR^{-3/2}s^{1/2},
		\qquad 0<s\le2\eps_R,
		\]
		with stronger powers of \(s\) in noncritical channels.  These are
		normalized first-branch bounds, not projected-complement estimates.  For
		each fixed \(l\le N\), the regular finite-box solution
		\[
		u_l(s,\rho),\qquad \rho=\pi/R+O_N(R^{-2}),
		\]
		is a small-energy perturbation of the positive zero-energy regular solution
		on \(0<s\le2\eps_R\).  That zero-energy solution has endpoint behavior
		\(s^{\nu_l+1/2}\) and a nonzero linear coefficient at infinity; hence its
		\((0,R)\)-norm is of order \(R^{3/2}\).  The finite-box Jost normalization
		from Lemma~\ref{lem:finite-box-jost-data}, the same input used for the
		boundary-slope asymptotic, transfers this uniformly for \(l\le N\) to the
		actual normalized first eigenfunctions:
		\[
		|y_{l,R}(s)|+|s\partial_s y_{l,R}(s)|
		\le C_NR^{-3/2}s^{\nu_l+1/2}
		\le C_NR^{-3/2}s^{1/2},
		\qquad 0<s\le2\eps_R .
		\]
		In the critical \((n,l)=(2,0)\) channel this says that the comparison
		solution is \(s^{1/2}\) at the endpoint and grows like \(s\) at infinity, so
		the normalized endpoint coefficient is \(O(R^{-3/2})\).  The same
		fixed-sector normalization gives the boundary-slope size \(R^{-3/2}\).
		Together with the finite-box bulk profile
		\(|\partial_s y_{l,R}|\le C_NR^{-3/2}\) on \([S_N,R]\), this yields
		\[
		\|(s/R)\partial_s y_{l,R}\|_2\le C_NR^{-1}.
		\]
		Thus the worst critical transition scalar satisfies
		\[
		R^{-1}\eps_R^{-2}\int_{I_R}|y_{l,R}y_{l',R}|\,ds
		\le C_NR^{-4},
		\]
		and the mixed transition term satisfies
		\[
		R^{-1}\eps_R^{-1}
		\int_{I_R}\bigl(|\partial_s y_{l,R}|\,|y_{l',R}|
		+|y_{l,R}|\,|\partial_s y_{l',R}|\bigr)\,ds
		\le C_NR^{-4}.
		\]
		The \(G_s^2\), quadratic metric, and cutoff-square terms are smaller, since
		they contain an additional factor \(R^{-1}\), \(\eps_R\), or a positive local
		majorant.  There is no hidden \(O(R^{-2})\) finite-height bulk term here.  The
		constant-height case is a consistency check on the linear coefficient and on
		the constant-direction quadratic scale: if \(h\equiv c\), the domain is
		exactly \(B_{R-c}(o)\), and
		\[
		\frac{\pi^2}{(R-c)^2}
		=
		\frac{\pi^2}{R^2}
		+2\pi^2cR^{-3}
		+O_C(R^{-4});
		\]
		hence the quadratic height contribution is \(O_C(R^{-4})\), not
		\(O(R^{-2})\).  The genuinely variable-height quadratic pieces are controlled
		separately.  In the exact coordinate pullback the radial metric quadratic
		remainder is bounded, on fixed first-branch inputs, by
		\[
		\int \frac{(\beta_R'h)^2}{1-\beta_R'h}\,
		|\partial_s y_{l,R}|^2\,ds
		\le C_{C}R^{-2}\int|\partial_s y_{l,R}|^2\,ds
		=O_{N,C}(R^{-4}).
		\]
		On the endpoint transition collar the displayed endpoint bounds give the same
		or smaller order.  In the critical channel the displayed endpoint bound gives
		\(|\partial_s y_{l,R}|\le C_NR^{-3/2}s^{-1/2}\), so on
		\([\eps_R,2\eps_R]\)
		\[
		\int_{I_R}(\beta_R')^2|\partial_s y_{l,R}|^2\,ds
		\le C_NR^{-5}\int_{\eps_R}^{2\eps_R}s^{-1}\,ds
		=O_N(R^{-5})=o_N(R^{-4}).
		\]
		The noncritical channels have better endpoint powers.
		The shear-square term has schematic size
		\[
		\int
		\frac{\beta_R^2|\nabla h|^2}{\sinh^2(F_h)}
		|D_sy_{l,R}|\,|D_sy_{l',R}|\,ds .
		\]
		Splitting at a fixed \(S_N\), the endpoint/Jost bounds on
		\([\eps_R,S_N]\) and the exponential factor \(\sinh^{-2}(F_h)\) on
		\([S_N,R]\) give
		\[
		\int
		\frac{\beta_R^2|\nabla h|^2}{\sinh^2(F_h)}
		|D_sy_{l,R}|\,|D_sy_{l',R}|\,ds
		\le C_{N,C,L}R^{-4}.
		\]
		Here \(D_s\) is the conjugated radial derivative appearing in the
		Schrodinger pullback; the protected origin has \(\beta_R=0\), and the
		transition collar is covered by the same endpoint estimates.  In the critical
		endpoint channel this estimate is borderline but smaller than required.  On
		\([\eps_R,S_N]\) one has \(\beta_R=O(s/R)\),
		\(\sinh(F_h)\simeq s\), and
		\(|D_sy_{l,R}|\le C_NR^{-3/2}s^{-1/2}\), hence the endpoint integrand is
		bounded by \(C_{N,C,L}R^{-5}s^{-1}\).  Its integral is
		\[
		C_{N,C,L}R^{-5}\log(S_N/\eps_R)=O_{N,C,L}(R^{-5}\log R)=o(R^{-4}).
		\]
		On \([S_N,R]\) the same derivative factors are \(O_N(R^{-3})\) and
		\(\sinh^{-2}(F_h)\) is exponentially small after enlarging \(S_N\), giving an
		even smaller contribution.
		The pulled-back mass term after the physical radial Schrodinger conjugation
		has coefficient \(1-\beta_R'h\) and is strictly linear in \(h\).  Although the
		unitary generator \(A_Xu_{l,R}\) has \(L^2\)-size \(O_N(R^{-1})\), the exact
		Hadamard--Rellich identity uses only the skew-adjoint spectral-defect
		commutator.  The identity may be written in physical \(L^2(dV_{\HH^n})\);
		after the radial Schrodinger conjugation, \(A_X\) is unitarily equivalent to
		the flat radial generator used in the size estimate.  No angular derivative is
		hidden here: \(X_h=-\beta_Rh\,\partial_s\) has only a radial component, so
		\(X_h\cdot\nabla\) differentiates only in \(s\), while
		\(\operatorname{div}X_h\) contains \(h(\theta)\) as a multiplier and no
		\(\nabla_\theta h\) term.  Thus the relevant term is
		\[
		(e_{l,R}-e_{l',R})
		\langle A_Xu_{l,R},u_{l',R}\rangle
		=O_N(R^{-3})O_{N,C,L}(R^{-1})
		=O_{N,C,L}(R^{-4}),
		\]
		not an independent \(\|A_Xu_{l,R}\|^2\) term.
		
		The estimate uses only the \(C^0\) and Lipschitz bounds on the smooth height.
		If angular second derivatives of \(h\) appear in a coordinate expansion of
		this finite block, one integrates by parts on the sphere at this stage; the
		derivatives then fall on the fixed finite set of spherical harmonics or on the
		weak gradient of \(h\), giving constants depending only on \(N\) and \(L\).
		Thus no \(C^2\) norm of a smooth approximating height enters
		\(\Delta_{ll',R}\).  Let \(\Delta_{ll',R}\) denote the sum of these
		finite-block remainders after the boundary flux and the skew-adjoint
		commutator have been subtracted.  Consequently
		\[
		|\Delta_{ll',R}|=o_{N,C,L,\delta_{\rm cut}}(R^{-3})
		\]
		uniformly for \(l,l'\le N\).  Hence the cutoff-flattened matrix has the same
		leading flux coefficient as the normal-flow derivative.
		
		\emph{Stage 6: assembly of the finite first-branch block.}
		The common finite-sector boundary-slope asymptotic is the fixed-\(l\)
		first-branch part of Lemma~\ref{lem:finite-box-jost-data}.  It does not rely
		on the projected Green estimates in
		Lemma~\ref{lem:critical-endpoint-jost-green}:
		\[
		\partial_s y_{l,R}(R)
		=
		-\sqrt{2/R}\,\pi/R+O_N(R^{-5/2}),
		\]
		and therefore, uniformly for \(l,l'\le N\),
		\[
		\partial_s y_{l,R}(R)\partial_s y_{l',R}(R)
		=
		\left(\sqrt{2/R}\,\pi/R\right)^2+O_N(R^{-4})
		=
		2\pi^2R^{-3}+O_N(R^{-4}).
		\]
		Thus, in particular,
		\[
		B_h^{E_R}(u_{l,R},u_{l,R})
		=
		2\pi^2R^{-3}\int_{\mathbb S^{n-1}}h|Y_l|^2\,d\theta
		+o_{N,C,L,\delta_{\rm cut}}(R^{-3}),
		\]
		and the polarized version gives the same \(2\pi^2h\) boundary matrix
		coefficient on the whole finite first-branch block.
		This gives \(M_{ll',R}=2\pi^2R^{-3}+o_{N,C,L,\delta_{\rm cut}}(R^{-3})\),
		uniformly for \(l,l'\le N\).  The finite-dimensional angular form identity
		follows by summing the displayed matrix identity against
		\(\int hY_\alpha\overline{Y_\beta}\,d\theta\); the error is uniform on the
		finite \(P_N\) space and over \(0\le h\le C\).
		The exact finite-block form is obtained by adding this first-variation matrix
		and the finite-block remainders to the unperturbed separated first-branch
		energies.  Since
		\[
		e_{l,R}-E_R=b_{n+2l}R^{-3}+o_l(R^{-3})
		=
		R^{-3}(b_n+\tau_{n,l})+o_l(R^{-3}),
		\]
		the unperturbed finite block contributes
		\[
		R^{-3}\langle f_{\le},P_N(T_n+b_n)P_Nf_{\le}\rangle
		+o_N(R^{-3})\|f_{\le}\|^2 .
		\]
		Combining this with the boundary-flux matrix identity and the bound on
		\(\Delta_{ll',R}\) gives the exact shifted finite-block form asserted in the
		statement.
	\end{proof}
	
	\begin{remark}[Separated-coordinate artifacts]
		\label{rem:separated-coordinate-artifacts}
		The row-dependent centrifugal terms that appear in separated Schrodinger
		coordinates are artifacts of expanding the already self-adjoint physical
		Hadamard variation after separation.  The proof of
		Lemma~\ref{lem:finite-first-branch-hadamard} applies the signed polarized
		Hadamard identity to the physical Dirichlet form before the rows are
		separated; that identity says that the first variation is the boundary normal
		flux, plus the skew-adjoint spectral-defect commutator displayed in the
		proof.  If the same calculation is expanded in separated Schrodinger
		coordinates, apparent off-diagonal Wronskian or \((V_l-V_{l'})\) bulk terms
		must be grouped with the angular-metric, Jacobian, and half-density pieces of
		the physical variation.  Their aggregate contribution is precisely the
		commutator \((e_l-e_{l'})\langle A_Xu_l,u_{l'}\rangle\), which is
		\(O_N(R^{-4})\) for fixed \(N\).  Thus no independent degree-difference bulk
		integral contributes to the leading \(2\pi^2\) coefficient.
	\end{remark}
	
	\paragraph{All-degree lower form.}
	For
	\[
	w=J_Rf_{\leq}+q_{\rm rad,\leq}+w_{>},\qquad f_{\leq}=P_Nf,
	\]
	the term \(q_{\rm rad,\leq}\) lies in the direct sum of the radial spectral
	complements to the first radial branch in the finite angular sector
	\(l\le N\), and \(w_>\) is the projection of \(w\) onto angular degrees
	\(>N\).  We write \({\mathcal Q}_{\rm rad}\) for the corresponding
	post-reservation positive radial-complement form on the finite angular
	sector.  Its precise fiberwise expression is not needed in the assembly
	below; the property used is
	\[
	{\mathcal Q}_{\rm rad}[q_{\rm rad,\leq}]
	\ge c_{\rm rad}R^{-2}\|q_{\rm rad,\leq}\|^2,
	\]
	after the fixed mixed-term reserve has already been removed.  The formal
	all-degree estimate used in the conclusion is recorded below as
	Lemma~\ref{lem:all-degree-lower-form-assembly}.  Combining
	Lemmas~\ref{lem:high-mode-angular-reserve},
	\ref{lem:high-mode-error-routing}, and
	\ref{lem:first-radial-complement-mixed} with the six-class routing in
	Definition~\ref{def:routed-coefficient-classes},
	the lower form gives the estimate
	\[
	Q_h[w]-E_R\|w\|^2
	\geq
	R^{-3}B_N^h[f_{\leq},w_{>}]
	-(\delta_{\rm eff}/32)R^{-3}\|f_{\leq}\|^2
	+c_{\rm rad}R^{-2}\|q_{\rm rad,\leq}\|^2
	-o_{N,C,L}(R^{-3})\|w\|^2.
	\]
	Here \(B_N^h\) is the pre-compression form
	\[
	\begin{aligned}
		B_N^h[f_{\leq},w_{>}]
		&=
		\langle f_{\leq},P_NL_hP_Nf_{\leq}\rangle
		+R^3(1-\rho_{\rm res})\operatorname{Sep}_>[w_{>}]
		-\frac{M_C^2}{4\eta_{\rm gap}}\|w_>\|^2.
	\end{aligned}
	\]
	The finite Hadamard boundary-flux matrix is used only on the finite block
	\(P_NL^2(\mathbb S^{n-1})\).  We do not identify a leading \(2\pi^2h\)
	boundary matrix for pairs with one degree above \(N\).  The bounded
	order-\(R^{-3}\) low/high coupling is instead Young-split as
	\[
	-M_CR^{-3}\|f_{\leq}\|\,\|w_>\|
	\ge
	-\eta_{\rm gap}R^{-3}\|f_{\leq}\|^2
	-\frac{M_C^2}{4\eta_{\rm gap}}R^{-3}\|w_>\|^2.
	\]
	The finite-low side is the displayed
	\(\delta_{\rm eff}/32=\eta_{\rm gap}\) loss; the high-side scalar penalty is
	the negative term included in \(B_N^h\).  The remaining low/high physical
	cross terms, after this bounded-multiplication part is separated, have a
	finite-low side \(o_{N,C,L}(R^{-3})\|f_{\leq}\|^2\) and a high side that is an
	\(R\)-vanishing multiple \(\tau'_{N,R}\operatorname{Sep}_>\), with
	\(\tau'_{N,R}\to0\) for fixed \(N\).  This additional vanishing high-side
	fraction is included in the same \(\rho_{\rm res}\operatorname{Sep}_>\)
	reserve before the clean \((1-\rho_{\rm res})\operatorname{Sep}_>\) summand
	is passed to compression.
	This is the output of the high-mode reserve step, not its input.  Before the
	reserve is spent, the high-mode contribution has the form
	\[
	\operatorname{Sep}_>[w_{>}]+\operatorname{Err}_{\rm hi}[w_{>}]
	=
	(1-\rho_{\rm res})\operatorname{Sep}_>[w_{>}]
	+
	\bigl(\rho_{\rm res}\operatorname{Sep}_>[w_{>}]
	+\operatorname{Err}_{\rm hi}[w_{>}]\bigr).
	\]
	Lemma~\ref{lem:high-mode-error-routing} is applied only to the parenthesized
	reserve summand, giving
	\[
	\rho_{\rm res}\operatorname{Sep}_>[w_{>}]
	+\operatorname{Err}_{\rm hi}[w_{>}]
	\ge -o_{N,C,L}(R^{-3})\|w_{>}\|^2.
	\]
	Thus the high term retained in \(B_N^h\) is precisely the unspent summand
	\((1-\rho_{\rm res})\operatorname{Sep}_>\); it is not used in the error
	absorption.
	Any favorable high-high boundary-height contribution has been discarded in
	this lower bound.  The notation suppresses the harmless \(R\)-dependence of
	this pre-compression form; it is replaced below by a fixed scalar high block
	using the high-mode compression estimate.
	The finite first-branch/radial-complement terms have already been absorbed in
	this display by Lemma~\ref{lem:first-radial-complement-mixed}.  A fixed
	fraction of the finite radial-complement form was reserved for the
	\(\eta{\mathcal Q}_{\rm rad}\) term in that lemma and is absorbed into the
	post-reservation notation \(c_{\rm rad}\); the remaining finite first-branch
	cost is \(o_{N,C,L}(R^{-3})\).
	The finite-box Jost normalization supplies the projected Green-kernel and
	endpoint-normalization estimates, but the branch gap itself uses the simpler
	Bessel comparison.  The high-angular reserve is applied before compression:
	a fixed fraction of the raw high separated form absorbs the high-mode
	coefficient errors, and only the remaining positive form is compressed to the
	coefficient \((1-\rho_{\rm res})(b_n+\tau_{n,N+1})\).  The subsequent Young
	step does not reuse the high-mode error absorption.
	The coefficient remainder is organized into three bounded components.  The
	finite first-branch block keeps the fixed
	\(\eta_{\rm gap}=\delta_{\rm eff}/32\) Young loss from the bounded low/high
	coupling; after that coupling has been separated, the remaining finite-low
	coefficient errors are \(o_{N,C,L}(R^{-3})\).  Finite radial complements may
	use a fixed \(\eta_{\rm form}\) fraction of the radial-complement form.  High
	modes, before compression, absorb bounded errors through a cutoff-small and
	\(R\)-vanishing reserve
	\[
	(\sigma_N+\sigma_{N,R})\operatorname{Ang}[w_{>}]
	+\tau_{N,R}\operatorname{Sep}_>[w_{>}],
	\qquad \sigma_N\to0\ (N\to\infty),\quad
	\sigma_{N,R},\tau_{N,R}\to0\ (R\to\infty)
	\]
	after the cutoff \(N\) has been chosen.  This notation is taken in the
	boundary-flattened coordinates
	\[
	r=F_h(s,\theta)=s-\beta_R(s)h(\theta),
	\]
	where one may take
	\[
	\beta_R(s)=\frac{s}{R}\chi(s/\eps_R)
	\]
	with a fixed smooth cutoff \(\chi=0\) on \([0,1]\), \(\chi=1\) on
	\([2,\infty)\), and \(0\le\chi\le1\).  Hence \(\beta_R=0\) near the
	singular endpoint, \(\beta_R=s/R\) on the regular bulk, \(\beta_R(R)=1\),
	and the transition derivatives have size
	\[
	\beta_R'=O(R^{-1}),\qquad
	\beta_R''=O(R^{-1}\eps_R^{-1})
	\]
	on \(\eps_R\le s\le2\eps_R\).  The terms produced by these transition
	derivatives are routed through classes (5) and (6) below.  The symbol \(w_l\)
	denotes the spherical-degree \(l\) radial component of
	\(w_{>}=\sum_{l>N}w_lY_l\).  With
	\(|\beta_R'|\le C_\beta/R\), the Jacobian
	\(\partial_sF_h=1-\beta_R'h\) satisfies
	\[
	\partial_sF_h\ge 1-C_\beta C/R\ge \frac12
	\]
	after enlarging \(R_0\).  Thus the boundary-flattening map is a
	diffeomorphism throughout the range where the form estimates are used.  With
	\(\lambda_l=l(l+n-2)\),
	\[
	\operatorname{Ang}[w_{>}]
	=\sum_{l>N}\lambda_l\int \sinh^{-2}(F_h)|w_l|^2
	\]
	and \(\operatorname{Sep}_>\) denotes the shifted raw high separated form in
	the fixed-cylinder Schrodinger-flattened coordinates:
	\[
	\operatorname{Sep}_>[w_>]
	=
	\sum_{l>N}
	\left\{
	\int\left(|\partial_sw_l|^2-\pi^2R^{-2}|w_l|^2\right)
	+c_l\int\sinh^{-2}(F_h)|w_l|^2
	\right\}.
	\]
	The form domain is still the fixed-cylinder Dirichlet domain.  Indeed
	\(\beta_R(R)=1\), so \(F_h(R,\theta)=R-h(\theta)\); hence the physical
	boundary condition \(z=0\) on the graph boundary pulls back to
	\(\widetilde z(R,\theta)=0\).  The half-density factors used in the
	Schrodinger conjugation are nonzero there, so every spherical coefficient
	appearing in \(w_>\) has trace \(w_l(R)=0\).  Thus ``raw'' here refers only
	to removing the perturbative Jacobian, metric, shear, and half-density
	coefficient errors from the pulled-back form; it does not change the fixed
	Dirichlet endpoint at \(s=R\).
	The \(a_n^2\) part of \(E_R=a_n^2+\pi^2R^{-2}\) has been cancelled by the
	radial Schrodinger conjugation.  The remaining \(-\pi^2R^{-2}\) mass is
	included in the displayed shifted radial kinetic term.  Perturbative
	Jacobian, metric, and half-density differences between the pulled-back
	physical form and this raw separated form are precisely the coefficient
	classes routed below.  In these classes \(a\) denotes only the coordinate
	Jacobian \(a=\partial_sF_h\).  The full hyperbolic volume factor has already
	been removed by the physical radial Schrodinger conjugation, so it is not an
	additional order-one mass multiplier in the fixed flat Hilbert space.
	
	\begin{lemma}[High-mode angular reserve]
		\label{lem:high-mode-angular-reserve}
		Assume \(N\ge N_A(n)\).  For every high-mode function \(w_>\), the raw high
		separated form satisfies
		\[
		\operatorname{Sep}_>[w_>]\ge c_A(n)\operatorname{Ang}[w_>],
		\qquad c_A(n)=1/2.
		\]
		Here \(\operatorname{Sep}_>\) is the raw high separated form with the bare
		centrifugal coefficient \(c_l\) evaluated in the flattened coordinate
		\(F_h\).  Perturbative differences, Jacobian factors, multiplicative
		coefficient errors, and finite-\(R\) remainders are not part of this reserve;
		they are charged separately in classes (2), (3), (4), (5), and (6).
	\end{lemma}
	
	\begin{proof}
		Before the bounded coefficient-error classes are routed, the raw high
		separated form contains the nonnegative shifted radial kinetic part and the
		unperturbed centrifugal term
		\[
		\sum_{l>N}\int\left(|\partial_sw_l|^2-\pi^2R^{-2}|w_l|^2\right)
		+
		\sum_{l>N}c_l\int\sinh^{-2}(F_h)|w_l|^2 .
		\]
		The first sum is nonnegative by the one-dimensional Dirichlet--Friedrichs
		Poincar\'e inequality on \((0,R)\): high angular modes have the zero-trace
		Friedrichs behavior at \(s=0\) and Dirichlet trace at \(s=R\).  More
		explicitly, for \(l>N_A(n)\) the inverse-square endpoint exponent satisfies
		\(\nu_l>1/2\).  This holds for all \(l\ge1\), hence for \(l>N_A(n)\), because
		the coefficient satisfies \(c_l>0\) and \(c_l=\nu_l^2-1/4\).  Functions in the
		Friedrichs form domain therefore have ordinary \(H^1\) trace \(w_l(0)=0\);
		the graph boundary gives \(w_l(R)=0\).  Thus
		\[
		\int_0^R|\partial_sw_l|^2\,ds
		\ge \pi^2R^{-2}\int_0^R|w_l|^2\,ds .
		\]
		This is the Dirichlet--Dirichlet constant, not the mixed
		Neumann--Dirichlet constant.  The coefficient \(F_h\) occurs only in the
		centrifugal weight above; the Jacobian perturbations of the radial kinetic
		term are not part of the raw reserve and are charged in classes (3)--(6).
		Since \(c_l\ge\lambda_l/2\) for \(l>N_A(n)\), this part alone gives
		\[
		\operatorname{Sep}_>[w_>]\ge c_A(n)\operatorname{Ang}[w_>],
		\qquad c_A(n)=1/2.
		\]
	\end{proof}
	
	\begin{definition}[Routed coefficient classes]
		\label{def:routed-coefficient-classes}
		After the boundary flattening \(F_h\), the radial Schrodinger conjugation, and
		the subtraction of the unperturbed separated first-branch energies, the
		coefficient bookkeeping used below is summarized schematically by
		\begin{equation}
			\label{eq:pulled-back-six-class-decomposition}
			Q_h^{\rm flat}[w]-E_R\|w\|^2
			=
			Q_{\rm low}^{\rm Had}[J_RP_Nf]
			+Q_{\rm sep}^{\rm raw}[q_{\rm rad,\leq}+w_>]
			+\sum_{j=1}^{6}{\mathcal E}_j[w]
			+{\mathcal R}_{N,R}[w].
		\end{equation}
		Here \(Q_{\rm low}^{\rm Had}\) is the finite first-branch Hadamard block of
		Lemma~\ref{lem:finite-first-branch-hadamard},
		\(Q_{\rm sep}^{\rm raw}\) is the raw separated radial-complement/high-mode
		form before perturbative coefficient errors, the six
		${\mathcal E}_j$ are exactly the classes listed below, and
		\({\mathcal R}_{N,R}\) denotes only the fixed-\(N\) little-\(o(R^{-3})\)
		remainders already isolated in the local lemmas.  Equation
		\eqref{eq:pulled-back-six-class-decomposition} is an identity used to organize
		the estimates; once a term is assigned to one of the six classes, it is routed
		exactly once.
		
		The six coefficient classes are allocated as follows.  On the finite
		first-branch block \(J_RP_Nf\), classes (2)--(6) are not estimated as
		separate absolute-value errors.  Their net contribution, after subtracting
		the constant-height model, is part of the physical finite-Hadamard variation in
		Lemma~\ref{lem:finite-first-branch-hadamard}.  The separate absolute-value
		routing in classes (2)--(6) applies to the radial complement and high-mode
		blocks.
		\begin{enumerate}
			\item[(1)] \emph{Global scaling.}  After subtracting the exact
			constant-height model, the remaining scaling coefficient is
			\(O_C(R^{-4})\) on finite first branches and \(O_C(R^{-1})\) relative to the
			radial-complement form.  For \(l,l'\le N\), the regular local-height matrix
			satisfies, by Lemma~\ref{lem:finite-first-branch-hadamard},
			\[
			M_{ll',R}=2\pi^2R^{-3}+o_{N,C,L,\delta_{\rm cut}}(R^{-3}).
			\]
			Thus the leading \(2\pi^2h\) term is the full finite angular multiplication
			matrix on the low block \(P_NL^2\), while the finite first-branch class-(1)
			remainder is only \(o_{N,C,L,\delta_{\rm cut}}(R^{-3})\).  The protected
			origin where \(\beta_R=0\), the endpoint displacement, and the transition
			derivatives are kept in classes (2) and (5), so no singular endpoint term is
			hidden in this flux computation.  The radial-complement part gives
			\[
			o_{N,C,L}(R^{-3})\|f_{\leq}\|^2
			+\eta_{\rm form}\,{\mathcal Q}_{\rm rad}[q_{\rm rad}]
			\]
			after \(R_0\) is large.
			\item[(2)] \emph{Singular-potential displacement.}  For the complement and
			high-mode routing, on \(\eps_R\le s\le1\) this has exterior size
			\(O_C(R^{-1}s^{-2})\), and for \(s\ge1\) it is exponentially weighted.
			Radial complements use the
			projected Green/Jost estimates from
			Lemma~\ref{lem:critical-endpoint-jost-green}.  In high modes the pure centrifugal
			displacement has the favorable sign and is not a negative error term: for
			\(N\ge N_A(n)\) one has
			\(c_l\ge0\), while \(0\le\beta_Rh\) gives \(F_h=s-\beta_Rh\le s\) and hence
			\[
			c_l\bigl(\sinh^{-2}(F_h)-\sinh^{-2}s\bigr)|w_l|^2\ge0 .
			\]
			Thus this piece is discarded in lower bounds.  The remaining multiplicative
			and Jacobian errors from the pullback are routed through classes (3), (5),
			and (6), where their high-mode errors are charged to the
			\((\sigma_N+\sigma_{N,R})\operatorname{Ang}
			+\tau_{N,R}\operatorname{Sep}_>\) reserve.
			\item[(3)] \emph{Factors of \(a-1\).}  Multiplicative coefficient errors of
			size \(O_C(R^{-1})\) are included in the physical finite-Hadamard remainder
			on the finite first-branch block after the constant-height model is removed.
			On complement/high blocks they are estimated relatively.  They contribute an
			\(\eta_{\rm form}\) radial-complement loss and a vanishing
			\(\tau_{N,R}\operatorname{Sep}_>\) error in high modes.
			\item[(4)] \emph{Unitary shear.}  The mixed term has schematic form
			\[
			\frac{\beta_R}{\sinh^2(F_h)}\,\nabla h\cdot D_sw\,\nabla_\theta w .
			\]
			On the finite first-branch block this term is not estimated separately by an
			absolute-value bound.  It is grouped with the companion metric, Jacobian, and
			half-density terms in the physical finite-Hadamard variation of
			Lemma~\ref{lem:finite-first-branch-hadamard}; the apparent separated-coordinate
			bulk terms cancel before the angular matrix is read off.  On the
			radial-complement and high-mode blocks, Lemma~\ref{lem:flattened-shear-form}
			is applied before angular projection.  On finite radial complements, the
			derivative part of the shear is charged to a fixed \(\eta_{\rm form}\)
			portion of the radial-complement kinetic form.  On high modes the Young
			parameter is instead chosen to vanish with \(R\); for instance
			\(\eta_R=R^{-1/2}\) gives
			\[
			\eta_R\operatorname{Sep}_>[w_>]
			+C_{C,L}\eta_R^{-1}R^{-2}\operatorname{Ang}[w_>],
			\qquad \eta_R\to0,\quad \eta_R^{-1}R^{-2}\to0.
			\]
			Thus high shear contributes only to \(\tau_{N,R}\operatorname{Sep}_>\) and
			\(\sigma_{N,R}\operatorname{Ang}\); it never spends a fixed
			\(\eta_{\rm form}\) portion of the high block passed to compression.
			\item[(5)] \emph{Radial-Jacobian endpoint and transition terms.}  The
			endpoint scalar has size \(O(R^{-1}s^{-2})\), while transition derivatives
			are supported on \(\eps_R\le s\le2\eps_R\).  Here \(G_s\) denotes the
			radial logarithmic Jacobian coefficient arising in the unitary pullback,
			for instance \(G_s=\partial_s\log(\partial_sF_h)\) up to uniformly bounded
			smooth factors.  Thus it is controlled by \(\beta_R''\), not merely by
			\(\beta_R'\), and
			\[
			G_s=O_C(R^{-1}\eps_R^{-1}),\qquad
			G_s'=O_C(R^{-1}\eps_R^{-2}).
			\]
			On the finite first-branch block these terms are part of
			Lemma~\ref{lem:finite-first-branch-hadamard}.  On finite radial complements,
			critical diagonal transition terms are integrated by parts with the smooth
			cutoff.  Because the cutoff derivative is compactly supported inside
			\((\eps_R,2\eps_R)\), the boundary terms vanish at the transition endpoints,
			and the remaining \(G_s'\) and \(G_s^2\) scalars are controlled by the
			projected Green estimates from Lemma~\ref{lem:critical-endpoint-jost-green}.
			Indeed, on the layer
			\(I_R=[\eps_R,2\eps_R]\),
			the sharper Green diagonal from
			Lemma~\ref{lem:critical-endpoint-jost-green} gives
			\[
			\int_{I_R}|q|^2\,ds
			\le C\eps_R^2\log(C/\eps_R)\,k[q],
			\]
			for projected critical complements and hence
			\[
			R^{-1}\eps_R^{-2}\int_{I_R}|q|^2\,ds
			\le C R^{-1}\log(C/\eps_R) k[q]=o_R(1)k[q].
			\]
			The \(G_s^2\) scalar is smaller, and mixed \(G_sD_sq\,q\) terms are reduced
			to the same scalar form by this smooth-cutoff integration by parts.  The
			conjugated difference \(D_s-\partial_s=O(s^{-1})\) is absorbed into the same
			\(O(R^{-1}\eps_R^{-2})\) transition scalar bound.  In high modes the
			corresponding endpoint and transition scalars are routed through the relative
			\(\tau_{N,R}\operatorname{Sep}_>\) bound or through the scalar-to-angular
			conversion used in Lemma~\ref{lem:high-mode-error-routing}.
			\item[(6)] \emph{Angular and mixed Jacobian terms.}  Angular Jacobian
			coefficients satisfy \(G_A=O_L(R^{-1})\).  The principal gradient shear was
			already separated in class (4), via
			Lemma~\ref{lem:flattened-shear-form}.  On the finite first-branch block, the
			lower-order angular Jacobian and half-density terms are grouped in
			Lemma~\ref{lem:finite-first-branch-hadamard}, together with the shear and
			radial-Jacobian artifacts.  On the radial-complement and high-mode blocks,
			they are first routed by Lemma~\ref{lem:angular-jacobian-ibp} above; the
			angular-kinetic part is charged to the class-(4) reserve, and the true scalar
			remainder has size
			\[
			C_{C,L}\alpha_h\sinh^{-2}(F_h)|w|^2,\qquad
			\alpha_h=\beta_R^2
			\bigl(\coth^2(F_h)+\operatorname{csch}^2(F_h)\bigr).
			\]
			Here \(\alpha_h\le C_C\) everywhere and \(\alpha_h\le C_CR^{-2}\) in the
			endpoint collar; the regular bulk is exponentially weighted by
			\(\sinh^{-2}(F_h)\).  The residual carries no angular derivative.  Although
			the weight
			\(\sinh^{-2}(F_h)\) depends on \(\theta\), the bounds
			\(0\le h\le C\), the protected origin where \(\beta_R=0\), and large \(R\)
			give a uniform comparison constant \(K(C)\) with
			\[
			K(C)^{-1}\sinh^{-2}s\le \sinh^{-2}(F_h)\le K(C)\sinh^{-2}s .
			\]
			Flat spherical orthogonality, this two-sided comparability, and
			\(\lambda_l\ge\lambda_{N+1}\) for \(l>N\) give the legitimate high-mode
			scalar-to-angular conversion
			\[
			C_{C,L}\int \alpha_h\sinh^{-2}(F_h)|w_>|^2
			\le C_{C,L}\lambda_{N+1}^{-1}\operatorname{Ang}[w_>],
			\]
			because these lower-order residuals no longer carry an angular derivative.
		\end{enumerate}
		No entry in this routing uses a critical Hardy inequality: critical
		finite-low diagonal or mixed terms are routed through the projected
		Green/Jost estimates from Lemma~\ref{lem:critical-endpoint-jost-green}.
	\end{definition}
	
	\begin{lemma}[High-mode error absorption and reserve]
		\label{lem:high-mode-error-routing}
		Fix \(N\ge N_A(n)\), the reserve fraction \(\rho_{\rm res}\in(0,1/4)\), and
		the collar exponent \(\delta_{\rm cut}\).  For \(0\le h\le C\) with
		\(\Lip(h)\le L\), let \(\operatorname{Err}_{\rm hi}\) denote the aggregate
		high-mode error from the routed coefficient classes after discarding the
		nonnegative pure high-mode singular-potential displacement.  This is a
		quadratic form on the high-angular subspace: it is the finite sum of the
		perturbative high-mode pieces left by the six coefficient classes, and it does
		not include the clean separated form \(\operatorname{Sep}_>\), the finite
		first-branch block, or the pure favorable-sign term just discarded.  There are
		quantities \(\sigma_N,\sigma_{N,R},\tau_{N,R}\), uniformly for the admissible
		heights \(h\), with
		\[
		\sigma_N\to0\quad(N\to\infty),\qquad
		\sigma_{N,R},\tau_{N,R}\to0\quad(R\to\infty\text{ for fixed }N),
		\]
		such that
		\[
		|\operatorname{Err}_{\rm hi}[w_>]|
		\leq
		\tau_{N,R}\operatorname{Sep}_>[w_>]
		+(\sigma_N+\sigma_{N,R})\operatorname{Ang}[w_>]
		+o_{N,C,L}(R^{-3})\|w_>\|^2 .
		\]
		After choosing \(N\) so that \(\sigma_N/c_A(n)\le\rho_{\rm res}/2\), and then
		enlarging \(R_0\) so that
		\[
		\tau_{N,R}+\sigma_{N,R}/c_A(n)\le \rho_{\rm res}/2,
		\]
		the reserve absorption inequality
		\[
		\rho_{\rm res}\operatorname{Sep}_>[w_>]
		+\operatorname{Err}_{\rm hi}[w_>]
		\ge -o_{N,C,L}(R^{-3})\|w_>\|^2
		\]
		holds.  Thus only the clean remainder
		\((1-\rho_{\rm res})\operatorname{Sep}_>[w_>]\) is passed to the
		block-form compression step.
		When this lemma is used together with the low/high physical cross routing in
		Lemma~\ref{lem:all-degree-lower-form-assembly}, the same statement allows an
		additional term
		\(\tau'_{N,R}\operatorname{Sep}_>[w_>]\), with
		\(\tau'_{N,R}\to0\) as \(R\to\infty\) for fixed \(N\), by replacing the last
		threshold condition with
		\[
		\tau_{N,R}+\tau'_{N,R}+\sigma_{N,R}/c_A(n)
		\le \rho_{\rm res}/2.
		\]
	\end{lemma}
	
	\begin{proof}
		The high-mode classes fix the accounting.  The symbol
		\(\operatorname{Err}_{\rm hi}\) denotes a finite sum of the high-mode pieces
		left after the six coefficient classes have been expanded once and each
		piece has been assigned to exactly one route: favorable sign, scalar
		high-angular conversion, \(R\)-vanishing shear reserve, transition reserve,
		or exponentially weighted tail.  This classification organizes the
		perturbative part of the pulled-back form; it is not a second decomposition of
		the clean separated form \(\operatorname{Sep}_>\).  By definition,
		\(\operatorname{Err}_{\rm hi}\) excludes the pure high-mode
		singular-potential displacement.  This exclusion is legitimate because
		class (2) gives
		\[
		c_l\bigl(\sinh^{-2}(F_h)-\sinh^{-2}s\bigr)|w_l|^2\ge0
		\]
		for \(N\ge N_A(n)\), because \(c_l\ge0\), \(0\le\beta_Rh\), and
		\(F_h=s-\beta_Rh\le s\).  Thus this term is removed before the reserve
		estimate is applied and is not used as an additional positive term in the
		absorption inequality.
		In contrast, the lower-order scalar residuals from the angular Jacobian and
		half-density coefficients have no angular derivative, so the scalar-to-angular
		conversion is legitimate and gives
		\[
		C_{C,L}\lambda_{N+1}^{-1}\operatorname{Ang}[w_{>}].
		\]
		Here the factor \(C_{C,L}\lambda_{N+1}^{-1}\) is one contribution to
		\(\sigma_N\), and it tends to zero when the high cutoff \(N\) is increased.
		The principal shear contribution is instead a kinetic error controlled by
		Lemma~\ref{lem:flattened-shear-form}, before angular projection; it contributes
		vanishing multiples of \(\operatorname{Ang}[w_>]\) and
		\(\operatorname{Sep}_>[w_>]\) to the reserve.  The mollifier-sensitive
		angular Jacobian scalar is resolved by
		Lemma~\ref{lem:angular-jacobian-ibp}, uniformly in the smooth approximation
		to \(h\).
		The remaining high-mode errors come only from multiplicative, transition, and
		Jacobian errors, and are absorbed by the cutoff-small and \(R\)-vanishing reserve before the
		remaining high block is compressed.
		Equivalently, after \(N\) is fixed the high-mode error satisfies
		\[
		|\operatorname{Err}_{\rm hi}[w_>]|
		\leq
		\tau_{N,R}\operatorname{Sep}_>[w_>]
		+(\sigma_N+\sigma_{N,R})\operatorname{Ang}[w_>]
		+o_{N,C,L}(R^{-3})\|w_>\|^2 .
		\]
		Here \(\sigma_N\to0\) as \(N\to\infty\), while
		\(\sigma_{N,R},\tau_{N,R}\to0\) as \(R\to\infty\) for fixed \(N\).
		Since
		\(\operatorname{Ang}[w_>]\le c_A(n)^{-1}\operatorname{Sep}_>[w_>]\)
		by Lemma~\ref{lem:high-mode-angular-reserve}, \(N\) is chosen so large that
		\(\sigma_N/c_A(n)\le\rho_{\rm res}/2\), and then the final \(R\)-threshold is
		enlarged so that
		\[
		\tau_{N,R}+\sigma_{N,R}/c_A(n)\le \rho_{\rm res}/2.
		\]
		Then
		\[
		\rho_{\rm res}\operatorname{Sep}_>[w_>]
		+\operatorname{Err}_{\rm hi}[w_>]
		\ge -o_{N,C,L}(R^{-3})\|w_>\|^2.
		\]
		Only the clean remainder
		\((1-\rho_{\rm res})\operatorname{Sep}_>[w_>]\) is passed to the
		block-form compression step.  This is a scalar decomposition of the
		nonnegative separated high-mode form:
		\[
		\operatorname{Sep}_>
		=
		\rho_{\rm res}\operatorname{Sep}_>
		+(1-\rho_{\rm res})\operatorname{Sep}_>.
		\]
		The first summand is spent only in this reserve absorption estimate.  The
		finite-block scalar compression below uses only the second summand, so the
		high-mode reserve is not reopened or counted twice.
		If the low/high physical cross terms in
		Lemma~\ref{lem:all-degree-lower-form-assembly} contribute an additional
		\(\tau'_{N,R}\operatorname{Sep}_>\), this is treated exactly like
		\(\tau_{N,R}\operatorname{Sep}_>\): for fixed \(N\) the coefficient
		\(\tau'_{N,R}\) tends to zero as \(R\to\infty\), and the last threshold is
		replaced by
		\[
		\tau_{N,R}+\tau'_{N,R}+\sigma_{N,R}/c_A(n)
		\le \rho_{\rm res}/2.
		\]
	\end{proof}
	
	\begin{lemma}[First-branch/radial-complement mixed terms]
		\label{lem:first-radial-complement-mixed}
		Fix \(N,C,L\) and the collar exponent.  Let
		\(p=J_Rf_{\leq}\) be a finite first-branch vector and let \(q_{\rm rad}\) be
		in the corresponding finite angular radial complement, both written in the
		flat unitary representation used above.  For every fixed \(\eta>0\), the
		mixed part of the shifted pulled-back form, after the finite first-branch
		Hadamard block has been separated, satisfies
		\[
		|\operatorname{Mix}_R(p,q_{\rm rad})|
		\le
		\eta\,{\mathcal Q}_{\rm rad}[q_{\rm rad}]
		+o_{N,C,L,\eta}(R^{-3})\|f_{\leq}\|^2
		+o_{N,C,L,\eta}(R^{-3})\|q_{\rm rad}\|^2 .
		\]
		Here \({\mathcal Q}_{\rm rad}\) denotes the positive radial-complement form
		whose spectral margin is
		\[
		{\mathcal Q}_{\rm rad}[q_{\rm rad}]
		\ge c_{\rm rad}R^{-2}\|q_{\rm rad}\|^2
		\]
		after the coefficient losses have been assigned.
	\end{lemma}
	
	\begin{proof}
		In the exact coordinate-unitary flat representation the mass form is the
		reference \(L^2(ds\,d\theta)\) inner product.  Hence the mass cross term
		between \(J_Rf_{\leq}\) and \(q_{\rm rad}\) vanishes exactly.  The
		unperturbed separated shifted form has no first-branch/radial-complement
		cross term either, because \(J_R\) uses the first eigenfunction of each fixed
		degree and \(q_{\rm rad}\) is projected off that branch in the same flat
		fiber.
		
		All remaining mixed terms are perturbative coefficient classes.  For the
		regular coefficient, angular-drift, and shear classes, the first branch
		satisfies, uniformly on the fixed finite angular sector,
		\[
		\int |D_sp|^2\le C_NR^{-2}\|f_{\leq}\|^2,\qquad
		\int\sinh^{-2}(F_h)|p|^2\le C_NR^{-3}\|f_{\leq}\|^2,
		\]
		and its angular form is \(O_N(R^{-3})\|f_{\leq}\|^2\).  The shear coefficient
		has
		\[
		{\beta_R|\nabla h|\over \sinh(F_h)}\le C_{C,L}R^{-1}.
		\]
		Cauchy--Schwarz and Young's inequality therefore route the derivative factor
		on \(q_{\rm rad}\) into \(\eta{\mathcal Q}_{\rm rad}[q_{\rm rad}]\), while
		the first-branch side costs at most
		\[
		C_{\eta,N,C,L}R^{-2}\int |D_sp|^2
		+C_{\eta,N,C,L}R^{-2}\int\sinh^{-2}(F_h)|p|^2
		=
		o_{N,C,L,\eta}(R^{-3})\|f_{\leq}\|^2 .
		\]
		The purely angular first-order drift terms are even smaller on the finite
		first branch, because their coefficients are \(O_L(R^{-1})\) and the
		weighted first-branch scalar or angular forms are \(O_N(R^{-3})\).
		
		It remains to check the endpoint coefficient classes, where pointwise
		smallness is false.  In the critical channel the finite first-branch Jost
		bounds give, on \(0<s\le2\eps_R\) and then on the fixed collar,
		\[
		|p(s,\theta)|+|sD_sp(s,\theta)|
		\le C_NR^{-3/2}s^{1/2}\|f_{\leq}\|.
		\]
		The projected Green estimates of Lemma~\ref{lem:critical-endpoint-jost-green}
		give
		\[
		R^{-1}\int_{\eps_R}^{1}s^{-2}|q_{\rm rad}|^2=o_R(1){\mathcal Q}_{\rm rad}[q_{\rm rad}],
		\qquad
		R^{-2}\int_{\eps_R}^{1}|D_sq_{\rm rad}|^2=o_R(1){\mathcal Q}_{\rm rad}[q_{\rm rad}]
		\]
		in the critical projected complement, and stronger estimates hold in the
		noncritical finite set of degrees.  Thus a typical scalar cross term with
		coefficient \(O(R^{-1}s^{-2})\) is bounded by
		\[
		\begin{aligned}
			R^{-1}\int_{\eps_R}^{1}s^{-2}|p\,q_{\rm rad}|
			&\le
			\left(R^{-1}\int_{\eps_R}^{1}s^{-2}|q_{\rm rad}|^2\right)^{1/2}
			\left(R^{-1}\int_{\eps_R}^{1}s^{-2}|p|^2\right)^{1/2}  \\
			&\le
			o_R(1)^{1/2}{\mathcal Q}_{\rm rad}[q_{\rm rad}]^{1/2}
			\,O_N(R^{-2}\sqrt{\log R})\|f_{\leq}\|.
		\end{aligned}
		\]
		Young's inequality makes this
		\[
		\eta{\mathcal Q}_{\rm rad}[q_{\rm rad}]
		+o_{N,C,L,\eta}(R^{-3})\|f_{\leq}\|^2 .
		\]
		The transition terms are the same calculation on
		\([\eps_R,2\eps_R]\), using
		\(\int_{I_R}|q_{\rm rad}|^2\le C\eps_R{\mathcal Q}_{\rm rad}[q_{\rm rad}]\)
		for scalar transition coefficients.  For mixed transition terms with a
		derivative on \(q_{\rm rad}\), no integration by parts is needed.  Since
		\(|G_s|\le C_CR^{-1}\eps_R^{-1}\) on \(I_R\),
		\[
		\begin{aligned}
			\int_{I_R}|G_s p\,D_sq_{\rm rad}|
			&\le
			C_CR^{-1}\eps_R^{-1}
			\left(\int_{I_R}|p|^2\right)^{1/2}
			\left(\int_{I_R}|D_sq_{\rm rad}|^2\right)^{1/2}  \\
			&\le
			o_R(1)^{1/2}R^{-3/2}
			\|f_{\leq}\|\,{\mathcal Q}_{\rm rad}[q_{\rm rad}]^{1/2}.
		\end{aligned}
		\]
		Here
		\(\int_{I_R}|p|^2\le C_NR^{-3}\eps_R^2\|f_{\leq}\|^2\), while
		Lemma~\ref{lem:critical-endpoint-jost-green} gives
		\(\int_{I_R}|D_sq_{\rm rad}|^2\le o_R(1)R^2{\mathcal Q}_{\rm rad}[q_{\rm rad}]\)
		in the critical channel, with stronger estimates in the noncritical finite
		set.  Young's inequality again gives
		\[
		\eta{\mathcal Q}_{\rm rad}[q_{\rm rad}]
		+o_{N,C,L,\eta}(R^{-3})\|f_{\leq}\|^2.
		\]
		The conjugated difference \(D_s-\partial_s=O(s^{-1})\) is absorbed into the
		same scalar transition estimates above.  The regular tail \(s\ge1\) carries
		an exponentially decaying weight and is controlled by the same Young splitting
		with an \(o(R^{-3})\) first-branch cost.
		
		Choosing \(\eta\) as part of the fixed radial-complement reserve, absorbing
		the finite first-branch \(o(R^{-3})\) cost into the trailing remainder, and
		then enlarging \(R_0\) proves the stated mixed estimate.
	\end{proof}
	
	\begin{lemma}[All-degree lower-form assembly]
		\label{lem:all-degree-lower-form-assembly}
		With the constants chosen in the order recorded in
		Subsection~\ref{sec:constants}, write
		\[
		w=J_Rf_{\leq}+q_{\rm rad,\leq}+w_{>},\qquad f_{\leq}=P_Nf .
		\]
		Then, uniformly for \(0\le h\le C\) and \(\Lip(h)\le L\), the pulled-back
		shifted physical Dirichlet form satisfies
		\[
		Q_h[w]-E_R\|w\|^2
		\geq
		R^{-3}B_N^h[f_{\leq},w_{>}]
		-(\delta_{\rm eff}/32)R^{-3}\|f_{\leq}\|^2
		+c_{\rm rad}R^{-2}\|q_{\rm rad,\leq}\|^2
		-o_{N,C,L}(R^{-3})\|w\|^2,
		\]
		where
		\[
		B_N^h[f_{\leq},w_{>}]
		=
		\langle f_{\leq},P_NL_hP_Nf_{\leq}\rangle
		+R^3(1-\rho_{\rm res})\operatorname{Sep}_>[w_{>}]
		-\frac{M_C^2}{4\eta_{\rm gap}}\|w_>\|^2.
		\]
	\end{lemma}
	
	\begin{proof}
		Start from the exact physical coordinate pullback
		\(r=F_h(s,\theta)=s-\beta_R(s)h(\theta)\), followed by the fixed radial
		Schrodinger half-density conjugation.  The radial and angular gradient
		principal parts are given by Lemma~\ref{lem:flattened-shear-form}; the
		remaining metric, Jacobian, half-density, and singular-potential differences
		are precisely the six coefficient classes in
		Definition~\ref{def:routed-coefficient-classes}.  The finite
		first-radial block \(J_RP_Nf\) is treated as a physical finite-dimensional
		Hadamard problem, not by termwise separated-coordinate estimates.  The
		exact-form conclusion of Lemma~\ref{lem:finite-first-branch-hadamard} gives
		the finite low block
		\[
		R^{-3}\langle f_{\leq},P_N(T_n+2\pi^2h+b_n)P_Nf_{\leq}\rangle
		+o_{N,C,L}(R^{-3})\|f_{\leq}\|^2 .
		\]
		
		Terms with one factor in the finite first branch \(J_RP_Nf\) and one factor
		in the high block \(w_>\) are not identified with the angular multiplication
		matrix \(2\pi^2h\).  The finite Hadamard lemma gives that coefficient only
		for pairs of degrees \(\le N\).  Instead these low/high physical cross terms
		are estimated in the pulled-back form before compression.  The genuinely
		order-\(R^{-3}\) part is the bounded multiplication coupling generated by
		the boundary-height coefficient on the first-radial part of the high angular
		sector.  Writing, only for this estimate, the high block as
		\(w_>=J_R^{>}g_>+q_>\), where \(J_R^{>}\) is the degree-adapted first-radial
		lift in degrees \(>N\), this leading part has the form
		\[
		2\pi^2R^{-3}\langle f_{\leq},h\,g_>\rangle_{L^2(\SSph^{n-1})},
		\]
		up to finite-\(R\) normalization errors.  Those errors are only the mismatch
		between the exact degree-dependent first-branch boundary slopes and the common
		\(2\pi^2R^{-3}\) flux coefficient, together with the difference between the
		flattened \(L^2\) normalization of \(J_R^{>}g_>\) and the angular norm of
		\(g_>\).  The fixed-degree boundary-slope asymptotic makes the former
		\(o_{N,C,L}(R^{-3})\|f_{\leq}\|\,\|g_>\|\) for the finite side, and the latter
		is an \(R\)-vanishing perturbation of the high separated form after
		Cauchy--Schwarz.  Hence both are assigned to the same finite-low
		\(o_{N,C,L}(R^{-3})\) cost and high-side
		\(\tau'_{N,R}\operatorname{Sep}_>\) reserve as the other non-leading
		low/high classes.  The high radial-complement part
		\(q_>\), transition pieces, and the non-leading boundary-trace corrections
		are likewise among the remaining low/high classes routed below into
		\(\tau'_{N,R}\operatorname{Sep}_>\) and the finite-low
		\(o_{N,C,L}(R^{-3})\) remainder.  Since \(\|M_h\|_{2\to2}\le C\), the choice
		\(M_C=4\pi^2C+1\) in Subsection~\ref{sec:constants} dominates the displayed
		bounded multiplication coupling after the final enlargement of \(R_0\).
		Hence it is bounded below
		by
		\[
		-M_CR^{-3}\|f_{\leq}\|\,\|w_>\|.
		\]
		Young's inequality with parameter \(\eta_{\rm gap}\) gives
		\[
		-M_CR^{-3}\|f_{\leq}\|\,\|w_>\|
		\ge
		-\eta_{\rm gap}R^{-3}\|f_{\leq}\|^2
		-\frac{M_C^2}{4\eta_{\rm gap}}R^{-3}\|w_>\|^2 .
		\]
		This is the fixed low/high Schur cost recorded in \(B_N^h\); it is not an
		\(o(R^{-3})\) term.
		
		The remaining regular, shear, Jacobian, and transition low/high classes have
		an additional small coefficient after the leading bounded multiplication
		coupling has been separated.  The finite-low first branch has
		\[
		\int |D_sJ_RP_Nf|^2\le C_NR^{-2}\|f_{\leq}\|^2,\qquad
		\int\sinh^{-2}(F_h)|J_RP_Nf|^2\le C_NR^{-3}\|f_{\leq}\|^2,
		\]
		and finite angular form \(O_NR^{-3}\|f_{\leq}\|^2\).  Applying
		Cauchy--Schwarz and Young's inequality to the regular, shear, Jacobian, and
		transition coefficient classes after the fixed bounded-multiplication part
		has been removed gives a high-side contribution of the form
		\[
		\tau'_{N,R}\operatorname{Sep}_>[w_>],
		\qquad \tau'_{N,R}\to0\quad(R\to\infty,\ N\text{ fixed}),
		\]
		and a low-side contribution
		\[
		o_{N,C,L}(R^{-3})\|f_{\leq}\|^2 .
		\]
		The high-side term is folded into the reserve absorption by using the
		strengthened threshold
		\[
		\tau_{N,R}+\tau'_{N,R}+\sigma_{N,R}/c_A(n)
		\le \rho_{\rm res}/2
		\]
		from Lemma~\ref{lem:high-mode-error-routing}; this is allowed after
		enlarging \(R_0\) because all three quantities tend to zero for fixed
		\(N\).  The low-side term is charged to the trailing
		\(o_{N,C,L}(R^{-3})\) remainder, not to an additional fixed fraction of
		\(\delta_{\rm eff}\).  The only fixed low-side loss from the low/high
		coupling is the Young loss
		\(\eta_{\rm gap}R^{-3}\|f_{\leq}\|^2
		=(\delta_{\rm eff}/32)R^{-3}\|f_{\leq}\|^2\) already displayed in the
		statement of the lemma.  The endpoint pieces use the same first-branch endpoint
		bounds and projected endpoint controls as
		Lemma~\ref{lem:first-radial-complement-mixed}; the high angular factor is
		controlled by Lemma~\ref{lem:high-mode-angular-reserve}.
		
		On the pure high block, the nonnegative singular-potential displacement is
		discarded, and the remaining high errors are absorbed by
		Lemma~\ref{lem:high-mode-error-routing}.  This spends only the reserved
		\(\rho_{\rm res}\operatorname{Sep}_>\) summand and leaves the clean
		\((1-\rho_{\rm res})\operatorname{Sep}_>\) term appearing in \(B_N^h\).
		On finite radial complements and first-branch/radial-complement mixed terms,
		Lemma~\ref{lem:critical-endpoint-jost-green} and
		Lemma~\ref{lem:first-radial-complement-mixed} route the critical endpoint
		mass and derivative contributions into a fixed radial-complement reserve;
		after this reserve is removed the remaining radial-complement form has the
		post-reservation lower bound
		\({\mathcal Q}_{\rm rad}\ge c_{\rm rad}R^{-2}\|q_{\rm rad,\leq}\|^2\).
		
		Summing the six routed classes, the finite Hadamard low block, the high
		reserve absorption, and the radial-complement mixed estimate gives the
		displayed lower bound.  All constants have already been fixed before the
		final enlargement of \(R_0\), so the remaining fixed-\(N\) and fixed-parameter
		remainders are \(o_{N,C,L}(R^{-3})\).
	\end{proof}
	
	\begin{lemma}[Finite-block compression]
		\label{lem:finite-block-compression}
		With the constants chosen as above, the scalar compression of the high angular
		block gives
		\[
		\nu_2(\widetilde B_N^h)\ge \mu_2(L_h)
		\]
		uniformly for \(0\le h\le C\) and \(\Lip(h)\le L\), after the final
		enlargement of \(R_0\).
	\end{lemma}
	
	\begin{proof}
		After the physical low/high cross terms have already been Young-routed in
		Lemma~\ref{lem:all-degree-lower-form-assembly}, the remaining high block is
		compressed by a scalar lower bound.  Let \(\widetilde B_N^h\) be the
		block-diagonal form obtained from \(B_N^h\) after this scalar compression.
		Set
		\[
		\Gamma_N
		=
		(1-\rho_{\rm res})(b_n+\tau_{n,N+1})
		-\frac{M_C^2}{4\eta_{\rm gap}}.
		\]
		After the high-mode compression estimate below, \(\widetilde B_N^h\) is the
		block-diagonal form on
		\[
		P_NL^2(\SSph^{n-1})\oplus {\mathcal H}_{>,R},
		\]
		where \({\mathcal H}_{>,R}\) denotes the flattened fixed-cylinder Hilbert
		subspace of angular degrees \(>N\), with the same Dirichlet--Friedrichs
		endpoint conditions as in the high separated form.  It is given by
		\[
		\widetilde B_N^h[g_{\leq}+z_{>}]
		=
		\langle g_{\leq},P_NL_hP_Ng_{\leq}\rangle
		+\Gamma_N\|z_{>}\|^2.
		\]
		Thus the full high cylinder block has been replaced by the scalar level
		\(\Gamma_N\), repeated with the multiplicity of \({\mathcal H}_{>,R}\); no
		identification of \(z_>\) with a purely angular function is being made.
		This step does not reopen \(\operatorname{Sep}_>\), does not reabsorb
		class-2, class-4, class-6, or low/high physical cross errors, and uses only
		the already reduced scalar high coefficient
		\((1-\rho_{\rm res})(b_n+\tau_{n,N+1})\), minus the fixed Young penalty
		\(M_C^2/(4\eta_{\rm gap})\) from the bounded low/high first-branch coupling.
		The positive coefficient comes from the single clean summand
		\((1-\rho_{\rm res})\operatorname{Sep}_>\), not from a second use of the
		reserve summand.  Indeed, after the raw high form has been split as
		\[
		\operatorname{Sep}_>
		=
		\rho_{\rm res}\operatorname{Sep}_>
		+(1-\rho_{\rm res})\operatorname{Sep}_>,
		\]
		the first summand is used only in Lemma~\ref{lem:high-mode-error-routing}.
		This is a scalar split of the full nonnegative quadratic-form value, not a
		decomposition into separate kinetic and centrifugal subenergies.  Consequently
		the reserve inequality and the compression floor may use different lower
		bounds for \(\operatorname{Sep}_>\), because they are applied to the disjoint
		scalar multiples \(\rho_{\rm res}\operatorname{Sep}_>\) and
		\((1-\rho_{\rm res})\operatorname{Sep}_>\), respectively.
		For the second summand, monotonicity in angular degree and the favorable
		inequality \(\sinh^{-2}(F_h)\ge\sinh^{-2}s\) give, for \(l>N\),
		\[
		\begin{aligned}
			\operatorname{Sep}_{l}[w_l]
			&=
			\int\left(|\partial_sw_l|^2-\pi^2R^{-2}|w_l|^2\right)
			+c_l\int\sinh^{-2}(F_h)|w_l|^2  \\
			&\ge
			\int\left(|\partial_sw_l|^2-\pi^2R^{-2}|w_l|^2\right)
			+c_{N+1}\int\sinh^{-2}(s)|w_l|^2  \\
			&=
			\langle w_l,(H_{N+1,R}-E_R)w_l\rangle  \\
			&\ge
			\bigl((b_n+\tau_{n,N+1})R^{-3}+o_N(R^{-3})\bigr)\|w_l\|^2 .
		\end{aligned}
		\]
		The last inequality is the Rayleigh--Ritz lower bound for the fixed-degree
		operator \(H_{N+1,R}-E_R\), together with the fixed-\(N\) first-eigenvalue
		asymptotic; no decomposition of \(w_l\) into radial eigenfunctions is needed.
		Here \(c_l\) is increasing for \(l\ge1\), \(N\ge N_A(n)\ge1\), and \(N\) is
		fixed before \(R_0\).  The asymptotic in this last line is the fixed-degree
		\(N+1\) ball asymptotic.  Since \(b_n+\tau_{n,N+1}=b_{n+2(N+1)}\), this is
		exactly the
		degree-\(N+1\) shifted coefficient.  Thus compression uses only
		\[
		(1-\rho_{\rm res})
		\bigl((b_n+\tau_{n,N+1})R^{-3}+o_N(R^{-3})\bigr)\|w_>\|^2 .
		\]
		All radial-bottom and centrifugal contributions in this scalar floor have
		already been discounted by the factor \(1-\rho_{\rm res}\).
		Here the compression estimate is fixed by the same choice of \(N\).  Let
		\[
		U_{n,C}=b_n+\tau_{n,1}+2\pi^2C .
		\]
		The two-dimensional space spanned by a constant spherical harmonic and any
		degree-one harmonic gives, uniformly for \(0\le h\le C\),
		\[
		\mu_2(L_h)\le U_{n,C}.
		\]
		Since \(\tau_{n,N+1}=b_{n+2(N+1)}-b_n\to\infty\), the cutoff \(N\) is chosen
		large enough, after \(\eta_{\rm gap}\) and \(\rho_{\rm res}\) are fixed, that
		\[
		(1-\rho_{\rm res})(b_n+\tau_{n,N+1})
		-\frac{M_C^2}{4\eta_{\rm gap}}
		\ge U_{n,C}+\delta_{\rm eff}/8 .
		\]
		This is a simultaneous finite list of requirements on \(N\): we enlarge the
		same cutoff so that \(N\ge N_A(n)\), \(\sigma_N/c_A(n)\le\rho_{\rm res}/2\),
		the strengthened displayed compression inequality holds, and the finite-block
		ground-state energy approximation recorded in Subsection~\ref{sec:constants}
		holds.  The later threshold \(R_0\) then handles only the \(R\)-dependent
		remainders \(\sigma_{N,R}\), \(\tau_{N,R}\), and the fixed-\(N\) asymptotic
		errors.
		For this fixed \(N\), the scalar compression estimate has an additional
		\(o_N(1)\) contribution after multiplication by \(R^3\).  The norm
		\(\|w_>\|\) is the flattened Schrodinger \(L^2(ds\,d\theta)\) norm used in
		the high separated form.  After one further enlargement of \(R_0\), the
		fixed-\(N\) scalar error is bounded below by \(-\delta_{\rm eff}/16\).  Hence
		the actual compressed high block is above \(U_{n,C}+\delta_{\rm eff}/16\),
		and therefore above
		\(\mu_2(L_h)+\delta_{\rm eff}/16\).
		On the low block, min-max for the restricted form gives
		\[
		\mu_2(P_NL_hP_N|_{P_NL^2})\ge \mu_2(L_h),
		\]
		because the form domain has been restricted.  The form
		\(\widetilde B_N^h\) is block diagonal on
		\[
		P_NL^2(\SSph^{n-1})\oplus {\mathcal H}_{>,R},
		\]
		so its eigenvalue list is the union of the eigenvalues of the low block
		\(P_NL_hP_N|_{P_NL^2}\) and the scalar high level \(\Gamma_N\), with
		\(\Gamma_N\) repeated with the multiplicity of \({\mathcal H}_{>,R}\).  The
		compression-floor
		choice above gives
		\[
		\Gamma_N\ge U_{n,C}+\delta_{\rm eff}/16
		\ge \mu_2(L_h)+\delta_{\rm eff}/16 .
		\]
		Therefore
		\[
		\nu_2(\widetilde B_N^h)
		\ge
		\min\{\mu_2(P_NL_hP_N|_{P_NL^2}),\Gamma_N\}
		\ge \mu_2(L_h).
		\]
		The low-block Young loss is not included in \(\widetilde B_N^h\); it remains
		the single external \(\eta_{\rm gap}\) loss in
		Lemma~\ref{lem:all-degree-lower-form-assembly}.  The compression margin
		\(\delta_{\rm eff}/16\) is used only to keep the high scalar block away from
		the low spectral window.
	\end{proof}
	
	\paragraph{Ground-state upper trial.}
	The ground-state upper trial uses the same fixed angular cutoff \(N\):
	\(f_{1,N}\in P_NL^2(\mathbb S^{n-1})\) is the normalized compressed
	effective ground state.  By the choice of \(N\),
	\[
	\langle f_{1,N},L_hf_{1,N}\rangle
	=
	\mu_1(P_NL_hP_N|_{P_NL^2})
	\le \mu_1(L_h)+\delta_{\rm eff}/32 .
	\]
	The trial is the degree-adapted lift \(w_{\le}=J_R f_{1,N}\).  The isometry
	of \(J_R\) gives \(\|w_{\le}\|=\|f_{1,N}\|=1\), so the additive shifted-form
	estimate below is exactly the Rayleigh quotient estimate used to bound
	\(\lambda_1\).  Since
	\(P_NL^2(\mathbb S^{n-1})\) is finite-dimensional, all angular derivatives
	of \(f_{1,N}\) needed in the finite first-branch estimate are bounded by
	constants depending on \(N\) and \(n\), not on \(R\).  The exact shifted
	pulled-back form is evaluated directly for this finite trial.  By the exact
	finite-block shifted-form conclusion of
	Lemma~\ref{lem:finite-first-branch-hadamard},
	\[
	Q_h[w_{\le}]-E_R\|w_{\le}\|^2
	=
	R^{-3}
	\langle f_{1,N},P_N(T_n+2\pi^2h+b_n)P_Nf_{1,N}\rangle
	+o_{N,C,L,\delta_{\rm cut}}(R^{-3}).
	\]
	Because \(L_h=T_n+2\pi^2h+b_n\), the upper trial obtains the full operator
	energy without using radial-complement routing or the high-mode reserve.
	
	\subsection{The Lipschitz Passage}
	
	\begin{lemma}[Offset-domain spectral squeeze]
		\label{lem:offset-squeeze}
		Fix \(R>C>0\), and let \(h:\SSph^{n-1}\to[0,C]\) be Lipschitz.  If
		\(h_j:\SSph^{n-1}\to[0,C]\) are continuous and converge uniformly to \(h\),
		then, for the radial-height domains
		\[
		\Om=\{0\leq \rho<R-h(\theta)\},\qquad
		\Om_j=\{0\leq \rho<R-h_j(\theta)\},
		\]
		one has
		\[
		\lambda_k(\Om_j)\to\lambda_k(\Om)
		\]
		for every fixed \(k\ge1\), with Dirichlet eigenvalues counted with
		multiplicity.
	\end{lemma}
	
	\begin{proof}
		Let \(\eps_j=\|h_j-h\|_\infty\).  Define the offset radial domains
		\[
		\Om^-(\eps)=\{0\leq \rho<R-h(\theta)-\eps\},
		\qquad
		\Om^+(\eps)=\{0\leq \rho<R-h(\theta)+\eps\}.
		\]
		Then \(\eps_j\to0\).  Discarding finitely many indices, assume
		\(\eps_j<\min\{R-C,1\}\), so the inner offsets are nonempty and all offset
		domains lie in \(B_{R+1}(o)\).  Daners is applied to the monotone
		one-parameter offset families as \(\eps\downarrow0\), and the resulting
		convergence holds along the sequence \(\eps=\eps_j\).  The monotonicity here
		is in the continuous parameter \(\eps\); after Daners gives convergence as
		\(\eps\downarrow0\), we evaluate that convergence on the particular sequence
		\(\eps_j=\|h_j-h\|_\infty\).
		Because \(h-\eps_j\leq h_j\leq h+\eps_j\),
		\[
		\Om^-(\eps_j)\subset \Om_j\subset \Om^+(\eps_j).
		\]
		To invoke Daners' Euclidean varying-domain results, freeze this value of
		\(R\) and use the Poincar\'e ball model with \(o\) mapped to the Euclidean
		origin.  Then \(B_{R+1}(o)\) is the Euclidean ball
		\[
		|x|<\tanh((R+1)/2)<1,
		\]
		and the hyperbolic metric is conformal to the Euclidean metric with conformal
		factor \(4(1-|x|^2)^{-2}\).  On the compact closure of this fixed Euclidean
		ball the coefficients are smooth, bounded, and uniformly elliptic; there is
		no polar-coordinate singularity at \(o\).  The ellipticity constants may
		depend on this frozen \(R\), which is harmless because only the limit
		\(j\to\infty\) is being taken in the domain squeeze.  If
		\(w=\sqrt{\det g}\) is the hyperbolic volume density in this chart, the map
		\(U\psi=w^{1/2}\psi\) is a unitary identification from \(L^2(dV_{\HH^n})\)
		to Euclidean \(L^2(dx)\).  Since \(w^{1/2}\) is smooth, bounded, and bounded
		away from zero on the fixed chart, this unitary identifies the Dirichlet form
		domain with the Euclidean \(H_0^1\) domain and preserves the Dirichlet
		spectrum.  It conjugates the Dirichlet hyperbolic Laplacian to a symmetric
		uniformly elliptic operator
		\[
		-\operatorname{div}(A\nabla\cdot)+V
		\]
		with \(A\) smooth and uniformly elliptic and \(V\) smooth and bounded on the
		fixed chart.  The domain-theoretic point is the Mosco convergence of the
		Dirichlet spaces \(H_0^1(\Om^\pm(\eps))\) for this one fixed uniformly
		elliptic form.  The coefficients do not vary with \(\eps\), and the graph norm
		of the principal part is equivalent to the standard \(H_0^1\) norm on every
		offset domain.  Daners' monotone convergence and stability hypotheses
		therefore give uniform-resolvent convergence for the fixed principal form; the
		bounded symmetric potential term \(\int V|u|^2\) is form-continuous on \(L^2\)
		and is carried through the same resolvent convergence after shifting by a
		constant if needed.  The eigenvalue convergence used below is then the
		standard finite-eigenvalue consequence of this resolvent convergence, via
		Daners' spectral-continuity corollaries.  The value of \(R\) is fixed
		throughout this domain-convergence step.  In the
		Poincar\'e chart the radial graph \(\rho=R-g(\theta)\) has Euclidean radius
		\[
		\varrho_g(\theta)=\tanh((R-g(\theta))/2),
		\]
		and hence
		\[
		|\nabla_{\SSph^{n-1}}\varrho_g|
		\le
		\frac12 \operatorname{sech}^2((R-C-\eps_0)/2)\,
		\Lip(g)
		\]
		whenever \(-\eps_0\le g\le C+\eps_0\).  Thus the inner and outer offset radial
		graphs, corresponding to \(g=h+\eps\) and \(g=h-\eps\), have Euclidean
		Lipschitz constants bounded uniformly for
		\(0<\eps<\eps_0\).  In particular, both offset families are bounded
		Lipschitz domains in the fixed chart with a uniform Lipschitz bound; this is
		the regularity input used in the decreasing-domain convergence step below.
		
		The inner domains \(\Om^-(\eps)\) increase to \(\Om\), so Daners'
		monotone-inside criterion \cite[Proposition 7.1]{Daners}, followed by
		uniform resolvent convergence and finite-eigenvalue continuity
		\cite[Corollaries 4.7, 4.2 and Remark 4.3]{Daners}, gives convergence of the
		indexed eigenvalues
		\[
		\lambda_k(\Om^-(\eps))\to\lambda_k(\Om)
		\]
		for every fixed \(k\).  For the decreasing outer family, the domains
		\(\Om^+(\eps)\) decrease in the fixed chart and have uniformly bounded
		Lipschitz constants as noted above.  Since \(h\) is Lipschitz and \(R>C\), the
		open limit domain
		\[
		\Om=\{0\le\rho<R-h(\theta)\}
		\]
		is a bounded Lipschitz domain in this chart.  Hence its closure
		\(\overline\Om\) has Lipschitz boundary, satisfies the segment property, and
		is stable in Daners' sense by the Lipschitz-domain stability criterion
		\cite[Proposition 7.3(1) and the following sentence]{Daners}.  Moreover
		\[
		\operatorname{int}\left(\bigcap_{\eps>0}\Om^+(\eps)\right)=\Om
		\]
		because \(h\) is continuous.  Therefore Daners' outside convergence criterion,
		applied to the decreasing outer family with open limit \(\Om\), stable closed
		limit \(\overline\Om\), the displayed interior-intersection identity, the
		bounded-domain uniform-resolvent corollary, and the finite-eigenvalue
		continuity result
		\cite[Proposition 7.4, Proposition 7.3(1), Corollaries 4.7, 4.2 and
		Remark 4.3]{Daners} give
		\[
		\lambda_k(\Om^+(\eps))\to\lambda_k(\Om)
		\]
		for every fixed \(k\).  The last step is indexed eigenvalue convergence
		counting multiplicity: Corollary 4.2 gives persistence of the corresponding
		finite spectral projections under uniform resolvent convergence, and
		Remark 4.3 states the resulting continuity of each finite eigenvalue system.
		Thus the assertion applies to \(k=2\) even if eigenvalue multiplicities cross.
		In particular, along \(\eps=\eps_j\),
		\[
		\lambda_k(\Om^-(\eps_j))\to\lambda_k(\Om),
		\qquad
		\lambda_k(\Om^+(\eps_j))\to\lambda_k(\Om).
		\]
		Domain monotonicity of indexed Dirichlet eigenvalues in the inclusions above
		gives
		\[
		\lambda_k(\Om^+(\eps_j))
		\leq \lambda_k(\Om_j)
		\leq \lambda_k(\Om^-(\eps_j)),
		\]
		and hence \(\lambda_k(\Om_j)\to\lambda_k(\Om)\).
	\end{proof}
	
	\subsection{Conclusion of the radial-height proof}
	
	By Lemma~\ref{lem:all-degree-lower-form-assembly}, and with constants chosen
	in the order
	\[
	\delta_{\rm eff}\to \eta_{\rm gap}\to N_A\to c_A\to \rho_{\rm res}
	\to c_{\rm rad}^0,R_{\rm rad}\to c_{\rm rad}\to \eta_{\rm form}
	\to \delta_{\rm cut}\to M_C\to N\to R_0,
	\]
	the shifted form is bounded below by the sum of the finite low block, the
	finite radial-complement block, and the high angular block, up to the single
	displayed low Young loss and the \(o_{N,C,L}(R^{-3})\) remainder.  The
	compression step replaces the high block by the scalar block
	\(\Gamma_N\|w_>\|^2\), so the min-max principle may be applied to this
	assembled lower form before the final subtraction.  Thus it remains only to
	locate the second level of the lower form.  After
	the choice of \(N\) and the subsequent enlargement of \(R_0\), the scalar
	compression after the Young split leaves the high block above the common
	energy offset by
	\[
	E_R+R^{-3}\bigl(U_{n,C}+\delta_{\rm eff}/16\bigr),
	\]
	while the radial-complement block lies above the first radial level in the
	same angular degree by at least \(c_{\rm rad}R^{-2}\).  This comparison is
	uniform in the angular degree by Lemma~\ref{lem:radial-branch-gap}.  The first
	radial level is nondecreasing in angular degree, since the centrifugal
	coefficient \(c_l\) is nondecreasing in \(l\).  Hence
	\[
	\lambda_{1,l,R}\ge \lambda_{1,0,R}
	=E_R+b_nR^{-3}+o_n(R^{-3})
	\]
	for every \(l\), and the radial-complement block lies above
	\[
	E_R+c_{\rm rad}R^{-2}-|b_n|R^{-3}-o_n(R^{-3}).
	\]
	The two-dimensional trial bound gives \(\mu_2(L_h)\le U_{n,C}\).  After
	enlarging \(R_0\) past \((U_{n,C}+1+|b_n|)/c_{\rm rad}\) and absorbing the
	fixed-dimensional remainder,
	\[
	E_R+c_{\rm rad}R^{-2}-|b_n|R^{-3}-o_n(R^{-3})
	>
	E_R+R^{-3}(U_{n,C}+1).
	\]
	Thus neither the high angular block nor the radial complement can create the
	second min-max level below the compressed low block;
	the \(o_{N,C,L}(R^{-3})\) remainder is absorbed after the same final enlargement
	of \(R_0\).  Therefore the min-max principle and the lower form give
	\[
	\lambda_2(\Om_R(h))
	\geq E_R+R^{-3}\nu_2(\widetilde B_N^h)
	-(\delta_{\rm eff}/32)R^{-3}
	-o_{N,C,L}(R^{-3}).
	\]
	The explicit \(\delta_{\rm eff}/32\) loss here is the fixed low side of the
	bounded low/high Young split.  The finite-low coefficient remainders are
	\(o_{N,C,L}(R^{-3})\), and the high-mode reserve coefficients
	\(\sigma_N,\sigma_{N,R},\tau_{N,R}\) and \(\tau'_{N,R}\) are made small in
	the \(N\)-then-\(R\) choices and are absorbed before compression, so no
	unabsorbed high-mode loss remains in the displayed finite block.
	The ground-state upper trial gives
	\[
	\lambda_1(\Om_R(h))
	\leq E_R+R^{-3}\bigl[\mu_1(L_h)+\delta_{\rm eff}/32\bigr]
	+o_{N,C,L}(R^{-3}).
	\]
	Lemma~\ref{lem:finite-block-compression} gives
	\[
	\nu_2(\widetilde B_N^h)
	\ge \mu_2(L_h).
	\]
	Subtracting the upper trial from the lower form and then using this
	compression estimate gives the explicit loss display
	\[
	\begin{aligned}
		\lambda_2-\lambda_1
		&\geq
		R^{-3}\bigl[\mu_2(L_h)-\mu_1(L_h)\bigr]
		-(\delta_{\rm eff}/32)R^{-3}
		-(\delta_{\rm eff}/32)R^{-3}
		-o_{N,C,L}(R^{-3})  \\
		&\geq (15\delta_{\rm eff}/16)R^{-3}-o_{N,C,L}(R^{-3}).
	\end{aligned}
	\]
	The two displayed losses are, respectively, the external
	\(\eta_{\rm gap}=\delta_{\rm eff}/32\) low-side Young loss from the bounded
	low/high first-branch splitting and the ground-state upper-trial slack.  The
	finite-low coefficient remainders are \(o_{N,C,L}(R^{-3})\), while the
	radial-complement and high-mode coefficient losses have already been absorbed
	by the post-reservation radial margin and by the high-mode reserve before
	compression.
	After increasing \(R_0\), this is at least
	\[
	(\delta_{\rm eff}/4)R^{-3}.
	\]
	
	This proves the asserted gap for smooth heights satisfying
	\(0\le h\le C\) and \(\Lip(h)\le L\), with constants depending on \(n,C,L\)
	only through the choices above and with final gap constant
	\(c(n,C)=\delta_{\rm eff}(n,C)/4\).  For a Lipschitz height \(h\), choose a
	positive smooth approximate identity \(K_j\) on \(SO(n)\) and set
	\[
	h_j(\theta)=\int_{SO(n)}h(A\theta)K_j(A)\,dA .
	\]
	Then \(h_j\in C^\infty(\SSph^{n-1})\), \(h_j\to h\) uniformly,
	\[
	0\le h_j\le C,\qquad \Lip(h_j)\le L,
	\]
	because rotations are isometries of the sphere and \(K_j\) is positive with
	unit mass.  The smooth estimate applies to \(\Om_R(h_j)\) with the same
	\(c(n,C)\) and \(R_0(n,C,L)\).  Lemma~\ref{lem:offset-squeeze} gives
	\(\lambda_k(\Om_R(h_j))\to\lambda_k(\Om_R(h))\) for \(k=1,2\), and passing to
	the limit proves Theorem~\ref{thm:sh-lip}.
	
	\section{Discussion and Further Questions}
	
	The theorem answers the large-diameter rate question for compact horoconvex
	domains by supplying a lower bound with the same \(D^{-3}\) power as the
	Nguyen--Stancu--Wei upper construction.
	Khan--Tuerkoen already give a qualitative diameter-dependent lower bound for
	horoconvex domains, but their displayed large-diameter rate is much smaller.
	The present proof identifies the polynomial scale by replacing the domain with
	a fixed-width radial-height problem and proving a uniform gap for the limiting
	angular family in all dimensions.  A direct ball comparison gives a shorter
	independent proof in dimensions \(n=2,3\).
	This section records the scope of the theorem and the remaining questions.
	
	The theorem gives the lower bound at the same diameter scale as the
	Nguyen--Stancu--Wei examples.  Earlier lower bounds for horoconvex domains
	show positivity, but their explicit dependence on the diameter is much smaller
	in the large-diameter limit.  The present argument rules out a gap below a
	fixed multiple of \(D^{-3}\) for large horoconvex domains.  It leaves the sharp
	constant and extremal geometry open.
	
	The theorem is a large-diameter rate bound for horoconvex domains in
	\(\mathbb H^n\).  General geodesic convexity and first-eigenfunction
	log-concavity remain separate problems.  The proof uses the geodesic ball as
	calibration, while a general large horoconvex domain retains the angular
	boundary height in the limiting problem.
	
	The proof does not identify the sharp constant.  The constant produced here
	comes from the compactness gap of the effective angular operator
	\(T_n+2\pi^2h+b_n\) over the bounded height class \(0\le h\le2\log 2\), and
	from several nonsharp reserves in the radial-height theorem.
	Thus the result should be read as a scale theorem.  It proves that the
	diameter decay cannot be worse than order \(D^{-3}\) for large horoconvex
	domains, but it does not identify the infimum of
	\(D^3(\lambda_2-\lambda_1)\) in the large-diameter limit.
	
	Ball asymptotics fix the expected scale and provide the calibration.  The
	scalar radius comparison by itself does not control a large horoconvex domain
	whose boundary height depends on direction.  The proof keeps that height as
	part of the limiting problem.  The large-domain gap is then reduced to a
	uniform spectral gap for an angular operator, rather than to a one-parameter
	radius comparison.
	
	In small diameter the Euclidean \(D^{-2}\) scale is the relevant model.  For
	large horoconvex domains, Nguyen--Stancu--Wei's ball and model-domain examples
	show that \(D^{-3}\) is the only possible diameter power for a uniform lower
	bound.  The present argument supplies that power from below.  It improves the
	previously known qualitative large-diameter lower bounds, while leaving the
	optimal coefficient and extremal geometry open.
	
	A first estimate suggested by the inradius/circumradius theorem is to compare
	\(\Omega\) with concentric balls of radii \(R-2\log 2\) and \(R\).  This
	correctly detects the ball asymptotic and gives a useful check on the
	coefficient arithmetic, but by itself it is too crude in dimensions \(n\ge3\).
	The reason is simple: shifting the radius of the first eigenvalue changes the
	\(R^{-2}\) correction at order \(R^{-3}\), and in higher dimension this
	loss is of the same order as the ball-gap coefficient and need not be
	smaller.  The radial-height argument keeps the boundary displacement as an
	angular variable instead of replacing the domain by its inner ball.  This is
	the step which recovers a positive uniform \(D^{-3}\) lower bound in all
	dimensions.
	
	More explicitly, the ball calculation gives, for every fixed \(a\),
	\[
	\lambda_2(B_{r+a}^{\mathbb H^n})-\lambda_1(B_r^{\mathbb H^n})
	=
	\left({4\pi^2\over n-1}-2a\pi^2\right)r^{-3}+o(r^{-3}).
	\]
	This check is sensitive to the precise shift.  With the sharper nonconcentric
	shift \(\log 2\), Proposition~\ref{prop:low-dim-ball-comparison} gives the
	large-diameter rate in dimensions \(n=2,3\).  In dimensions \(n\ge4\), the same
	coefficient is negative.  With the cruder \(2\log 2\) concentric annular
	comparison it is already negative for \(n\ge3\), so the \(n=3\) sign is
	reversed.  This sensitivity is exactly why the scalar radius-shift comparison
	cannot prove a uniform \(D^{-3}\) lower bound in every dimension.  The theorem
	instead uses the effective angular operator for the full radial height in all
	dimensions.
	Nor can this be repaired in dimensions \(n\ge4\) by sharpening the
	inradius/circumradius pinching: the sharp Borisenko--Miquel expression
	\[
	\log{(1+\sqrt{\tau})^2\over 1+\tau},
	\qquad \tau=\tanh(r/2),
	\]
	still tends to \(\log2\) as \(r\to\infty\).  Thus the large-diameter
	coefficient obstruction in Proposition~\ref{prop:low-dim-ball-comparison} is
	not an artifact of using a nonsharp radius estimate.
	
	The proof also avoids the interior rolling-ball or quantitative-boundary
	hypotheses which appear in several boundary-neighborhood methods; compare, for
	example, the integral-Ricci gap estimates of
	Ramos Oliv\'e--Rose--Wang--Wei
	\cite{RamosOliveRoseWangWei}.  Horoconvexity supplies the needed large-scale
	control in a different way: it gives a common-center annulus and, through
	horospherical support functions, a uniformly Lipschitz radial height.  The
	analytic estimates are then made on this radial-height family rather than on a
	tubular neighborhood of a smooth boundary with a uniform interior ball
	condition.
	
	The argument also explains why the existing exponentially small bounds do not
	see the expected large-diameter behavior.  After the common-center annulus
	reduction, a bounded radial-height perturbation of a large ball shifts the
	first radial branch at order \(R^{-3}\).  The high angular modes and the
	radial complement remain separated, so the low-energy spectral problem reduces
	to a compact angular problem.  This is the mechanism behind the polynomial
	scale.  Horoconvexity enters through the
	Borisenko--Miquel radius control and the
	horospherical-support representation; together they provide a uniform
	radial-height class without imposing an interior rolling-ball condition.
	
	This also points to the main limitation of the argument.  The compactness
	proof of the angular gap is qualitative.  It proves that some positive
	constant survives uniformly over all admissible radial heights, but it does
	not reveal which height profiles nearly minimize the gap.  An effective
	version of this proof would require a quantitative spectral gap for the family
	\(T_n+2\pi^2h+b_n\) with \(0\le h\le2\log 2\).  More direct semigroup or
	kernel-pinching approaches may also be possible, provided the corresponding
	cone-contraction step can be made quantitative.  None of those effective
	questions is needed for the large-diameter power proved here.
	
	The geodesic ball remains the calibration case.  Krist\'aly's large-radius
	first-eigenvalue expansion, combined with the \(n\mapsto n+2\) branch shift
	for the \(l=1\) mode, gives
	\[
	\lambda_2(B_R^{\HH^n})-\lambda_1(B_R^{\HH^n})
	\sim {4\pi^2\over (n-1)R^3}.
	\]
	For a general horoconvex domain the boundary height \(h\) enters the limiting
	angular operator.  The proof shows that this operator has a uniform spectral
	gap over the compact admissible height class.  It does not show that the
	round height minimizes the leading constant, nor does it rule out a different
	large-diameter extremal profile.
	
	The same radial-height block structure also suggests a stronger individual
	eigenvalue asymptotic.  For fixed \(k\), one expects the first \(k\) branches
	to be governed by the first \(k\) eigenvalues of the effective angular
	operator \(T_n+2\pi^2h+b_n\).  The present proof does not state that
	refinement.  It uses fixed finite reserves and records only the estimates
	needed for the gap lower bound.  A separate expansion theorem would require an
	\(\varepsilon\)-version of the low/high compression, uniform finite-angular
	truncation for the relevant effective spectral subspaces, and matching upper
	trial spaces, already for the two-dimensional space controlling
	\(\lambda_2\).  This is a natural companion refinement, but it is not part of
	the theorem proved here.
	
	Several refinements remain.  An effective or sharp value of the constant
	\(c_n\) would require a quantitative lower bound for the angular gap over all
	admissible heights, replacing the compactness argument used here.  Numerical
	exploration of the limiting height problem can guide this question, but no
	computation is used in the proof of the present lower bound.
	
	A sharper asymptotic problem asks for the infimum of the gap among large
	horoconvex domains.  The present proof shows the correct power but does not
	identify a minimizing angular height.
	
	The log-concavity problem is separate.  The large-diameter gap proof does not
	prove \(\operatorname{Hess}\log\psi_1\le0\).  Wei--Xiao's small-diameter theorem
	and the conformal-concavity estimates of Khan--Saha--Tuerkoen suggest useful
	tools, but the present argument uses a spectral reduction rather than a
	pointwise Hessian estimate.  A perturbative question remains: whether the first
	eigenfunction of a bounded radial-height perturbation of a large ball is
	governed, to leading order, by the positive ground state of the effective
	angular operator, and whether such structure can be converted into a global
	hyperbolic log-concavity principle.
	
	\paragraph*{Disclosure on AI assistance.}
	The authors used AI systems and AI-assisted search and review tools, including
	ChatGPT/Codex, Gemini, Claude, Grok, OpenRouter-routed models, Manus, and Exa,
	for brainstorming, source triage, objection generation, notation and exposition
	review, and repository/paper-editing assistance.  AI output was not treated as
	mathematical authority.  The authors reviewed, edited, and remain responsible
	for all theorem statements, proofs, constants, citations, and final prose.

\end{document}